\documentclass[reqno]{amsart}
\usepackage{color}
\usepackage{amsmath,amssymb,amsfonts,amsthm,bm}
\usepackage{mathrsfs}
\usepackage{graphicx}
\usepackage{enumerate}
\usepackage[numbers, sort&compress]{natbib}
\usepackage[shortlabels]{enumitem}
 \usepackage{newtxtext}
  \usepackage[varvw]{newtxmath}
 \usepackage{textcomp}
\usepackage[reversemp,paperwidth=155mm,paperheight=235mm,top={22mm},headheight={5.5pt},headsep={6.0mm},text={123mm,195mm},marginparsep=5mm,marginparwidth=12mm, bindingoffset=6mm, footskip=10mm]{geometry}
  
\newtheorem{theorem}{Theorem}[section]

\newtheorem{lemma}[theorem]{Lemma}
\newtheorem{proposition}[theorem]{Proposition}
\theoremstyle{assumption}

\theoremstyle{definition}
\newtheorem{definition}[theorem]{Definition}
\theoremstyle{remark}
\newtheorem{remark}[theorem]{Remark}
\numberwithin{equation}{section}

\newcommand{\eps}{\varepsilon}
\newcommand{\norm}[1]{\Vert#1\Vert}
\newcommand{\abs}[1]{\left\vert#1\right\vert}

\newcommand{\inner}[1]{\left(#1\right)}

\newcommand{\normm}[1]{{ \vert\kern-0.25ex \vert\kern-0.25ex \vert #1
		\vert\kern-0.25ex \vert\kern-0.25ex \vert}}

\makeatletter

\newbox \abstractbox
\renewenvironment{abstract}{\global\setbox\abstractbox=\vbox\bgroup
	\hsize=\textwidth
	\vskip 1.2cm
	\noindent\unskip \textbf{Abstract.}
}
{
	\egroup}

\@namedef{subjclassname@2020}{%
	\textup{2020} Mathematics Subject Classification}

\def\@startsection#1#2#3#4#5#6{%
	\if@noskipsec \leavevmode \fi
	\par \@tempskipa #4\relax
	\@afterindentfalse
	\ifdim \@tempskipa <\z@ \@tempskipa -\@tempskipa \@afterindentfalse\fi
	\if@nobreak \everypar{}\else
	\addpenalty\@secpenalty\addvspace\@tempskipa\fi
	\@ifstar{\@dblarg{\@sect{#1}{\@m}{#3}{#4}{#5}{#6}}}%
	{\@dblarg{\@sect{#1}{#2}{#3}{#4}{#5}{#6}}}%
}

\def\@settitle{%
	\bgroup
	\centering
	\vglue1cm
	\fontsize{12}{15}\fontseries{b}\selectfont
	\@title
	\vskip 20pt plus 6pt minus 8pt
	\egroup
}

\def\@setauthors{%
	\begingroup
	\trivlist
	\centering \bfseries
	\normalsize\@topsep30\p@\relax
	\advance\@topsep by -\baselineskip
	\item\relax
	\andify\authors
	{\rmfamily\authors}%
	\endtrivlist
	\endgroup
}

\def\@setaddresses{\par
	\nobreak \begingroup
	\normalsize
	\def\author##1{\nobreak\addvspace\bigskipamount}%
	\def\\{\unskip, \ignorespaces}%
	\interlinepenalty\@M
	\def\address##1##2{\begingroup
		\par\addvspace\bigskipamount\noindent
		\@ifnotempty{##1}{(\ignorespaces##1\unskip) }%
		{\ignorespaces##2}\par\endgroup}%
	\def\curraddr##1##2{\begingroup
		\@ifnotempty{##2}{\nobreak\indent{\itshape Current address}%
			\@ifnotempty{##1}{, \ignorespaces##1\unskip}\/:\space
			##2\par}\endgroup}%
	\def\email##1##2{\begingroup
		\@ifnotempty{##2}{\nobreak\noindent{\itshape E-mail address}%
			\@ifnotempty{##1}{, \ignorespaces##1\unskip}\/:
			##2\par}\endgroup}%
	\def\urladdr##1##2{\begingroup
		\@ifnotempty{##2}{\nobreak\indent{\itshape URL}%
			\@ifnotempty{##1}{, \ignorespaces##1\unskip}\/:\space
			\ttfamily##2\par}\endgroup}%
	\addresses
	\endgroup
}

\renewcommand\section{\@startsection{section}{1}{\z@}%
	{27pt plus 6pt minus 8pt}{14pt plus 6pt minus 8pt}
	{\center\normalfont\large\bfseries}}

\makeatother

\begin{document}
	
	\title[Gevrey well-posedness of the hydrostatic MHD-wave system]{Gevrey well-posedness of the hydrostatic MHD-wave system}
	
\author[W.-X.Li and Z.Xu]{Wei-Xi Li and  Zhan Xu}
	
\address[W.-X. Li]{School of Mathematics and Statistics,   \&  Hubei Key Laboratory of Computational Science, Wuhan University, Wuhan 430072,  China } \email{wei-xi.li@whu.edu.cn}

\address[Z. Xu]{School of Mathematics and Statistics,  Wuhan University, Wuhan 430072, China\newline
and \newline
Department of Applied Mathematics, The Hong Kong Polytechnic
University, Kowloon, Hong Kong, SAR, China.}
\email{xuzhan@whu.edu.cn}

	\keywords{hydrostatic MHD-wave system, Gevrey Well-posedness}

 \subjclass[2020]{76W05,76D10,35Q35}

\begin{abstract}
This paper investigates the well-posedness of the hydrostatic MHD-wave system. Unlike the standard hydrostatic MHD equations, the tangential magnetic field equation in this system is degenerate hyperbolic rather than parabolic, which leads to substantial mathematical difficulties. Using the boundary decomposition method, we establish local well-posedness in Gevrey $\frac{7}{6}$ space for convex initial data.
\end{abstract}
	
	\maketitle
	

	\section{Introduction and main result}
	
This paper investigates the following hydrostatic MHD-wave system in the domain $\Omega=\{(x,y)\in \mathbb T\times\mathbb R, 0<y<1\}$ (where the period of $\mathbb{T}$  is taken to be 1):
	\begin{equation}
		\left\{
		\begin{aligned}
			&  \partial_t u+u\partial_xu+v\partial_yu+  \partial_x p -\partial_y^2u
			=f\partial_x f+g\partial_yf,\\
            &\partial_yp=0,\\
			& \eta\partial_t^2 f +\partial_t f+u\partial_xf+v\partial_yf-\partial_y^2f
			=f\partial_x u+g\partial_yu, \\
            & \eta\partial_t^2 g +\partial_t g+u\partial_xg+v\partial_yg-\partial_y^2g
			=f\partial_x v+g\partial_yv, \\
			& \partial_xu+\partial_yv=0,\quad \partial_xf+\partial_yg=0,\\
   			& (u,v,\partial_yf,g)|_{y=0,1}=\bm 0,\quad (u,f,\partial_tf)|_{t=0}=(u_0,f_0,f_1),\\
		\end{aligned}
		\right.
		\label{hyMHD}
	\end{equation}
 where the parameter $\eta>0$ is a constant, and the unknowns $(u,v)$, $(f,g)$ and $p$ represent the velocity field, the magnetic field and the scalar pressure of the fluid, respectively.

We begin by recalling the origin of this system and the main physical and mathematical motivations for our study. Let $\Omega_{\varepsilon}={(x,Y)\in \mathbb T\times\mathbb R, 0<Y<\varepsilon}$, where $\varepsilon$ denotes the width of the strip. The MHD-wave system in this thin strip is given by
\begin{equation}\label{orsystem}
    \left\{
    \begin{aligned}
&\partial_tU+U\cdot\nabla U+\nabla P-\varepsilon^2\Delta U=H\cdot\nabla H,\\
&\eta\partial_t^2H+\partial_tH+U\cdot\nabla H-\varepsilon^2\Delta H=H\cdot\nabla U,\\
&\nabla\cdot U=\nabla\cdot H=0,\quad (U,\partial_Yb_1,b_2)|_{Y=0,\varepsilon}=\bm 0,\\
&(U,H,\partial_tH)|_{t=0}=(U_0,H_0,H_1).
    \end{aligned}
    \right.
\end{equation}
Here, $U$, $H$, and $P$ denote the velocity field, the magnetic field, and the scalar pressure, respectively, while $b_1$ and $b_2$ represent the tangential and normal components of the magnetic field $H$. For a detailed mathematical derivation of the governing equations \eqref{orsystem}, we refer to \cite{MR4508068,MR4151564}. Applying the scaling transformation
\begin{equation*}
\left\{
\begin{aligned}
    &U(t,x,Y)=(u^\varepsilon,\varepsilon v^\varepsilon)(t,x,\varepsilon^{-1}Y),\quad P(t,x,Y)=p^\varepsilon(t,x,\varepsilon^{-1}Y),\\
    &H(t,x,Y)=(f^\varepsilon,\varepsilon g^\varepsilon)(t,x,\varepsilon^{-1}Y),
\end{aligned}
\right.
\end{equation*}
we reformulate system \eqref{orsystem} into the following rescaled MHD-wave system in $\Omega$:
	\begin{equation}
		\left\{
		\begin{aligned}
			&  \partial_t u^\varepsilon+u^\varepsilon\partial_xu^\varepsilon+v^\varepsilon\partial_yu^\varepsilon -\varepsilon^2\partial_x^2u^\varepsilon-\partial_y^2u^\varepsilon+ \partial_x p^\varepsilon
			=f^\varepsilon\partial_x f^\varepsilon+g^\varepsilon\partial_yf^\varepsilon,\\
&\eps^2(\partial_tv^\varepsilon+u^\varepsilon\partial_xv^\varepsilon+v^\varepsilon\partial_yv^\varepsilon-\varepsilon^2\partial_x^2v^\varepsilon-\partial_y^2v^\varepsilon)+\partial_yp^\varepsilon=\eps^2(f^\varepsilon\partial_x g^\varepsilon+g^\varepsilon\partial_yg^\varepsilon),\\
			& \eta\partial_t^2 f^\varepsilon +\partial_t f^\varepsilon+u^\varepsilon\partial_xf^\varepsilon+v^\varepsilon\partial_yf^\varepsilon-\varepsilon^2\partial_x^2f^\varepsilon-\partial_y^2f^\varepsilon
			=f^\varepsilon\partial_x u^\varepsilon+g^\varepsilon\partial_yu^\varepsilon, \\
            & \eps^2(\eta\partial_t^2 g^\varepsilon +\partial_t g^\varepsilon+u^\varepsilon\partial_xg^\varepsilon+v^\varepsilon\partial_yg^\varepsilon-\varepsilon^2\partial_x^2g^\varepsilon-\partial_y^2g^\varepsilon)
			=\varepsilon^2\inner{f^\varepsilon\partial_x v^\varepsilon+g^\varepsilon\partial_yv^\varepsilon}, \\
			& \partial_xu^\varepsilon+\partial_yv^\varepsilon=0,\quad \partial_xf^\varepsilon+\partial_yg^\varepsilon=0,\\
   			& (u^\varepsilon,v^\varepsilon,\partial_yf^\varepsilon,g^\varepsilon)|_{y=0,1}=\bm 0,\\
          &(u^\varepsilon,v^\varepsilon,f^\varepsilon,g^\varepsilon,\partial_tf^\varepsilon,\partial_tg^\varepsilon)|_{t=0}=(u_0^\varepsilon,v_0^\varepsilon,f_0^\varepsilon,f_1^\varepsilon,g_0^\varepsilon,g_1^\varepsilon).\\
		\end{aligned}
		\right.
		\label{anisoMHD}
	\end{equation}
Formally passing to the limit $\varepsilon \to 0$ in \eqref{anisoMHD} yields the hydrostatic MHD-wave system \eqref{hyMHD}.

The aim of this paper is to study the well-posedness of system \eqref{hyMHD} with initial data belonging to a Gevrey class. The mathematical analysis of the Prandtl boundary layer and related models has a long history, and their well-posedness or ill-posedness has been extensively explored in various function spaces; see, e.g., \cite{MR3327535, MR3795028, MR3925144, MR1476316, MR2601044, MR2849481, MR3461362, MR2975371, MR3590519, MR3493958, MR2020656, MR4465902, MR3464051,MR4293727,lixuyang2025,MR4834601,MR4915128,MR4726840,MR4661716,MR4700999,MR4498949,MR4494626} and the references therein.
Similar to the classical Prandtl equation, the nonlinear terms $v\partial_yu$, $g\partial_yf$, $v\partial_yf$, and $g\partial_yu$ in \eqref{hyMHD} induce a loss of one tangential derivative in the energy estimates. Thus, it is natural to work with analytic data when no structural assumptions are imposed on the initial data (see \cite{MR2601044,MR2563627}). Indeed, Renardy \cite{MR2563627} demonstrated the ill-posedness of the hydrostatic Navier-Stokes system for general Sobolev data. On the other hand, Paicu-Zhang-Zhang \cite{MR4125518} established the global well-posedness of the anisotropic Navier-Stokes system and the hydrostatic Navier-Stokes system with small analytic initial data. For convex initial data, G\'erard-Varet, Masmoudi  and Vicol \cite{MR4149066} proved the local well-posedness of the hydrostatic Navier-Stokes system in the Gevrey class of index $9/8$ (see \cite{MR4621052,MR4803680} for recent improvements on the Gevrey index). We also refer to related results for the classical Prandtl equation: under Oleinik's monotonicity assumption, the Prandtl equation is well-posed in Sobolev spaces (see \cite{MR3327535,MR3765768,MR3385340,MR1697762,MR2020656}); another line of research establishes well-posedness in Gevrey spaces (see \cite{MR3795028,MR3925144,MR3429469,MR4465902,MR4055987,MR2049030,MR1617542,Xu2026} for instance).

When a magnetic field is present, Aarach \cite{MR4362378} justified the inviscid limit from the anisotropic MHD system to the hydrostatic MHD system globally in time for small analytic data. Recently, this inviscid limit was justified in Sobolev spaces in \cite{MR4384041} when the initial tangential magnetic field is non-zero. For the classical MHD boundary layer equations, under a uniform tangential magnetic field, Liu, Xie, and Yang \cite{MR3882222} proved well-posedness in Sobolev spaces without Oleinik's monotonicity assumption and justified the Prandtl ansatz in \cite{MR3975147}. More recently, Li and Yang \cite{MR4270479} established a well-posedness result in the Gevrey class of index $3/2$ without any structural assumptions.

The hydrostatic MHD-wave system \eqref{hyMHD} is a mixed degenerate system coupling parabolic and hyperbolic equations. In contrast to the purely parabolic hydrostatic MHD equations, there is no cancellation mechanism between the equations for the tangential velocity field $u$ and the tangential magnetic field $f$. This absence of cancellation makes it difficult to establish the well-posedness of \eqref{hyMHD} in Sobolev settings.

We first introduce the Gevrey spaces that will be used throughout the paper.
\begin{definition}\label{def:one}
 For given $\rho>0$, $r\in\mathbb{R}$ and $\sigma\ge 1$, the space $X_{\rho,\sigma,r}$ of Gevrey  functions consists of all smooth functions $h(x,y)$ defined in domain $\Omega$ such that the norm $\norm{h}_{X_{\rho,\sigma,r}}<+\infty$, where
 	\begin{equation}\label{def1}
  \begin{aligned}
\norm{h}^2_{X_{\rho,\sigma,r}}\stackrel{\rm def}{=}\sum_{m\ge0}L^2_{\rho,m,\sigma,r}\|\partial^m_{x} h\|^2_{L^2},
  \end{aligned}
	\end{equation}  
  with 
   	\begin{equation}\label{def2}
  \begin{aligned}
L_{\rho,m,\sigma,r}\stackrel{\rm def}{=}\frac{\rho^{m+1}(m+1)^r}{(m!)^{\sigma}}, \ m\in \mathbb{Z}_+.
  \end{aligned}
	\end{equation}  
\end{definition}

The main result of this paper concerns the well-posedness of system \eqref{hyMHD} with initial data in a more general Gevrey class. Without loss of generality, we set $\eta = 1$ in \eqref{hyMHD}. The  main result  is stated as follows.

\begin{theorem}\label{thm3}
    Let $1\leq\sigma\leq \frac{7}{6}$, $r\ge10$ and $\rho_0>0$. Assume that the initial data $(u_0,f_0,f_1)$ of system \eqref{hyMHD} satisfies the regularity condition
\begin{align}\label{conregular}
\sum_{k=0}^5\|\partial_y^ku_0\|_{X_{2\rho_0,\sigma,r}}^2+\sum_{k=0}^4\|\partial_y^kf_0\|_{X_{2\rho_0,\sigma,r}}^2+\sum_{k=0}^3\|\partial_y^kf_1\|_{X_{2\rho_0,\sigma,r}}^2<+\infty,
\end{align}
the convexity condition
\begin{align*}
    \inf_{\Omega}\partial_y^2u_0>0,
\end{align*}
and the compatibility conditions. Then there exist $T>0$ and $0<\tilde{\rho}<\rho_0$ such that system \eqref{hyMHD} admits a unique local-in-time solution $(u,f)$  satisfying for any $t\in [0,T]$,
\begin{equation}\label{ret3}
    \begin{aligned}
\sum_{k=0}^2\|\partial_t^k(u,\partial_yu)\|_{X_{\tilde{\rho},\sigma,r}}^2+\sum_{k=0}^2\|\partial_t^k(f,\partial_tf,\partial_yf,\partial_t\partial_yf,\partial_y^2f)\|_{X_{\tilde{\rho},\sigma,r}}^2 < +\infty,
    \end{aligned}
\end{equation}
and 
\begin{align}\label{retconvex}
    \inf_{\Omega}\partial_y^2u>0.
\end{align}
\end{theorem}

\begin{remark}
The inclusion of time derivatives of $(u,f)$ in the estimates is primarily to ensure the convexity of $u$. Further details are provided in Section \ref{sec:maximum}.
\end{remark}

\begin{remark}
If the right-hand side of the tangential velocity equation in \eqref{hyMHD} were treated as a given source term, the analysis of the hydrostatic Navier-Stokes equations in \cite{MR4803680} suggests that the optimal Gevrey index for convex initial data would be $\frac{3}{2}$. However, our estimates for these source terms are not sufficient to apply this argument directly; see estimate \eqref{est:P5} and, more specifically, inequality \eqref{ineqkey} for details.
\end{remark}

This paper is organized as follows. We state the \emph{a priori} estimate corresponding to Theorem \ref{thm3} in Section \ref{sec:pri}, present the maximum and minimum principles for $\partial_y^2u$ in Section \ref{sec:maximum}, and complete the proof of the \emph{a priori} estimate in Sections \ref{sec:x0}-\ref{sec:proof}.

\section{The \emph{a priori} estimate of Theorem \ref{thm3}}\label{sec:pri}
The key part in the proof of Theorem \ref{thm3} is to derive the \emph{a priori} estimate for system \eqref{hyMHD} so that the existence and uniqueness follow from a standard argument. Hence, for brevity, we only present the proof of the \emph{a priori} estimate and omit the regularization procedure.

\subsection{Notations and Gevrey norms}\label{subsec:norm}
To present our arguments more clearly, we will simplify the previously introduced notations and norms, and define new ones in the following discussion. Recall $\norm{\cdot}_{X_{\rho,\sigma,r}}$ and $L_{\rho,m,\sigma,r}$ are defined in Definition \ref{def:one}. For convenience, we abbreviate them as $\norm{\cdot}_{\rho,r}$ and $L_{\rho,m,r}$ for given $\sigma\ge 1$. For any function $h(x,y)$ and $m\in\mathbb{Z}_+$, we define $h_m(x,y)$ by 
\begin{equation*}
    h_m(x,y)\stackrel{\rm def}{=}L_{\rho,m,r}\partial_x^mh(x,y).    
\end{equation*}
With this notation, we have
\begin{equation*}
\norm{h}_{\rho,r}^2=\sum_{m\ge0}\norm{h_m}_{L^2(\Omega)}^2    
\end{equation*}
in view of \eqref{def1}.
On the other hand, we introduce the space $\tilde{X}_{\rho,\sigma,r}$ which consists of all sequences of functions $\vec h=\{h_{(m)}\}_{m\ge 0}$ defined in $\Omega$ such that $\norm{\vec{h}}_{\tilde{X}_{\rho,\sigma,r}}<+\infty$, where
\begin{equation*}
    \norm{\vec{h}}_{\tilde{X}_{\rho,\sigma,r}}^2\stackrel{\rm def}{=}\sum_{m\ge 0}L_{\rho,m,r}^2\norm{h_{(m)}}^2_{L^2(\Omega)},
\end{equation*}
with $L_{\rho,m,r}$ given as above. Obviously, if $h\in X_{\rho,\sigma,r}$, then $\{\partial_x^mh\}_{m\ge 0}\in\tilde{X}_{\rho,\sigma,r}$. For convenience, we still abbreviate $\norm{\cdot}_{\tilde{X}_{\rho,\sigma,r}}$ as $\norm{\cdot}_{\rho,r}$.  Correspondingly, we define the one-dimensional counterpart $\abs{\cdot}_{\rho,r}$ by
\begin{equation*}
    |h|_{\rho,r}^2\stackrel{\rm def}{=}\sum_{m\ge0}\norm{h_m}_{L_{x}^2(\mathbb{T})}^2\quad\mathrm{and}\quad |\vec{h}|_{\rho,r}^2\stackrel{\rm def}{=}\sum_{m\ge 0}L_{\rho,m,r}^2\norm{h_{(m)}}^2_{L_{x}^2(\mathbb{T})}.  
\end{equation*}
Moreover, for $I_0=(0,+\infty)$ and $I_1=(-\infty,1)$, we define the norm $\norm{\cdot}_{\rho,r,I_i}$ $(i=0,1)$ by
\begin{equation*}
    \norm{h}_{\rho,r,I_i}^2\stackrel{\rm def}{=}\sum_{m\ge0}\norm{h_m}_{L^2(\mathbb{T}\times I_i)}^2\quad\mathrm{and}\quad \norm{\vec{h}}_{\rho,r,I_i}^2\stackrel{\rm def}{=}\sum_{m\ge 0}L_{\rho,m,r}^2\norm{h_{(m)}}^2_{L^2(\mathbb{T}\times I_i)}.  
\end{equation*}
 It is obvious to see $\norm{\cdot}_{\rho,r}\leq \norm{\cdot}_{\rho,r,I_i}$ for $i=0,1$.
And we introduce the stream function $\phi(t,x,y)$ which satisfies
\begin{equation}\label{def:phi}
    v=-\partial_x\phi,\quad u=\partial_y\phi+\mathcal{C}(t),\quad \mathcal{C}(t)=\int_{\Omega}u(t,x,y)dxdy.
\end{equation}

 Let $(u,f)$ be the solution to system \eqref{hyMHD} with initial data satisfying all assumptions in Theorem \ref{thm3}. For given $\sigma\ge 1$, $r\in\mathbb{R}$ and $k\in\mathbb{N}$, we define energy functionals $\mathcal{X}_{\rho,k}(t)$, $\mathcal{Y}_{\rho,k}(t)$ and $\mathcal{Z}_{\rho,k}(t)$ by
 	\begin{equation}\label{def:energy}
		\left\{
		\begin{aligned}
\mathcal{X}_{\rho,k}(t)\stackrel{\rm def}{=}&\norm{\partial_t^ku}_{\rho,r-k\sigma+\frac{1}{2}}^2+\norm{\partial_t^k\partial_yu}_{\rho,r-k\sigma}^2+\norm{\partial_t^kf}_{\rho,r-(k+1)\sigma+\frac{5}{2}}^2\\
&+\norm{\partial_t^k(\partial_tf,\partial_yf)}_{\rho,r-(k+1)\sigma+\frac{3}{2}}^2+\norm{\partial_t^k\partial_yf}_{\rho,r-(k+1)\sigma+2}^2\\
&+\norm{\partial_t^k(\partial_t\partial_yf,\partial_y^2f)}_{\rho,r-(k+1)\sigma+1}^2,\\
\mathcal{Y}_{\rho,k}(t)\stackrel{\rm def}{=}&\norm{\partial_t^ku}_{\rho,r-k\sigma+1}^2+\norm{\partial_t^k\partial_yu}_{\rho,r-k\sigma+\frac{1}{2}}^2+\norm{\partial_t^kf}_{\rho,r-(k+1)\sigma+3}^2\\
&+\norm{\partial_t^k(\partial_tf,\partial_yf)}_{\rho,r-(k+1)\sigma+2}^2+\norm{\partial_t^k\partial_yf}_{\rho,r-(k+1)\sigma+\frac{5}{2}}^2\\
&+\norm{\partial_t^k(\partial_t\partial_yf,\partial_y^2f)}_{\rho,r-(k+1)\sigma+\frac{3}{2}}^2,\\
\mathcal{Z}_{\rho,k}(t)\stackrel{\rm def}{=}&\norm{\partial_t^k\partial_yu}_{\rho,r-k\sigma+\frac{1}{2}}^2+\norm{\partial_t^k\partial_y^2u}_{\rho,r-k\sigma}^2+\norm{\partial_t^{k+1}f}_{\rho,r-(k+1)\sigma+\frac{3}{2}}^2\\
&+\norm{\partial_t^{k+1}\partial_yf}_{\rho,r-(k+1)\sigma+1}^2.\\
		\end{aligned}
		\right.
	\end{equation}
Accordingly, we define
\begin{equation}\label{def:energytwo}
    \mathcal{X}_\rho(t)\stackrel{\rm def}{=}\sum_{k=0}^2\mathcal{X}_{\rho,k}(t),\quad  \mathcal{Y}_\rho(t)\stackrel{\rm def}{=}\sum_{k=0}^2\mathcal{Y}_{\rho,k}(t),\quad \mathcal{Z}_\rho(t)\stackrel{\rm def}{=}\sum_{k=0}^2\mathcal{Z}_{\rho,k}(t).
\end{equation}
In view of the definitions of $\mathcal{X}_\rho(t)$ and $\mathcal{Y}_\rho(t)$, it holds that
\begin{equation}\label{xy}
    \mathcal{X}_\rho(t)\leq\mathcal{Y}_\rho(t).
\end{equation}

\subsection{Statement of the \emph{a priori} estimate} \label{subsec:pri}

With notations and norms above, we now state the theorem on the \emph{a priori} estimate.
\begin{theorem}[\emph{A priori} estimate]\label{thm:pri}
    Let $1\leq\sigma\leq \frac{7}{6}$, $r\ge10$ and $\rho_0>0$.  Under the same assumption as in Theorem \ref{thm3}, there exist two constants $\beta, C_*\ge 1$ depending only on $\sigma$, $r$, $\rho_0$, the Sobolev embedding constants and the  initial data such that if $(u,f)$ is the solution to system \eqref{hyMHD} satisfying that
    \begin{equation}\label{ass:pri}
    \sup_{t\in[0,T]}\mathcal{X}_{\rho}(t)+\beta\int^T_0\mathcal{Y}_\rho(t)dt+\int^T_0\mathcal{Z}_\rho(t)dt\leq 2C_*M,  
    \end{equation}
    and
    \begin{equation}\label{ass:convex}
 2\delta\leq \inf_{\Omega}\partial_y^2u_0\leq \frac{1}{2\delta},
    \end{equation}
    for some $0<\delta<\frac{1}{2}$ and $T=\beta^{-1}$, then it holds that
        \begin{equation}\label{ret:pri}
    \sup_{t\in[0,T]}\mathcal{X}_{\rho}(t)+\beta\int^T_0\mathcal{Y}_\rho(t)dt+\int^T_0\mathcal{Z}_\rho(t)dt\leq C_*M,  
    \end{equation}
    and
    \begin{equation}\label{ret:convex}
        \forall\ t\in[0,T],\quad \delta\leq \inf_{\Omega}\partial_y^2u(t)\leq \frac{1}{\delta}.
    \end{equation}
    Here $\mathcal{X}_\rho$, $\mathcal{Y}_\rho$ and $\mathcal{Z}_\rho$ are defined in \eqref{def:energytwo}, $\rho$ is defined by
    \begin{align}\label{defrho2}
    \rho(t)\stackrel{\rm def}{=}\rho_0 e^{-\beta t},
\end{align}
and $M$ is defined by
    \begin{equation}\label{def:M}
        M\stackrel{\rm def}{=}\max\Big\{\mathcal{X}_{\rho}\big|_{t=0},\sum_{k=0}^5\|\partial_y^ku_0\|_{X_{2\rho_0,\sigma,r}}^2+\sum_{k=0}^4\|\partial_y^kf_0\|_{X_{2\rho_0,\sigma,r}}^2+\sum_{k=0}^3\|\partial_y^kf_1\|_{X_{2\rho_0,\sigma,r}}^2,1\Big\}.
    \end{equation}
\end{theorem}

\begin{remark}
    Note that for given $n\in\mathbb{Z}_+$, $0<\rho_1<\rho_2$, $r_1,r_2\in\mathbb{R}$ and suitable function $h$,
    \begin{equation}\label{note:norm}
        \norm{\partial_x^n h}_{\rho_1,r_1}\leq C_{\rho_1,\rho_2,n,r_1,r_2}\norm{h}_{\rho_2,r_2},
    \end{equation}
    where $C_{\rho_1,\rho_2,n,r_1,r_2}$ depends only on $\rho_1,\rho_2,n,r_1$  and $r_2$. Then a direct computation gives
    \begin{equation*}
    \left\{
    \begin{aligned}
       & \mathrm{condition}\ \eqref{conregular}\ \Longrightarrow \mathcal{X}_{\rho}(t)\big|_{t=0}<+\infty,\\
       &\mathrm{conclusions}\ \eqref{ret:pri},\ \eqref{ret:convex}\Longrightarrow \mathrm{conclusions}\ \eqref{ret3},\ \eqref{retconvex},
        \end{aligned}
    \right.
    \end{equation*}
    by choosing $\tilde{\rho}<\rho$ in view of \eqref{defrho2}. Consequently, Theorem \ref{thm:pri} is the \emph{a priori} estimate of Theorem \ref{thm3}. 
\end{remark}

The proof of Theorem \ref{thm:pri} will be given in Sections \ref{sec:maximum}-\ref{sec:proof}. We first list some facts which will be used later. In view of \eqref{defrho2} and \eqref{def2}, we have
\begin{equation}\label{factone}
    \forall\ m\ge 0\ \textrm{and}\ r\in\mathbb{R},\quad \frac{d}{dt}L_{\rho,m,r}=-\beta(m+1)L_{\rho,m,r},
\end{equation}
and
\begin{equation}\label{facttwo}
    \forall\ 0\leq t\leq T=\beta^{-1},\quad e^{-1}\rho_0\leq \rho(t)\leq \rho_0.
\end{equation}
 We will use the following Young's inequality for discrete convolution:
\begin{equation}\label{young}
\Big{[}\sum^{+\infty}_{m=0}\Big{(}\sum^m_{j=0}p_jq_{m-j}\Big{)}^2\Big{]}^{\frac{1}{2}}\leq \Big{(}\sum^{+\infty}_{m=0}q_m^2\Big{)}^{\frac{1}{2}}\sum^{+\infty}_{j=0}p_j,
\end{equation}
where $\{p_j\}_{j\ge 0}$ and $\{q_j\}_{j\ge 0}$ are positive sequences.

In the following discussion, to simplify the notations, we will use $C$ to denote a generic constants which may vary from line to line and depend only on $\sigma$, $r$, $\rho_0$ and the Sobolev embedding constants, but are independent of $C_*$, $\beta$, M and $\delta$.

\section{Maximum and minimum principle for $\partial_y^2u$}\label{sec:maximum}
This section is dedicated to proving the maximum and minimum principle for $\partial_y^2u$, that is, justifying the validity of \eqref{ret:convex} under the assumption as in Theorem \ref{thm:pri}. This can be stated as follows.
\begin{proposition}\label{prop:principle}
Under the same assumption as given in Theorem \ref{thm:pri}, conclusion \eqref{ret:convex} holds, provided $\beta$ is sufficiently large.
\end{proposition}

\begin{proof}
The quantity $\partial_y^2u$ obeys a (degenerate) parabolic equation with Dirichlet boundary conditions:
\begin{equation*}
    \left\{
    \begin{aligned}
&(\partial_t+u\partial_x+v\partial_y-\partial_y^2)\partial_y^2u+(\partial_xu)\partial_y^2u=(\partial_yu)\partial_x\partial_yu+\partial_y(f\partial_x\partial_yf+g\partial_y^2f),\\
&\partial_y^2u|_{y=0,1}=-2\int^1_0u\partial_xudy+2\int^1_0f\partial_xfdy+\int^1_0\partial_y^2udy-\int^1_0\int_{\mathbb{T}}\partial_y^2udxdy.
    \end{aligned}
    \right.
\end{equation*}
As shown in \cite[Proposition 6.1]{MR4149066}, the key to deduce the convexity of $u$ is to derive the $L_t^2L_{x,y}^{\infty}$ estimates on $(\partial_yu)\partial_x\partial_yu$ and $\partial_y(f\partial_x\partial_yf+g\partial_y^2f)$. To avoid redundancy, we present these estimates in Lemma \ref{lem:principle} below and omit the proof details for Proposition \ref{prop:principle}.  
\end{proof}

\begin{lemma}\label{lem:principle}
 Under the same assumption as given in Theorem \ref{thm:pri}, it holds that
 \begin{equation}\label{est:dy2u}
       \sup_{t\in [0,T]}\norm{\partial_y^2u}_{\rho,r-\sigma}\leq C\sqrt{C_*M},   
 \end{equation}
 and
 \begin{equation}\label{est:principle}
\int^T_0\big(\norm{(\partial_yu)\partial_x\partial_yu}_{L^\infty}^2+\norm{\partial_y(f\partial_x\partial_yf+g\partial_y^2f)}_{L^\infty}^2\big)dt\leq CC_*^4M^4,
 \end{equation}
 provided $\beta$ is sufficiently large.
\end{lemma}

\begin{proof}
  We first write that
\begin{align*}
\partial_y^2u(t)=\partial_y^2u(0)+\int^t_0\partial_t\partial_y^2u(s)ds.
\end{align*}
Then for any $0\leq t\leq T=\beta^{-1}$, we have, recalling $M$ is given in \eqref{def:M},
\begin{equation*}
\begin{aligned}
    \norm{\partial_y^2u(t)}_{\rho,r-\sigma}&\leq \norm{\partial_y^2u(0)}_{\rho_0,r-\sigma}+\int^t_0\norm{\partial_t\partial_y^2u(s)}_{\rho,r-\sigma}ds\\
    \leq \sqrt{M}+&\sqrt{T}\left(\int^T_0\norm{\partial_t\partial_y^2u(t)}_{\rho,r-\sigma}^2dt\right)^\frac{1}{2}\leq \sqrt{M}+\beta^{-\frac{1}{2}}\sqrt{C_*M}\leq C\sqrt{C_*M}.
\end{aligned}
\end{equation*}
Thus, assertion \eqref{est:dy2u} holds.

It remains to prove assertion \eqref{est:principle} and we first note that for any $0\leq t\leq T=\beta^{-1}$,
\begin{equation*}
    \norm{\partial_yu}_{H_x^2H_y^1}\leq C\mathcal{X}_{\rho}^\frac{1}{2}+C\norm{\partial_y^2u}_{\rho,r-\sigma}\leq C\sqrt{C_*M},
\end{equation*}
the last inequality using assumption \eqref{ass:pri} as well as \eqref{est:dy2u}. This with Sobolev embedding inequality and $T=\beta^{-1}\leq 1$ gives
\begin{equation}\label{principle1}
\int^T_0\norm{(\partial_yu)\partial_x\partial_yu}_{L^\infty}^2dt\leq C\int^T_0\norm{\partial_yu}_{H_x^2H_y^1}^4dt\leq CTC_*^2M^2\leq CC_*^2M^2. 
\end{equation}
On the other hand,  by assumption \eqref{ass:pri} and the definition of $\mathcal{X}_\rho$ in \eqref{def:energytwo}, we note that
\begin{align*}
 \norm{\partial_y(f\partial_x\partial_yf+g\partial_y^2f)}_{L^\infty}&\leq C\norm{f}_{H_x^2H_y^2}\big(\norm{f}_{H_x^2H_y^2}+\norm{\partial_y^3f}_{H_x^2L_y^2}+\norm{\partial_y^4f}_{H_x^2L_y^2}\big)\\
 &\leq CC_*M+C\sqrt{C_*M}\big(\norm{\partial_y^3f}_{H_x^2L_y^2}+\norm{\partial_y^4f}_{H_x^2L_y^2}\big).
\end{align*}
To deal with  $\norm{\partial_y^3f}_{H_x^2L_y^2}$, observing $\partial_y^3f=\partial_t^2\partial_yf+\partial_t\partial_yf-\partial_y(f\partial_xu+g\partial_yu)+\partial_y(u\partial_xf+v\partial_yf),$
we use assumption \eqref{ass:pri} and \eqref{est:dy2u} to obtain
\begin{align*}
\norm{\partial_y^3f}_{H_x^2L_y^2}\leq&\norm{\partial_t\partial_yf}_{H_x^2L_y^2}+\norm{\partial_t^2\partial_yf}_{H_x^2L_y^2}\\
&+\norm{\partial_y(f\partial_xu+g\partial_yu)}_{H_x^2L_y^2}+\norm{\partial_y(u\partial_xf+v\partial_yf)}_{H_x^2L_y^2}\\
\leq&C\sqrt{C_*M}+CC_*M\leq CC_*M,
\end{align*}
the last inequality holding because of $C_*M\ge 1$. For the term $\norm{\partial_y^4f}_{H_x^2L_y^2}$, we repeat a similar argument to deduce that
\begin{equation}\label{est:dy3u}
    \norm{\partial_y^3u}_{H_x^2L_y^2}\leq CC_*M,
\end{equation}
and furthermore, observing $\partial_y^4f=\partial_t^2\partial_y^2f+\partial_t\partial_y^2f+\partial_y^2(u\partial_xf+v\partial_yf)-\partial_y^2(f\partial_xu+g\partial_yu),$
\begin{equation*}
    \norm{\partial_y^4f}_{H_x^2L_y^2}\leq CC_*^\frac{3}{2}M^\frac{3}{2}.
\end{equation*}
Combining with these estimates above, we have
\begin{equation*}
     \norm{\partial_y(f\partial_x\partial_yf+g\partial_y^2f)}_{L^\infty}\leq CC_*M+CC_*^2M^2\leq CC_*^2M^2.
\end{equation*}
Consequently,
\begin{equation} \label{principle2}
\int^T_0\norm{\partial_y(f\partial_x\partial_yf+g\partial_y^2f)}_{L^\infty}^2dt\leq CTC_*^4M^4\leq CC_*^4M^4.
\end{equation}
Then assertion \eqref{est:principle} holds by combining \eqref{principle1} and \eqref{principle2}. This completes the proof of Lemma \ref{lem:principle}.
\end{proof}

\section{Estimate on $\mathcal{X}_{\rho,0}$}\label{sec:x0}
This part is devoted to deriving the estimate for $\mathcal{X}_{\rho,0}$, which can be stated as follows.
\begin{proposition}\label{prop:x0}
 Under the same assumption as given in Theorem \ref{thm:pri}, it holds that
 \begin{equation}\label{est:x0}
 \begin{aligned}
   &\sup_{t\in[0,T]}\mathcal{X}_{\rho,0}+\beta\int^T_0\mathcal{Y}_{\rho,0}dt+\int^T_0\mathcal{Z}_{\rho,0}dt\\
 &\leq C\delta^{-2}\int^T_0\norm{\phi}_{\rho,r+\sigma}^2dt+C\delta^{-3}C_*^\frac{3}{2}M^\frac{3}{2}\int^T_0\mathcal{Y}_\rho dt+C\delta^{-3}C_*^\frac{3}{2}M^\frac{3}{2}\int^T_0\mathcal{Y}_\rho^\frac{1}{2}\mathcal{Z}_\rho^\frac{1}{2} dt\\
 &\quad+C\delta^{-3}C_*^\frac{3}{2}M^\frac{3}{2}\int^T_0\mathcal{Y}_\rho^\frac{1}{4}\mathcal{Z}_\rho^\frac{3}{4}dt+C\delta^{-2}M,
  \end{aligned}
 \end{equation}
 provided $\beta$ is sufficiently large. Recall that $\mathcal{X}_{\rho,0},\mathcal{Y}_{\rho,0}$ and $\mathcal{Z}_{\rho,0}$ are defined in \eqref{def:energy}, $\mathcal{Y}_\rho$ and $\mathcal{Z}_\rho$ are defined in \eqref{def:energytwo} and $\phi$ is given in \eqref{def:phi}.
\end{proposition}

The proof of Proposition \ref{prop:x0} is highly non-trivial, and we will establish it via the following three lemmas.

\begin{lemma}\label{lem:u}
   Under the same assumption as given in Theorem \ref{thm:pri}, it holds that
 \begin{multline*}
   \sup_{t\in[0,T]}\norm{u}_{\rho,r+\frac{1}{2}}^2+\beta\int^T_0\norm{u}_{\rho,r+1}^2dt+\int^T_0\norm{\partial_yu}_{\rho,r+\frac{1}{2}}^2dt\\
   \leq CC_*M\int^T_0\mathcal{Y}_{\rho}dt+C\sqrt{C_*M}\int^T_0\mathcal{Y}_\rho^\frac{1}{2}\mathcal{Z}_\rho^\frac{1}{2}dt+C\int^T_0\norm{\phi}^2_{\rho,r+\sigma}dt+CM,
 \end{multline*}
 provided $\beta$ is sufficiently large. Recall $\mathcal{Y}_\rho$ and $\mathcal{Z}_\rho$ are defined in \eqref{def:energytwo} and $\phi$ is given in \eqref{def:phi}.  
\end{lemma}

\begin{proof}
For $m\in \mathbb{Z}_+$, applying $\partial_x^m$ to the velocity equation in \eqref{hyMHD} and using $v=-\partial_x\phi$, we get
\begin{multline}\label{eq:u}
\partial_t\partial_x^mu+u\partial_x^{m+1}u+v\partial_x^m\partial_yu-\partial_x^m\partial_y^2u+\partial_x^{m+1}p
=\partial_x^m\inner{f\partial_xf+g\partial_yf}\\-\sum^m_{k=1}\binom{m}{k}(\partial_x^ku)\partial_x^{m-k+1}u-\sum^{m-1}_{k=1}\binom{m}{k}(\partial_x^kv)\partial_x^{m-k}\partial_yu+(\partial_x^{m+1}\phi)\partial_yu.    
\end{multline}
Observing that
\begin{equation*}
     \inner{\partial_x^{m+1}p,\ \partial_x^mu}_{L^2}=\inner{\partial_x^mp,\ \partial_x^m\partial_yv}_{L^2}=-\inner{\partial_x^m\partial_yp,\ \partial_x^mv}_{L^2}=0
 \end{equation*}
 implied by $\partial_xu+\partial_yv=0$ and $\partial_yp=0$, we take the $L^2$-product with $\partial_x^mu$ on both sides of \eqref{eq:u}, multiply by $L_{\rho,m,r+\frac{1}{2}}^2$, use the fact \eqref{factone} and then take summation over $m$ to get
\begin{equation}\label{est:uu}
    \begin{aligned}
 \frac{1}{2}\frac{d}{dt}\norm{u}_{\rho,r+\frac{1}{2}}^2+\beta\norm{u}_{\rho,r+1}^2+\norm{\partial_yu}_{\rho,r+\frac{1}{2}}^2\leq I_1+I_2+I_3,
    \end{aligned}
\end{equation}
where
	\begin{equation*}
		\left\{
		\begin{aligned}
I_1=&\sum^{+\infty}_{m=0}\sum^m_{k=0}L_{\rho,m,r+\frac{1}{2}}^2\binom{m}{k}\left|\inner{(\partial_x^kf)\partial_x^{m-k+1}f+(\partial_x^kg)\partial_x^{m-k}\partial_yf,\ \partial_x^mu}_{L^2}\right|,\\
I_2=&\sum^{+\infty}_{m=0}\sum^m_{k=1}L_{\rho,m,r+\frac{1}{2}}^2\binom{m}{k}\left|\inner{(\partial_x^ku)\partial_x^{m-k+1}u,\ \partial_x^mu}_{L^2}\right|\\
&+\sum^{+\infty}_{m=0}\sum^{m-1}_{k=1}L_{\rho,m,r+\frac{1}{2}}^2\binom{m}{k}\left|\inner{(\partial_x^kv)\partial_x^{m-k}\partial_yu,\ \partial_x^mu}_{L^2}\right|,\\
I_3=&\sum^{+\infty}_{m=0}L_{\rho,m,r+\frac{1}{2}}^2\left|\inner{(\partial_x^{m+1}\phi)\partial_yu,\ \partial_x^mu}_{L^2}\right|.
		\end{aligned}
		\right.
	\end{equation*}
We now deal with $I_1-I_3$ one by one.

\underline{\it The $I_1$ bound}. We first write
\begin{equation}\label{I11}
\begin{aligned}
&\sum^{+\infty}_{m=0}\sum^m_{k=0}L_{\rho,m,r+\frac{1}{2}}^2\binom{m}{k}\left|\inner{(\partial_x^kf)\partial_x^{m-k+1}f,\ \partial_x^mu}_{L^2}\right|\\
    &\leq \sum^{+\infty}_{m=0}\sum^{[\frac{m}{2}]}_{k=0}L_{\rho,m,r+\frac{1}{2}}^2\binom{m}{k}\norm{\partial_x^kf}_{L^{\infty}}\norm{\partial_x^{m-k+1}f}_{L^2}\norm{\partial_x^mu}_{L^2}\\&\quad+\sum^{+\infty}_{m=0}\sum^{m}_{k=[\frac{m}{2}]+1}L_{\rho,m,r+\frac{1}{2}}^2\binom{m}{k}\norm{\partial_x^kf}_{L^2}\norm{\partial_x^{m-k+1}f}_{L^{\infty}}\norm{\partial_x^mu}_{L^2}.
\end{aligned}
\end{equation}
Here and below, $[p]$  represents the largest integer less than or equal to $p$. On the other hand, a direct computation with \eqref{facttwo} gives for $r\ge 10$ and $1\leq\sigma\leq\frac{7}{6}$,
\begin{equation}\label{ineq1}
    \left\{
    \begin{aligned}
&\binom{m}{k}\frac{L_{\rho,m,r+\frac{1}{2}}^2}{L_{\rho,k+1,r-\sigma+\frac{3}{2}}L_{\rho,m-k+1,r-\sigma+3}L_{\rho,m,r+1}}\leq \frac{C}{k+1},\quad \mathrm{if}\ 0\leq k\leq \big[\frac{m}{2}\big],\\
&\binom{m}{k}\frac{L_{\rho,m,r+\frac{1}{2}}^2}{L_{\rho,k,r-\sigma+3}L_{\rho,m-k+2,r-\sigma+\frac{3}{2}}L_{\rho,m,r+1}}\leq \frac{C}{m-k+1},\quad \mathrm{if}\  \big[\frac{m}{2}\big]+1\leq k\leq m.
    \end{aligned}
    \right.
\end{equation}
Then Young's inequality \eqref{young} with Cauchy inequality and \eqref{I11}-\eqref{ineq1} gives
\begin{equation}\label{est:I11}
    \begin{aligned}
&\sum^{+\infty}_{m=0}\sum^m_{k=0}L_{\rho,m,r+\frac{1}{2}}^2\binom{m}{k}\left|\inner{(\partial_x^kf)\partial_x^{m-k+1}f,\ \partial_x^mu}_{L^2}\right|\\
&\leq C\sum_{m=0}^{+\infty}\sum_{k=0}^{[\frac{m}{2}]}\frac{L_{\rho,k+1,r-\sigma+\frac{3}{2}}\norm{\partial_x^kf}_{L^\infty}}{k+1}\big(L_{\rho,m-k+1,r-\sigma+3}\norm{\partial_x^{m-k+1}f}_{L^2}\big)\\
&\qquad\times\big(L_{\rho,m,r+1}\norm{\partial_x^mu}_{L^2}\big)\\
&\quad +C\sum_{m=0}^{+\infty}\sum_{k=[\frac{m}{2}]+1}^{m}\frac{L_{\rho,m-k+2,r-\sigma+\frac{3}{2}}\norm{\partial_x^{m-k+1}f}_{L^\infty}}{m-k+1}\big(L_{\rho,k,r-\sigma+3}\norm{\partial_x^{k}f}_{L^2}\big)\\
&\qquad\times\big(L_{\rho,m,r+1}\norm{\partial_x^mu}_{L^2}\big)\\
&\leq C\sum_{m=0}^{+\infty}\frac{L_{\rho,m+1,r-\sigma+\frac{3}{2}}\norm{\partial_x^mf}_{L^\infty}}{m+1}\mathcal{Y}_\rho\leq C\mathcal{X}_\rho^\frac{1}{2}\mathcal{Y}_\rho,
    \end{aligned}
\end{equation}
the last line using the definitions of $\mathcal{X}_\rho$ and $\mathcal{Y}_\rho$ and the estimate that
\begin{equation}\label{ff}
\sum_{m=0}^{+\infty}\frac{L_{\rho,m+1,r-\sigma+\frac{3}{2}}\norm{\partial_x^mf}_{L^\infty}}{m+1}\leq C \Big(\sum_{m=0}^{+\infty}L_{\rho,m+1,r-\sigma+\frac{3}{2}}^2\norm{\partial_x^mf}_{L^\infty}^2\Big)^\frac{1}{2}\leq C\mathcal{X}_\rho^\frac{1}{2}
\end{equation}
implied by $L_{\rho,m+1,r}\leq CL_{\rho,m,r}$ for $m\ge 0$. Similarly,
\begin{equation*}
    \sum^{+\infty}_{m=0}\sum^m_{k=0}L_{\rho,m,r+\frac{1}{2}}^2\binom{m}{k}\left|\inner{(\partial_x^kg)\partial_x^{m-k}\partial_yf,\ \partial_x^mu}_{L^2}\right|\leq C\mathcal{X}_\rho^\frac{1}{2}\mathcal{Y}_\rho.
\end{equation*}
Combining this and estimate \eqref{est:I11} yields
\begin{equation}\label{est:I1}
  I_1\leq  C\mathcal{X}_\rho^\frac{1}{2}\mathcal{Y}_\rho. 
\end{equation}

\underline{\it The $I_2$ bound}. Following a similar decomposition in \eqref{I11} and using the definitions of $\mathcal{X}_\rho$ and $\mathcal{Y}_\rho$ as well as Young's inequality \eqref{young}, one can verify that
\begin{equation}\label{I21}
    \sum^{+\infty}_{m=0}\sum^m_{k=1}L_{\rho,m,r+\frac{1}{2}}^2\binom{m}{k}\left|\inner{(\partial_x^ku)\partial_x^{m-k+1}u,\ \partial_x^mu}_{L^2}\right|\leq C\mathcal{X}_\rho^\frac{1}{2}\mathcal{Y}_\rho. 
\end{equation}
For the remainder term in $I_2$, we split it as
\begin{equation*}
    \begin{aligned}
&\sum^{+\infty}_{m=0}\sum^{m-1}_{k=1}L_{\rho,m,r+\frac{1}{2}}^2\binom{m}{k}\left|\inner{(\partial_x^kv)\partial_x^{m-k}\partial_yu,\ \partial_x^mu}_{L^2}\right|\\
&\leq \sum_{m=0}^{+\infty}\sum_{k=1}^{[\frac{m-1}{2}]}L_{\rho,m,r+\frac{1}{2}}^2\binom{m}{k}\norm{\partial_x^kv}_{L^\infty}\norm{\partial_x^{m-k}\partial_yu}_{L^2}\norm{\partial_x^mu}_{L^2}\\
&\quad+\sum_{m=0}^{+\infty}\sum_{k=[\frac{m-1}{2}]+1}^{m-1}L_{\rho,m,r+\frac{1}{2}}^2\binom{m}{k}\norm{\partial_x^kv}_{L_x^2L_y^\infty}\norm{\partial_x^{m-k}\partial_yu}_{L_x^\infty L_y^2}\norm{\partial_x^mu}_{L^2}.\\
    \end{aligned}
\end{equation*}
Then Young's inequality \eqref{young}, with the fact that
\begin{equation*}
    \left\{
    \begin{aligned}
&\binom{m}{k}\frac{L_{\rho,m,r+\frac{1}{2}}^2}{L_{\rho,k+2,r+\frac{1}{2}}L_{\rho,m-k,r+\frac{1}{2}}L_{\rho,m,r+1}}\leq \frac{C}{k+1},\quad \mathrm{if}\ 1\leq k\leq \big[\frac{m-1}{2}\big],\\
&\binom{m}{k}\frac{L_{\rho,m,r+\frac{1}{2}}^2}{L_{\rho,k+1,r+1}L_{\rho,m-k+1,r}L_{\rho,m,r+1}}\leq \frac{C}{m-k+1},\ \ \mathrm{if}\  \big[\frac{m-1}{2}\big]+1\leq k\leq m-1,
    \end{aligned}
    \right.
\end{equation*}
gives
\begin{equation*}
    \sum^{+\infty}_{m=0}\sum^{m-1}_{k=1}L_{\rho,m,r+\frac{1}{2}}^2\binom{m}{k}\left|\inner{(\partial_x^kv)\partial_x^{m-k}\partial_yu,\ \partial_x^mu}_{L^2}\right|\leq C\mathcal{X}_\rho^\frac{1}{2}\mathcal{Y}_\rho^\frac{1}{2}\mathcal{Z}_\rho^\frac{1}{2}+ C\mathcal{X}_\rho^\frac{1}{2}\mathcal{Y}_\rho. 
\end{equation*}
This with \eqref{I21} yields
\begin{equation}\label{est:I2}
    I_2\leq C\mathcal{X}_\rho^\frac{1}{2}\mathcal{Y}_\rho^\frac{1}{2}\mathcal{Z}_\rho^\frac{1}{2}+C\mathcal{X}_\rho^\frac{1}{2}\mathcal{Y}_\rho. 
\end{equation}

\underline{\it The $I_3$ bound}. A direct computation with the definitions of $\mathcal{X}_\rho$ and $\mathcal{Y}_\rho$ gives
\begin{equation}\label{est:I3}
\begin{aligned}
   I_3&\leq \norm{\partial_yu}_{L^\infty}\sum_{m=0}^{+\infty}\frac{L_{\rho,m,r}}{L_{\rho,m+1,r+\sigma}}\big(L_{\rho,m+1,r+\sigma}\norm{\partial_x^{m+1}\phi}_{L^2}\big)\big(L_{\rho,m,r+1}\norm{\partial_x^mu}_{L^2}\big)\\
   &\leq C\big(\mathcal{X}_\rho^\frac{1}{2}+\norm{\partial_y^2u}_{\rho,r-\sigma}\big)\mathcal{Y}_\rho^\frac{1}{2}\norm{\phi}_{\rho,r+\sigma}\leq C\big(\mathcal{X}_\rho+\norm{\partial_y^2u}_{\rho,r-\sigma}^2\big)\mathcal{Y}_\rho+C\norm{\phi}_{\rho,r+\sigma}^2.
   \end{aligned}
\end{equation}
Finally, substituting \eqref{est:I1}, \eqref{est:I2} and \eqref{est:I3} into \eqref{est:uu} and integrating from $0$ to $T$ with respect to time, we obtain that, recalling $M$ is defined in \eqref{def:M},
\begin{align*}
   &\sup_{t\in[0,T]}\norm{u}_{\rho,r+\frac{1}{2}}^2+\beta\int^T_0\norm{u}_{\rho,r+1}^2dt+\int^T_0\norm{\partial_yu}_{\rho,r+\frac{1}{2}}^2dt\\
   &\leq C\sup_{t\in[0,T]}\big(\mathcal{X}_\rho^\frac{1}{2}+\mathcal{X}_\rho+\norm{\partial_y^2u}_{\rho,r-\sigma}^2\big)\int^T_0\mathcal{Y}_{\rho}dt+C\sup_{t\in[0,T]}\mathcal{X}_\rho^\frac{1}{2}\int^T_0\mathcal{Y}_\rho^\frac{1}{2}\mathcal{Z}_\rho^\frac{1}{2}dt\\
   &\quad +C\int^T_0\norm{\phi}^2_{\rho,r+\sigma}dt+C\norm{u_0}_{\rho_0,r+\frac{1}{2}}^2\\
&\leq CC_*M\int^T_0\mathcal{Y}_{\rho}dt+C\sqrt{C_*M}\int^T_0\mathcal{Y}_\rho^\frac{1}{2}\mathcal{Z}_\rho^\frac{1}{2}dt+C\int^T_0\norm{\phi}^2_{\rho,r+\sigma}dt+CM,
\end{align*}
the last line using assumption \eqref{ass:pri} along with Lemma \ref{lem:principle} and $C_*M\ge 1$. The proof of Lemma \ref{lem:u} is thus completed.
\end{proof}

\begin{lemma}\label{lem:yu}
Under the same assumption as given in Theorem \ref{thm:pri}, it holds that
\begin{equation}\label{est:yu}
    \begin{aligned}
 &\sup_{t\in [0,T]}\norm{\partial_yu}_{\rho,r}^2+\beta\int^T_0\norm{\partial_yu}_{\rho,r+\frac{1}{2}}^2dt+\int^T_0\norm{\partial_y^2u}_{\rho,r}^2dt\\
 &\leq C\delta^{-2}\int^T_0\norm{\phi}_{\rho,r+\sigma}^2dt+C\delta^{-3}C_*^\frac{3}{2}M^\frac{3}{2}\int^T_0\mathcal{Y}_\rho dt+C\delta^{-3}C_*^\frac{3}{2}M^\frac{3}{2}\int^T_0\mathcal{Y}_\rho^\frac{1}{2}\mathcal{Z}_\rho^\frac{1}{2} dt\\
 &\quad+C\delta^{-3}C_*^\frac{3}{2}M^\frac{3}{2}\int^T_0\mathcal{Y}_\rho^\frac{1}{4}\mathcal{Z}_\rho^\frac{3}{4}dt+C\delta^{-2}M,\\
    \end{aligned}
\end{equation}
 provided $\beta$ is sufficiently large. Recall $\mathcal{Y}_\rho$ and $\mathcal{Z}_\rho$ are defined in \eqref{def:energytwo} and $\phi$ is given in \eqref{def:phi}.
\end{lemma}

\begin{proof}
    For given $m\in\mathbb{Z}_+$, applying $\partial_x^m\partial_y$ to the velocity equation in \eqref{hyMHD} yields
\begin{multline}\label{4eq10}
 (\partial_t+u\partial^{m+1}_x+v\partial_x^m\partial_y-\partial_x^m\partial_y^2)\partial_yu=\partial_x^m\inner{f\partial_x\partial_yf+g\partial_y^2f}\\-\sum^m_{k=1}\binom{m}{k}(\partial_x^ku)\partial_x^{m-k+1}\partial_yu-\sum^{m-1}_{k=1}\binom{m}{k}(\partial_x^kv)\partial_x^{m-k}\partial_y^2u-(\partial_x^mv)\partial_y^2u.   
\end{multline} 
Noticing that
\begin{align*}
    \inner{(\partial_x^mv)\partial_y^2u,\ \frac{\partial_x^m\partial_yu}{\partial_y^2u}}_{L^2}=\inner{\partial_x^mv,\ \partial_x^m\partial_yu}_{L^2}=\inner{\partial_x^{m+1}u,\ \partial_x^mu}_{L^2}=0,
\end{align*}
we take the $L^2$-product with $\frac{\partial_x^m\partial_yu}{\partial_y^2u}$ on both sides of  \eqref{4eq10} to get
\begin{align*}
&\frac{1}{2}\frac{d}{dt}\left\|\frac{\partial_x^m\partial_yu}{\sqrt{\partial_y^2u}}\right\|_{L^2}^2+\left\|\frac{\partial_x^m\partial_y^2u}{\sqrt{\partial_y^2u}}\right\|_{L^2}^2\\
&= \Big(\partial_x^m\inner{f\partial_x\partial_yf+g\partial_y^2f},\ \frac{\partial_x^m\partial_yu}{\partial_y^2u}\Big)_{L^2}-\sum^m_{k=1}\binom{m}{k}\Big((\partial_x^ku)\partial_x^{m-k+1}\partial_yu,\ \frac{\partial_x^m\partial_yu}{\partial_y^2u}\Big)_{L^2}\\
&\quad-\sum^{m-1}_{k=1}\binom{m}{k}\inner{(\partial_x^kv)\partial_x^{m-k}\partial_y^2u,\ \frac{\partial_x^m\partial_yu}{\partial_y^2u}}_{L^2}-\frac{1}{2}\int_{\Omega}\frac{(\partial_x^m\partial_yu)^2\partial_t\partial_y^2u}{(\partial_y^2u)^2}dxdy\\
&\quad-\frac{1}{2}\int_{\Omega}\frac{(\partial_x^m\partial_yu)^2(u\partial_x+v\partial_y)\partial_y^2u}{(\partial_y^2u)^2}dxdy+\int_{\Omega}\frac{(\partial_x^m\partial_y^2u)(\partial_x^m\partial_yu)\partial_y^3u}{(\partial_y^2u)^2}dxdy
\\
&\quad+\int_{\mathbb{T}}\frac{(\partial_x^m\partial_y^2u)\partial_x^m\partial_yu}{\partial_y^2u}\bigg|^{y=1}_{y=0}dx\\
&\stackrel{\rm def}{=}T_{1m}+T_{2m}+T_{3m}+T_{4m}+T_{5m}+T_{6m}+T_{7m}.
\end{align*}
Then we multiply the above equality by $L_{\rho,m,r}^2$, use Proposition \ref{prop:principle} and \eqref{factone}, take summation over $m$ and then integrate on $[0,T]$ with respect to $t$ to get
\begin{multline}\label{4eq12}
\sup_{t\in [0,T]}\norm{\partial_yu}_{\rho,r}^2+\beta\int^T_0\norm{\partial_yu}_{\rho,r+\frac{1}{2}}^2dt+\int^T_0\norm{\partial_y^2u}_{\rho,r}^2dt\\
\leq C\delta^{-1}\sum_{i=1}^7\int^T_0 T_idt+C\delta^{-1}\sum_{m=0}^{+\infty}L_{\rho_0,m,r}^2\Big\|\frac{\partial_x^m\partial_yu_0}{\sqrt{\partial_y^2u_0}}\Big\|_{L^2}^2\\
\leq C\delta^{-1}\sum_{i=1}^7\int^T_0 T_idt+C\delta^{-2}M.
\end{multline}
Here $T_i$ is defined by
\begin{equation}\label{def:Ti}
    T_i\stackrel{\rm def}{=}\sum^{+\infty}_{m=0}L_{\rho,m,r}^2|T_{im}|.
\end{equation}
The rest of the proof is dedicated to estimating $T_i$ for $1\leq i\leq 7$.

\underline{\it The $T_{1}$, $T_2$ and $T_3$ bounds}. By Proposition \ref{prop:principle}, we write
\begin{equation}\label{4eq13}
\begin{aligned}
&\sum^{+\infty}_{m=0}L_{\rho,m,r}^2\left|\inner{\partial_x^m\inner{f\partial_x\partial_yf},\ \frac{\partial_x^m\partial_yu}{\partial_y^2u}}_{L^2}\right|\\
&\leq C\delta^{-1}\sum^{+\infty}_{m=0}\sum^{[\frac{m}{2}]}_{k=0}L_{\rho,m,r}^2\binom{m}{k}\norm{\partial_x^kf}_{L^{\infty}}\norm{\partial_x^{m-k+1}\partial_yf}_{L^2}\norm{\partial_x^m\partial_yu}_{L^2}\\
&\quad+C\delta^{-1}\sum^{+\infty}_{m=0}\sum^{m}_{k=[\frac{m}{2}]+1}L_{\rho,m,r}^2\binom{m}{k}\norm{\partial_x^kf}_{L^2}\norm{\partial_x^{m-k+1}\partial_yf}_{L^{\infty}}\norm{\partial_x^m\partial_yu}_{L^2}.
\end{aligned}
\end{equation}
On the other hand, we note that
\begin{equation}\label{ineq3}
    \left\{
    \begin{aligned}
&\binom{m}{k}\frac{L_{\rho,m,r}^2}{L_{\rho,k+1,r-\sigma+\frac{3}{2}}L_{\rho,m-k+1,r-\sigma+\frac{5}{2}}L_{\rho,m,r+\frac{1}{2}}}\leq \frac{C}{k+1},\quad \mathrm{if}\ 0\leq k\leq \big[\frac{m}{2}\big],\\
&\binom{m}{k}\frac{L_{\rho,m,r}^2}{L_{\rho,k,r-\sigma+3}L_{\rho,m-k+2,r-\sigma+1}L_{\rho,m,r+\frac{1}{2}}}\leq \frac{C}{m-k+1},\quad \mathrm{if}\  \big[\frac{m}{2}\big]+1\leq k\leq m.
    \end{aligned}
    \right.
\end{equation}
Then Young's inequality \eqref{young} with \eqref{ff} and \eqref{4eq13}-\eqref{ineq3} gives
\begin{equation*}
\sum^{+\infty}_{m=0}L_{\rho,m,r}^2\left|\inner{\partial_x^m\inner{f\partial_x\partial_yf},\ \frac{\partial_x^m\partial_yu}{\partial_y^2u}}_{L^2}\right|\leq C\delta^{-1}\mathcal{X}_\rho^\frac{1}{2}\mathcal{Y}_\rho.   
\end{equation*}
Similarly, we have
\begin{equation*}
\sum^{+\infty}_{m=0}L_{\rho,m,r}^2\left|\inner{\partial_x^m\inner{g\partial_y^2f},\ \frac{\partial_x^m\partial_yu}{\partial_y^2u}}_{L^2}\right|\leq C\delta^{-1}\mathcal{X}_\rho^\frac{1}{2}\mathcal{Y}_\rho.     
\end{equation*}
Recall $T_1$ is given in \eqref{def:Ti}. Combining the two estimates above yields
\begin{equation*}
    T_1\leq C\delta^{-1}\mathcal{X}_\rho^\frac{1}{2}\mathcal{Y}_\rho.     
\end{equation*}
Moreover, repeating the argument used in estimate \eqref{est:I2} with slight modifications, we use Lemma \ref{lem:principle} to conclude that for any $t\in [0,T]$,
\begin{equation}\label{est:T23}
T_2+T_3\leq C\delta^{-1}\big(\mathcal{X}_\rho^\frac{1}{2}+\sqrt{C_*M}\big)\mathcal{Y}_\rho+ C\delta^{-1}\mathcal{X}_\rho^\frac{1}{2}\mathcal{Y}_\rho^\frac{1}{2}\mathcal{Z}_\rho^\frac{1}{2}.  
\end{equation}

\underline{The $T_{4}$, $T_{5}$ and $T_{6}$ bounds}. Recalling $T_4$ is defined in \eqref{def:Ti}, we use Sobolev embedding inequality and the definitions of $\mathcal{X}_\rho$, $\mathcal{Y}_\rho$ and $\mathcal{Z}_\rho$ to conclude that
\begin{equation}\label{est:T4}
    \begin{aligned}
     & T_4\leq C\delta^{-2}\sum^{+\infty}_{m=0}L_{\rho,m,r}^2\norm{\partial_x^m\partial_yu}_{L^2}\norm{\partial_x^m\partial_yu}_{L_x^2L_y^{\infty}}\norm{\partial_t\partial_y^2u}_{L_x^{\infty}L_y^2}\\
      &\leq C\delta^{-2}\norm{\partial_t\partial_y^2u}_{\rho,r-\sigma}\sum^{+\infty}_{m=0}L_{\rho,m,r}^2\norm{\partial_x^m\partial_yu}_{L^2}\big(\norm{\partial_x^m\partial_yu}_{L^2}+\norm{\partial_x^m\partial_yu}_{L^2}^{\frac{1}{2}}\norm{\partial_x^m\partial_y^2u}_{L^2}^{\frac{1}{2}}\big)\\
      &\leq C\delta^{-2}\mathcal{X}_\rho\mathcal{Z}_\rho^\frac{1}{2}+C\delta^{-2}\mathcal{X}_\rho^\frac{3}{4}\mathcal{Z}_\rho^\frac{3}{4}\leq C\delta^{-2}\mathcal{X}_\rho^\frac{1}{2}\mathcal{Y}_\rho^\frac{1}{2}\mathcal{Z}_\rho^\frac{1}{2}+C\delta^{-2}\mathcal{X}_\rho^\frac{1}{2}\mathcal{Y}_\rho^\frac{1}{4}\mathcal{Z}_\rho^\frac{3}{4},
    \end{aligned}
\end{equation}
the last inequality using \eqref{xy}. 
On the other hand, for terms $T_5$ and $T_6$, by Lemma \ref{lem:principle} and estimates \eqref{xy} and \eqref{est:dy3u}, 
we have
\begin{equation}\label{est:T5}
    \begin{aligned}
      T_5&\leq C\delta^{-2}\sum^{+\infty}_{m=0}L_{\rho,m,r}^2\norm{\partial_x^m\partial_yu}_{L^2}\norm{\partial_x^m\partial_yu}_{L_x^2L_y^{\infty}}\norm{(u\partial_x+v\partial_y)\partial_y^2u}_{L_x^{\infty}L_y^2}\\
&\leq C\delta^{-2}C_*M\mathcal{X}_\rho^\frac{1}{2}\sum^{+\infty}_{m=0}L_{\rho,m,r}^2\norm{\partial_x^m\partial_yu}_{L^2}\inner{\norm{\partial_x^m\partial_yu}_{L^2}+\norm{\partial_x^m\partial_y^2u}_{L^2}}\\
&\leq C\delta^{-2}C_*M\mathcal{X}_\rho^\frac{3}{2}+C\delta^{-2}C_*M\mathcal{X}_\rho\mathcal{Z}_\rho^\frac{1}{2}\leq  C\delta^{-2}C_*M\mathcal{X}_\rho^\frac{1}{2}\big(\mathcal{Y}_\rho+\mathcal{Y}_\rho^\frac{1}{2}\mathcal{Z}_\rho^\frac{1}{2}\big),
    \end{aligned}
\end{equation}
where in the second inequality we have used that for any $t\in [0,T]$,
\begin{equation*}
\norm{(u\partial_x+v\partial_y)\partial_y^2u}_{L_x^{\infty}L_y^2}\leq  \norm{u}_{L^\infty}\norm{\partial_x\partial_y^2u}_{L_x^{\infty}L_y^2}+\norm{v}_{L^\infty}\norm{\partial_y^3u}_{L_x^{\infty}L_y^2}\leq CC_*M\mathcal{X}_\rho^\frac{1}{2},
\end{equation*}
and
\begin{equation}\label{est:T6}
    \begin{aligned}
      T_6&\leq C\delta^{-2}\sum^{+\infty}_{m=0}L_{\rho,m,r}^2\norm{\partial_x^m\partial_y^2u}_{L^2}\norm{\partial_x^m\partial_yu}_{L_x^2L_y^{\infty}}\norm{\partial_y^3u}_{L_x^{\infty}L_y^2}\\
&\leq C\delta^{-2}C_*M\sum^{+\infty}_{m=0}L_{\rho,m,r}^2\norm{\partial_x^m\partial_y^2u}_{L^2}\big(\norm{\partial_x^m\partial_yu}_{L^2}+\norm{\partial_x^m\partial_yu}_{L^2}^{\frac{1}{2}}\norm{\partial_x^m\partial_y^2u}_{L^2}^{\frac{1}{2}}\big)\\
      &\leq C\delta^{-2}C_*M\mathcal{X}_\rho^\frac{1}{2}\mathcal{Z}_\rho^\frac{1}{2}+C\delta^{-2}C_*M\mathcal{X}_\rho^\frac{1}{4}\mathcal{Z}_\rho^\frac{3}{4}\\
      &\leq C\delta^{-2}C_*M\mathcal{Y}_\rho^\frac{1}{2}\mathcal{Z}_\rho^\frac{1}{2}+C\delta^{-2}C_*M\mathcal{Y}_\rho^\frac{1}{4}\mathcal{Z}_\rho^\frac{3}{4}.
    \end{aligned}
\end{equation}

\underline{\it The $T_{7}$ bound}. Observing that
\begin{equation*}
\begin{aligned}
\partial_y^2u|_{y=0,1}=\partial_xp|_{y=0,1}-(f\partial_xf)|_{y=0,1},
\end{aligned}
\end{equation*}
we use Proposition \ref{prop:principle} to obtain
\begin{multline}\label{T7}
T_7 \leq  2\delta^{-1}\sum^{+\infty}_{m=0}L_{\rho,m,r}^2\norm{\partial_x^{m+1}p|_{y=0,1}}_{L_x^2}\norm{\partial_x^m\partial_yu}_{L_x^2L_y^{\infty}}\\
\quad+\sum^{+\infty}_{m=0}L_{\rho,m,r}^2\left|\int_{\mathbb{T}}\frac{(\partial_x^m\partial_yu)\partial_x^m(f\partial_xf)}{\partial_y^2u}\bigg|_{y=0}^{y=1}dx\right|.\\
\end{multline}
For the first term on the right-hand side of \eqref{T7}, we  use the fact that
\begin{equation*}
\partial_xp|_{y=0,1}=-2\int^1_0u\partial_xudy+2\int^1_0f\partial_xfdy+\int^1_0\partial_y^2udy-\int^1_0\int_{\mathbb{T}}\partial_y^2u dxdy  
\end{equation*}
to write
\begin{equation}\label{T712}
2\delta^{-1}\sum^{+\infty}_{m=0}L_{\rho,m,r}^2\norm{\partial_x^{m+1}p|_{y=0,1}}_{L_x^2}\norm{\partial_x^m\partial_yu}_{L_x^2L_y^{\infty}}\leq C\delta^{-1}(T_{7,1}+T_{7,2}), 
\end{equation}
where
\begin{equation*}
    \left\{
    \begin{aligned}
&T_{7,1}=   \sum^{+\infty}_{m=0}L_{\rho,m,r}^2\left\|\partial_x^m\int^1_0u\partial_xudy\right\|_{L^2}\norm{\partial_x^m\partial_yu}_{L_x^2L_y^{\infty}},\\
&T_{7,2}=\sum^{+\infty}_{m=0}L_{\rho,m,r}^2\big(\norm{\partial_x^m(f\partial_xf)}_{L^2}\norm{\partial_x^m\partial_yu}_{L_x^2L_y^{\infty}}+\norm{\partial_x^m\partial_y^2u}_{L^2}\norm{\partial_x^m\partial_yu}_{L_x^2L_y^{\infty}}\big).
    \end{aligned}
    \right.
\end{equation*}
To estimate $T_{7,1}$, we use the fact $\int^1_0u\partial_x^{m+1}udy=-\int^1_0(\partial_yu)\partial_x^{m+1}\phi dy$ implied by \eqref{def:phi} to write
\begin{equation}\label{T711}
\begin{aligned}
T_{7,1}  &\leq\sum^{+\infty}_{m=0}L_{\rho,m,r}^2\|\partial_yu\|_{L^{\infty}}\norm{\partial_x^{m+1}\phi}_{L^2}\norm{\partial_x^{m}\partial_yu}_{L_x^2L_y^{\infty}}\\
&\quad+ \sum^{+\infty}_{m=0}\sum^{[\frac{m}{2}]}_{k=1}L_{\rho,m,r}^2\binom{m}{k}\norm{\partial_x^ku}_{L^{\infty}}\norm{\partial_x^{m-k+1}u}_{L^2}\norm{\partial_x^m\partial_yu}_{L_x^2L_y^{\infty}} \\
&\quad+\sum^{+\infty}_{m=0}\sum^{m}_{k=[\frac{m}{2}]+1}L_{\rho,m,r}^2\binom{m}{k}\norm{\partial_x^ku}_{L^{2}}\norm{\partial_x^{m-k+1}u}_{L^{\infty}}\norm{\partial_x^m\partial_yu}_{L_x^2L_y^{\infty}}.
    \end{aligned}
\end{equation}
For the first term on the right-hand side of \eqref{T711}, observing that
\begin{equation}\label{yinf1}
\begin{aligned}
\sum_{m=0}^{+\infty}L_{\rho,m,r}^2\norm{\partial_x^m\partial_yu}_{L_x^2L_y^\infty}^2&\leq C \sum_{m=0}^{+\infty}L_{\rho,m,r}^2\big(\norm{\partial_x^{m}\partial_yu}_{L^2}^2+\norm{\partial_x^{m}\partial_yu}_{L^2}\norm{\partial_x^{m}\partial_y^2u}_{L^2}\big)\\
&\leq C \mathcal{X}_\rho+C\mathcal{X}_\rho^\frac{1}{2}\mathcal{Z}_{\rho}^\frac{1}{2}\leq C \mathcal{Y}_\rho+C\mathcal{Y}_\rho^\frac{1}{2}\mathcal{Z}_{\rho}^\frac{1}{2}
\end{aligned}
\end{equation}
implied by \eqref{xy}, and
\begin{equation*}
    \sup_{t\in [0,T]}\|\partial_yu\|_{L^{\infty}}\leq C\sqrt{C_*M}
\end{equation*}
implied by assumption \eqref{ass:pri} and Lemma \ref{lem:principle},
we have
\begin{equation}\label{est:T7111}
\begin{aligned}
&\sum^{+\infty}_{m=0}L_{\rho,m,r}^2\|\partial_yu\|_{L^{\infty}}\norm{\partial_x^{m+1}\phi}_{L^2}\norm{\partial_x^{m}\partial_yu}_{L_x^2L_y^{\infty}}\\
&\leq C\sqrt{C_*M}\sum_{m=0}^{+\infty}\big(L_{\rho,m+1,r+\sigma}\norm{\partial_x^{m+1}\phi}_{L^2}\big)\big(L_{\rho,m,r}\norm{\partial_x^{m}\partial_yu}_{L_x^2L_y^{\infty}}\big)\\
&\leq C\norm{\phi}_{\rho,r+\sigma}^2+CC_*M\big(\mathcal{Y}_\rho+\mathcal{Y}_\rho^\frac{1}{2}\mathcal{Z}_{\rho}^\frac{1}{2}\big).
\end{aligned}
\end{equation}
 For the last two terms on the right-hand side of \eqref{T711}, repeating the argument in \eqref{est:I11}, we use Young's inequality \eqref{young}, \eqref{yinf1} and the following estimate that
 \begin{equation*}
     \left\{
     \begin{aligned}
&\binom{m}{k}\frac{L_{\rho,m,r}^2}{L_{\rho,k+1,r}L_{\rho,m-k+1,r+1}L_{\rho,m,r}}\leq \frac{C}{k+1},\quad &&\mathrm{if}\ 1\leq k\leq \big[\frac{m}{2}\big],\\
&\binom{m}{k}\frac{L_{\rho,m,r}^2}{L_{\rho,k,r+1}L_{\rho,m-k+2,r}L_{\rho,m,r}}\leq \frac{C}{m-k+1},\quad &&\mathrm{if}\  \big[\frac{m}{2}\big]+1\leq k\leq m,         
     \end{aligned}
     \right.
 \end{equation*}
to conclude that
\begin{align*}
  &\sum^{+\infty}_{m=0}\sum^{[\frac{m}{2}]}_{k=1}L_{\rho,m,r}^2\binom{m}{k}\norm{\partial_x^ku}_{L^{\infty}}\norm{\partial_x^{m-k+1}u}_{L^2}\norm{\partial_x^m\partial_yu}_{L_x^2L_y^{\infty}} \\
&+\sum^{+\infty}_{m=0}\sum^{m}_{k=[\frac{m}{2}]+1}L_{\rho,m,r}^2\binom{m}{k}\norm{\partial_x^ku}_{L^{2}}\norm{\partial_x^{m-k+1}u}_{L^{\infty}}\norm{\partial_x^m\partial_yu}_{L_x^2L_y^{\infty}}\\
&\leq C\mathcal{X}_\rho^\frac{1}{2}\mathcal{Y}_\rho^\frac{1}{2}\big(\mathcal{Y}_\rho^\frac{1}{2}+\mathcal{Y}_\rho^\frac{1}{4}\mathcal{Z}_{\rho}^\frac{1}{4}\big)\leq C\mathcal{X}_\rho^\frac{1}{2}\mathcal{Y}_\rho+C\mathcal{X}_\rho^\frac{1}{2}\mathcal{Y}_\rho^\frac{1}{2}\mathcal{Z}_\rho^\frac{1}{2}.
\end{align*}
Substituting the above estimate and \eqref{est:T7111} into \eqref{T711} yields
\begin{equation}\label{est:T71}
    T_{7,1}\leq C\norm{\phi}_{\rho,r+\sigma}^2+C\big(C_*M+\mathcal{X}_\rho^\frac{1}{2}\big)\big(\mathcal{Y}_\rho+\mathcal{Y}_\rho^\frac{1}{2}\mathcal{Z}_{\rho}^\frac{1}{2}\big).
\end{equation}
We now turn to deal with the term $T_{7,2}$. For the first term in $T_{7,2}$, we begin by decomposing it as
\begin{equation*}
    \begin{aligned}
&\sum^{+\infty}_{m=0}L_{\rho,m,r}^2\norm{\partial_x^m(f\partial_xf)}_{L^2}\norm{\partial_x^m\partial_yu}_{L_x^2L_y^{\infty}}\\&
\leq\sum^{+\infty}_{m=0}\sum^{[\frac{m}{2}]}_{k=0}L_{\rho,m,r}^2\binom{m}{k}\norm{\partial_x^kf}_{L^{\infty}}\norm{\partial_x^{m-k+1}f}_{L^2}\norm{\partial_x^m\partial_yu}_{L_x^2L_y^{\infty}}\\&\quad+\sum^{+\infty}_{m=0}\sum^{m}_{k=[\frac{m}{2}]+1}L_{\rho,m,r}^2\binom{m}{k}\norm{\partial_x^kf}_{L^{2}}\norm{\partial_x^{m-k+1}f}_{L^{\infty}}\norm{\partial_x^m\partial_yu}_{L_x^2L_y^{\infty}}.
    \end{aligned}
\end{equation*}
Following the argument above, we combine the above estimate with the fact that
 \begin{equation*}
     \left\{
     \begin{aligned}
&\binom{m}{k}\frac{L_{\rho,m,r}}{L_{\rho,k+1,r-\sigma+\frac{3}{2}}L_{\rho,m-k+1,r-\sigma+3}}\leq \frac{C}{k+1},\quad &&\mathrm{if}\ 0\leq k\leq \big[\frac{m}{2}\big],\\
&\binom{m}{k}\frac{L_{\rho,m,r}}{L_{\rho,k,r-\sigma+3}L_{\rho,m-k+2,r-\sigma+\frac{3}{2}}}\leq \frac{C}{m-k+1},\quad &&\mathrm{if}\  \big[\frac{m}{2}\big]+1\leq k\leq m,         
     \end{aligned}
     \right.
 \end{equation*}
 to conclude that
\begin{equation*}
    \begin{aligned}
\sum^{+\infty}_{m=0}L_{\rho,m,r}^2\norm{\partial_x^m(f\partial_xf)}_{L^2}\norm{\partial_x^m\partial_yu}_{L_x^2L_y^{\infty}}\leq C\mathcal{X}_\rho^\frac{1}{2}\mathcal{Y}_\rho+C\mathcal{X}_\rho^\frac{1}{2}\mathcal{Y}_\rho^\frac{1}{2}\mathcal{Z}_\rho^\frac{1}{2}.
    \end{aligned}
\end{equation*}
For the remainder term in $T_{7,2}$, a direct computation with \eqref{xy} and \eqref{yinf1} yields
\begin{equation*}
\sum^{+\infty}_{m=0}L_{\rho,m,r}^2\norm{\partial_x^m\partial_y^2u}_{L^2}\norm{\partial_x^m\partial_yu}_{L_x^2L_y^{\infty}}  \leq C\mathcal{Y}_\rho^\frac{1}{2}\mathcal{Z}_\rho^\frac{1}{2}+C\mathcal{Y}_\rho^\frac{1}{4}\mathcal{Z}_\rho^\frac{3}{4}.
\end{equation*}
Combing the two estimates above gives
\begin{equation*}
    T_{7,2}\leq C\mathcal{X}_\rho^\frac{1}{2}\mathcal{Y}_\rho+C\big(1+\mathcal{X}_\rho^\frac{1}{2}\big)\mathcal{Y}_\rho^\frac{1}{2}\mathcal{Z}_\rho^\frac{1}{2}+C\mathcal{Y}_\rho^\frac{1}{4}\mathcal{Z}_\rho^\frac{3}{4}.
\end{equation*}
Substituting this and \eqref{est:T71} into \eqref{T712} and observing $C_*M\ge 1$, we obtain that
\begin{multline}
    \label{est:T712}
2\delta^{-1}\sum^{+\infty}_{m=0}L_{\rho,m,r}^2\norm{\partial_x^{m+1}p|_{y=0,1}}_{L_x^2}\norm{\partial_x^m\partial_yu}_{L_x^2L_y^{\infty}}\\\leq  C\delta^{-1}\norm{\phi}_{\rho,r+\sigma}^2+C\delta^{-1}\big(C_*M+\mathcal{X}_\rho^\frac{1}{2}\big)\big(\mathcal{Y}_\rho+\mathcal{Y}_\rho^\frac{1}{2}\mathcal{Z}_{\rho}^\frac{1}{2}\big)+C\delta^{-1}\mathcal{Y}_\rho^\frac{1}{4}\mathcal{Z}_\rho^\frac{3}{4}.   
\end{multline}
On the other hand, for the remainder term on the right-hand side of \eqref{T7}, we write
\begin{equation*}
    \begin{aligned}
&\int_{\mathbb{T}}\frac{(\partial_x^m\partial_yu)\partial_x^m(f\partial_xf)}{\partial_y^2u}\bigg|_{y=0}^{y=1}dx= \int_{\Omega}\partial_y\left(\frac{(\partial_x^m\partial_yu)\partial_x^m(f\partial_xf)}{\partial_y^2u}\right)dxdy\\
&=\int\bigg[\frac{(\partial_x^m\partial_y^2u)\partial_x^m(f\partial_xf)}{\partial_y^2u}+\frac{(\partial_x^m\partial_yu)\partial_x^m\partial_y(f\partial_xf)}{\partial_y^2u}-\frac{(\partial_y^3u)(\partial_x^m\partial_yu)\partial_x^m(f\partial_xf)}{(\partial_y^2u)^2}\bigg]dxdy.
    \end{aligned}
\end{equation*}
This enables us to repeat the argument above to conclude that for any $t\in[0,T]$,
\begin{multline}\label{est:T73}
\sum^{+\infty}_{m=0}L_{\rho,m,r}^2\left|\int_{\mathbb{T}}\frac{(\partial_x^m\partial_yu)\partial_x^m(f\partial_xf)}{\partial_y^2u}\bigg|_{y=0}^{y=1}dx\right|\\
\leq C\delta^{-1}\mathcal{X}_\rho^\frac{1}{2}\mathcal{Y}_\rho^\frac{1}{2}\mathcal{Z}_\rho^\frac{1}{2}+C\delta^{-2}C_*M\mathcal{X}_\rho^\frac{1}{2}\mathcal{Y}_\rho^\frac{1}{2}\mathcal{Z}_\rho^\frac{1}{2}.   
\end{multline}
Substituting \eqref{est:T712} and \eqref{est:T73} into \eqref{T7} gives for any $t\in[0,T]$,
\begin{multline*}
    T_7\leq C\delta^{-1}\norm{\phi}_{\rho,r+\sigma}^2+C\delta^{-1}\big(C_*M+\mathcal{X}_\rho^\frac{1}{2}\big)\big(\mathcal{Y}_\rho+\mathcal{Y}_\rho^\frac{1}{2}\mathcal{Z}_{\rho}^\frac{1}{2}\big)\\
    +C\delta^{-1}\mathcal{Y}_\rho^\frac{1}{4}\mathcal{Z}_\rho^\frac{3}{4}+C\delta^{-2}C_*M\mathcal{X}_\rho^\frac{1}{2}\mathcal{Y}_\rho^\frac{1}{2}\mathcal{Z}_\rho^\frac{1}{2}.  
\end{multline*}

Finally, substituting these estimates of $T_{1}-T_{7}$ into \eqref{4eq12} and using assumption \eqref{ass:pri} along with $0<\delta<\frac{1}{2}$ and $C_*M\ge 1$, we obtain the desired assertion \eqref{est:yu}. This completes the proof of Lemma \ref{lem:yu}.
\end{proof}

\begin{lemma}\label{lem:fyf}
Under the same assumption as given in Theorem \ref{thm:pri}, it holds that
\begin{equation}\label{est:f}
     \begin{aligned}
&\sup_{t\in[0,T]}\inner{\norm{(\partial_tf,\partial_yf)}_{\rho,r-\sigma+\frac{3}{2}}^2+\norm{f}_{\rho,r-\sigma+\frac{5}{2}}^2}+\int^T_0\norm{\partial_tf}_{\rho,r-\sigma+\frac{3}{2}}^2dt\\
&\quad+\beta\int^T_0\inner{\norm{(\partial_tf,\partial_yf)}_{\rho,r-\sigma+2}^2+\norm{f}_{\rho,r-\sigma+3}^2}dt\\
&\leq C\sqrt{C_*M}\int^T_0\mathcal{Y}_\rho dt+C\sqrt{C_*M}\int^T_0\mathcal{Y}_\rho ^\frac{1}{2}\mathcal{Z}_\rho^\frac{1}{2}dt+CM,
     \end{aligned}    
\end{equation}
and
\begin{equation}\label{est:yf}
     \begin{aligned}
&\sup_{t\in[0,T]}\inner{\norm{(\partial_t\partial_yf,\partial_y^2f)}_{\rho,r-\frac{1}{2}}^2+\norm{\partial_yf}_{\rho,r+\frac{1}{2}}^2}+\int^T_0\norm{\partial_t\partial_yf}_{\rho,r-\frac{1}{2}}^2dt\\
&\quad+\beta\int^T_0\inner{\norm{(\partial_t\partial_yf,\partial_y^2f)}_{\rho,r}^2+\norm{\partial_yf}_{\rho,r+1}^2}dt\\
&\leq C\sqrt{C_*M}\int^T_0\mathcal{Y}_\rho dt+C\sqrt{C_*M}\int^T_0\mathcal{Y}_\rho ^\frac{1}{2}\mathcal{Z}_\rho^\frac{1}{2}dt+CM,
     \end{aligned}    
\end{equation}
provided $\beta$ is sufficiently large. Recall $\mathcal{Y}_\rho$ and $\mathcal{Z}_\rho$ are defined in \eqref{def:energytwo}.
\end{lemma}

\begin{proof}
    We only give the proof of assertion \eqref{est:f}, as assertion \eqref{est:yf} can be handled in the same way. For given $m\in\mathbb{Z}_+$, applying $\partial_x^m$ to equation for the tangential magnetic field in system \eqref{hyMHD} yields
\begin{equation}\label{5eq14}
\begin{aligned}
\partial_t^2\partial_x^mf+\partial_t\partial_x^mf-\partial_x^m\partial_y^2f=&\sum^m_{k=0}\binom{m}{k}\big[(\partial_x^kf)\partial_x^{m-k+1}u+(\partial_x^kg)\partial_x^{m-k}\partial_yu\big]\\
&-\sum^m_{k=0}\binom{m}{k}\big[(\partial_x^ku)\partial_x^{m-k+1}f+(\partial_x^kv)\partial_x^{m-k}\partial_yf\big].
\end{aligned}
\end{equation}
Then we take the $L^2$-product with $\partial_t\partial_x^mf$ on both sides of \eqref{5eq14}, multiply by $L_{\rho,m,r-\sigma+\frac{3}{2}}^2$, use the fact \eqref{factone}, take summation over $m\in\mathbb{Z}_+$ and then integrate over $[0,T]$ with respect to time to deduce that
 \begin{multline}\label{tf}
\sup_{t\in[0,T]}\norm{(\partial_tf,\partial_yf)}_{\rho,r-\sigma+\frac{3}{2}}^2+\beta\int^T_0\norm{(\partial_tf,\partial_yf)}_{\rho,r-\sigma+2}^2dt\\
+\int^T_0\norm{\partial_tf}_{\rho,r-\sigma+\frac{3}{2}}^2dt
\leq C\int^T_0(S_1+S_2)dt+CM,
\end{multline}
where
	\begin{equation*}
		\left\{
		\begin{aligned}
S_1=&\sum^{+\infty}_{m=0}\sum^m_{k=0}L_{\rho,m,r-\sigma+\frac{3}{2}}^2\binom{m}{k}\left|\inner{(\partial_x^kf)\partial_x^{m-k+1}u+(\partial_x^kg)\partial_x^{m-k}\partial_yu,\ \partial_t\partial_x^mf}_{L^2}\right|,\\
S_2=&\sum^{+\infty}_{m=0}\sum^m_{k=0}L_{\rho,m,r-\sigma+\frac{3}{2}}^2\binom{m}{k}\left|\inner{(\partial_x^ku)\partial_x^{m-k+1}f+(\partial_x^kv)\partial_x^{m-k}\partial_yf,\ \partial_t\partial_x^mf}_{L^2}\right|.\\
		\end{aligned}
		\right.
	\end{equation*}
To estimate $S_1$, we first estimate that
\begin{align*}
    &\sum^{+\infty}_{m=0}\sum^m_{k=0}L_{\rho,m,r-\sigma+\frac{3}{2}}^2\binom{m}{k}\left|\inner{(\partial_x^kf)\partial_x^{m-k+1}u,\ \partial_t\partial_x^mf}_{L^2}\right|\\
    &\leq \sum^{+\infty}_{m=0}\sum^{[\frac{m}{2}]}_{k=0}L_{\rho,m,r-\sigma+\frac{3}{2}}^2\binom{m}{k}\norm{\partial_x^kf}_{L^\infty}\norm{\partial_x^{m-k+1}u}_{L^2}\norm{\partial_t\partial_x^mf}_{L^2}\\
    &\quad+\sum^{+\infty}_{m=0}\sum^{m}_{k=[\frac{m}{2}]+1}L_{\rho,m,r-\sigma+\frac{3}{2}}^2\binom{m}{k}\norm{\partial_x^kf}_{L^2}\norm{\partial_x^{m-k+1}u}_{L^\infty}\norm{\partial_t\partial_x^mf}_{L^2}.
\end{align*}
Then following an analogous argument with estimate \eqref{est:I1}, we use the estimate instead that
\begin{equation*}
    \left\{
    \begin{aligned}
&\binom{m}{k}\frac{L_{\rho,m,r-\sigma+\frac{3}{2}}^2}{L_{\rho,k+1,r-\sigma+\frac{3}{2}}L_{\rho,m-k+1,r+1}L_{\rho,m,r-\sigma+2}}\leq \frac{C}{k+1},\quad &&\mathrm{if}\ 0\leq k\leq \big[\frac{m}{2}\big],\\
&\binom{m}{k}\frac{L_{\rho,m,r-\sigma+\frac{3}{2}}^2}{L_{\rho,k,r-\sigma+3}L_{\rho,m-k+2,r}L_{\rho,m,r-\sigma+2}}\leq \frac{C}{m-k+1},\quad &&\mathrm{if}\  \big[\frac{m}{2}\big]+1\leq k\leq m,
    \end{aligned}
    \right.
\end{equation*}
to conclude
\begin{equation}\label{s11}
    \sum^{+\infty}_{m=0}\sum^m_{k=0}L_{\rho,m,r-\sigma+\frac{3}{2}}^2\binom{m}{k}\left|\inner{(\partial_x^kf)\partial_x^{m-k+1}u,\ \partial_t\partial_x^mf}_{L^2}\right|\leq C\mathcal{X}_\rho^\frac{1}{2}\mathcal{Y}_\rho,
\end{equation}
recalling that $X_\rho$ and $\mathcal{Y}_\rho$ are defined in \eqref{def:energytwo}. For the remainder term in $S_1$, we split it as
\begin{align*}
&\sum^{+\infty}_{m=0}\sum^m_{k=0}L_{\rho,m,r-\sigma+\frac{3}{2}}^2\binom{m}{k}\left|\inner{(\partial_x^kg)\partial_x^{m-k}\partial_yu,\ \partial_t\partial_x^mf}_{L^2}\right|\\
 &\leq \sum^{+\infty}_{m=0}\sum^{[\frac{m}{2}]}_{k=0}L_{\rho,m,r-\sigma+\frac{3}{2}}^2\binom{m}{k}\norm{\partial_x^kg}_{L^\infty}\norm{\partial_x^{m-k}\partial_yu}_{L^2}\norm{\partial_t\partial_x^mf}_{L^2}\\
    &\quad+\sum^{+\infty}_{m=0}\sum^{m}_{k=[\frac{m}{2}]+1}L_{\rho,m,r-\sigma+\frac{3}{2}}^2\binom{m}{k}\norm{\partial_x^kg}_{L_x^2L_y^\infty}\norm{\partial_x^{m-k}\partial_yu}_{L_x^\infty L_y^2}\norm{\partial_t\partial_x^mf}_{L^2}.
\end{align*}
This, with the estimate 
\begin{equation*}
    \left\{
    \begin{aligned}
&\binom{m}{k}\frac{L_{\rho,m,r-\sigma+\frac{3}{2}}^2}{L_{\rho,k+2,r-\sigma+\frac{5}{2}}L_{\rho,m-k,r+\frac{1}{2}}L_{\rho,m,r-\sigma+2}}\leq \frac{C}{k+1},\quad \mathrm{if}\ 0\leq k\leq \big[\frac{m}{2}\big],\\
&\binom{m}{k}\frac{L_{\rho,m,r-\sigma+\frac{3}{2}}^2}{L_{\rho,k+1,r-\sigma+3}L_{\rho,m-k+1,r}L_{\rho,m,r-\sigma+2}}\leq \frac{C}{m-k+1},\quad \mathrm{if}\  \big[\frac{m}{2}\big]+1\leq k\leq m,
    \end{aligned}
    \right.
\end{equation*}
gives
\begin{equation}\label{s12}
 \sum^{+\infty}_{m=0}\sum^m_{k=0}L_{\rho,m,r-\sigma+\frac{3}{2}}^2\binom{m}{k}\left|\inner{(\partial_x^kg)\partial_x^{m-k}\partial_yu,\ \partial_t\partial_x^mf}_{L^2}\right|\leq C\mathcal{X}_\rho^\frac{1}{2}\mathcal{Y}_\rho^\frac{1}{2}\mathcal{Z}_\rho^\frac{1}{2}+C\mathcal{X}_\rho^\frac{1}{2}\mathcal{Y}_\rho.   
\end{equation}
Combining \eqref{s11} and \eqref{s12} yields
\begin{equation*}
    S_1\leq C\mathcal{X}_\rho^\frac{1}{2}\mathcal{Y}_\rho^\frac{1}{2}\mathcal{Z}_\rho^\frac{1}{2}+C\mathcal{X}_\rho^\frac{1}{2}\mathcal{Y}_\rho. 
\end{equation*}
Similarly,
\begin{equation*}
    S_2\leq C\mathcal{X}_\rho^\frac{1}{2}\mathcal{Y}_\rho^\frac{1}{2}\mathcal{Z}_\rho^\frac{1}{2}+C\mathcal{X}_\rho^\frac{1}{2}\mathcal{Y}_\rho. 
\end{equation*}
Substituting the two estimates into \eqref{tf} and using assumption \eqref{ass:convex}, we obtain
 \begin{multline}\label{est:tf}
\sup_{t\in[0,T]}\norm{(\partial_tf,\partial_yf)}_{\rho,r-\sigma+\frac{3}{2}}^2+\beta\int^T_0\norm{(\partial_tf,\partial_yf)}_{\rho,r-\sigma+2}^2dt\\
+\int^T_0\norm{\partial_tf}_{\rho,r-\sigma+\frac{3}{2}}^2dt
\leq C\sqrt{C_*M}\int^T_0\mathcal{Y}_\rho dt+C\sqrt{C_*M}\int^T_0\mathcal{Y}_\rho ^\frac{1}{2}\mathcal{Z}_\rho^\frac{1}{2}dt+CM.
\end{multline}

On the other hand, observing that
\begin{equation*}
        \frac{1}{2}\frac{d}{dt}\norm{\partial_x^mf}_{L^2}^2=\inner{\partial_t\partial_x^mf,\ \partial_x^mf}_{L^2}\leq \norm{\partial_t\partial_x^mf}_{L^2}\norm{\partial_x^mf}_{L^2},
\end{equation*}
we use \eqref{factone} to deduce that
\begin{align*}
  &\sup_{t\in [0,T]}\norm{f}_{\rho,r-\sigma+\frac{5}{2}}^2+\beta\int^T_0 \norm{f}_{\rho,r-\sigma+3}^2dt\\
  &\leq C\int^T_0\sum_{m=0}^{+\infty}L_{\rho,m,r-\sigma+\frac{5}{2}}^2\norm{\partial_t\partial_x^mf}_{L^2}\norm{\partial_x^mf}_{L^2}dt+CM\\
  &\leq C\int^T_0\norm{\partial_tf}_{\rho,r-\sigma+2}\norm{f}_{\rho,r-\sigma+3}dt+CM\leq C\int^T_0\mathcal{Y}_\rho dt+CM.
\end{align*}
Combining this and \eqref{est:tf} yields the desired assertion \eqref{est:f}. This completes the proof of Lemma \ref{lem:fyf}.
\end{proof}

\begin{proof}[Completing the proof of Proposition \ref{prop:x0}]
Recall that $\mathcal{X}_{\rho,0}$, $\mathcal{Y}_{\rho,0}$ and $\mathcal{Z}_{\rho,0}$ are defined in \eqref{def:energy}. Assertion \eqref{est:x0} follows by combining these estimates in Lemmas \ref{lem:u}-\ref{lem:fyf} with $C_*M\ge 1$ and $0<\delta<\frac{1}{2}$. This completes the proof of Proposition \ref{prop:x0}.  
\end{proof}

\section{Estimate on $\phi$}
This section is devoted to handling the term $\int^T_0\norm{\phi(t)}_{\rho,r+\sigma}^2dt$ appearing in Proposition \ref{prop:x0}. The precise statement can be presented as follows.
\begin{proposition}\label{prop:phi}
Under the same assumption as given in Theorem \ref{thm:pri}, it holds that
\begin{equation}\label{est:phi}
    \int^T_0\norm{\phi(t)}_{\rho,r+\sigma}^2dt\leq C\int^T_0\mathcal{Y}_\rho dt+\frac{C\delta^{-2}C_*M}{\beta},
\end{equation}
provided $\beta$ is sufficiently large.
\end{proposition}

Establishing estimate \eqref{est:phi} represents the central challenge in this work due to its highly non-trivial nature. To address this, we will employ the boundary decomposition method introduced in \cite{MR4803680,MR4818200} with slight modifications.

\subsection{Boundary decomposition}
Recalling $\phi$ and $\mathcal{C}(t)$ are given in \eqref{def:phi}, we set
\begin{equation}\label{def:phia}
 \phi_a(t,x,y)\stackrel{\rm def}{=}\phi(t,x,y)-\phi(0,x,y)=\phi-\phi_0,   
\end{equation}
which satisfies
\begin{equation*}
    \left\{
    \begin{aligned}
&\partial_t\partial_y^2\phi_a+u\partial_x\partial_yu+v\partial_y^2u-\partial_y^4\phi_a=f\partial_x\partial_yf+g\partial_y^2f+\partial_y^4\phi_0,\\
&\phi_a|_{y=0,1}=0,\quad \partial_y\phi_a|_{y=0,1}=\mathcal{C}(0)-\mathcal{C}(t),\\
&\phi_a|_{t=0}=0.
    \end{aligned}
    \right.
\end{equation*}
In view of \eqref{def:phia}, we have
\begin{equation*}
\left\{
    \begin{aligned}
&u=\partial_y\phi+\mathcal{C}(t)=\partial_y\phi_a+\partial_y\phi_0+\mathcal{C}(t),\\
&v=-\partial_x\phi=-\partial_x\phi_a-\partial_x\phi_0.
    \end{aligned}
    \right.
\end{equation*}
Inspired by \cite{MR4803680,MR4818200}, we 
decompose $\{\partial_x^m\phi_a\}_{m\ge0}$ as
\begin{align}\label{de:phia}
   \{\partial_x^m \phi_a\}_{m\ge 0}=\{\phi_{s,(m)}\}_{m\ge0}+\{\phi_{b,(m)}\}_{m\ge 0}\stackrel{\rm def}{=}\vec{\phi}_s+\vec{\phi}_b,
\end{align}
where $\vec{\phi}_s$ enjoying a good
boundary condition satisfies 
\begin{equation}\label{eq:phis}
		\left\{
		\begin{aligned}
&\partial_t\partial_y^2\phi_{s,(m)}+u\partial_x\partial_y^2\phi_{s,(m)}-(\partial_x\phi_{s,(m)})\partial_y^2u-\partial_y^4\phi_{s,(m)}\\
&\quad=-\sum^{[\frac{m}{2}]}_{k=1}\binom{m}{k}(\partial_x^ku)\partial_y^2\phi_{s,(m-k+1)}-\sum^m_{k=[\frac{m}{2}]+1}\binom{m}{k}(\partial_y\phi_{s,(k)})\partial_x^{m-k+1}\partial_yu\\
&\qquad+\sum^{m-1}_{k=[\frac{m}{2}]+1}\binom{m}{k}\phi_{s,(k+1)}\partial_x^{m-k}\partial_y^2u-\sum^{[\frac{m}{2}]}_{k=0}\binom{m}{k}(\partial_x^kv)\partial_y^3\phi_{s,(m-k)}\\
&\qquad+\partial_x^m(f\partial_x\partial_yf+g\partial_y^2f)+\partial_x^m\partial_y^4\phi_0+\mathcal{R}_{s,(m)},\\
&\phi_{s,(m)}|_{y=0,1}=0,\quad\partial_y^2\phi_{s,(m)}|_{y=0,1}=0,\\
&\phi_{s,(m)}|_{t=0}=0,
		\end{aligned}
		\right.
\end{equation}
and $\vec{\phi}_b$ recovering the non-slip boundary condition satisfies  
\begin{equation}\label{eq:phib}
		\left\{
		\begin{aligned}
&\partial_t\partial_y^2\phi_{b,(m)}+u\partial_x\partial_y^2\phi_{b,(m)}-(\partial_x\phi_{b,(m)})\partial_y^2u-\partial_y^4\phi_{b,(m)}\\
&\quad=-\sum^{[\frac{m}{2}]}_{k=1}\binom{m}{k}(\partial_x^ku)\partial_y^2\phi_{b,(m-k+1)}-\sum^m_{k=[\frac{m}{2}]+1}\binom{m}{k}(\partial_y\phi_{b,(k)})\partial_x^{m-k+1}\partial_yu\\
&\qquad+\sum^{m-1}_{k=[\frac{m}{2}]+1}\binom{m}{k}\phi_{b,(k+1)}\partial_x^{m-k}\partial_y^2u-\sum^{[\frac{m}{2}]}_{k=0}\binom{m}{k}(\partial_x^kv)\partial_y^3\phi_{b,(m-k)},\\
&\phi_{b,(m)}|_{y=0,1}=0,\quad \partial_y\phi_{b,(m)}|_{y=0,1}=-\partial_y\phi_{s,(m)}|_{y=0,1}+\partial_x^m\mathcal{C}(0)-\partial_x^m\mathcal{C}(t),\\
&\phi_{b,(m)}|_{t=0}=0.
		\end{aligned}
		\right.
\end{equation}
Here $\mathcal{\vec R}_s=\{\mathcal{R}_{s,(m)}\}_{m\ge 0}$ in \eqref{eq:phis} is given by
\begin{equation}\label{def:Rm}
   \begin{aligned}
\mathcal{R}_{s,(m)}=&-\sum^{[\frac{m}{2}]}_{k=1}\binom{m}{k}(\partial_x^ku)\partial_x^{m-k+1}\partial_y^2\phi_0-\sum^m_{k=[\frac{m}{2}]+1}\binom{m}{k}(\partial_x^k\partial_y\phi_0)\partial_x^{m-k+1}\partial_yu\\
&+\sum^{m-1}_{k=[\frac{m}{2}]+1}\binom{m}{k}(\partial_x^{k+1}\phi_0)\partial_x^{m-k}\partial_y^2u-\sum^{[\frac{m}{2}]}_{k=0}\binom{m}{k}(\partial_x^kv)\partial_x^{m-k}\partial_y^3\phi_{0}.\\       
   \end{aligned} 
\end{equation}

\subsection{The estimate of $\vec\phi_s$}. This subsection is devoted to the estimate of $\vec\phi_s$, which is stated as follows.

\begin{lemma}\label{lem:phis}
Under the same assumption as given in Theorem \ref{thm:pri}, it holds that
\begin{multline}\label{est:phis}
\sup_{t\in[0,T]}||\partial_y^2\vec\phi_{s}||_{\rho,r+\sigma-\frac{1}{2}}^2+\beta\int^T_0||\partial_y^2\vec\phi_{s}||_{\rho,r+\sigma}^2dt
+\int^T_0||\partial_y^3\vec\phi_{s}||_{\rho,r+\sigma-\frac{1}{2}}^2dt\\
\leq C\int^T_0\mathcal{Y}_\rho dt+\frac{C\delta^{-2}C_*M}{\beta}.
\end{multline}
provided $\beta$ is sufficiently large. Moreover, we have
\begin{equation}\label{est:phis2}
\int^T_0||\vec\phi_{s}||_{\rho,r+\sigma}^2dt+\int^T_0|\partial_y\vec\phi_{s}|_{y=0,1}|_{\rho,r+\sigma}^2dt\leq \frac{C}{\beta}\int^T_0\mathcal{Y}_\rho dt+\frac{C\delta^{-2}C_*M}{\beta^2}.
\end{equation}
Recall that $\vec\phi_{s}=\{\phi_{s,(m)}\}_{m\ge0}$ where for each $m \ge 0$, $\phi_{s,(m)}$ is  the solution of \eqref{eq:phis}.
\end{lemma}

\begin{proof}
The proof of Lemma \ref{lem:phis} closely follows that of Lemma \ref{lem:yu}, utilizing the convexity of $\partial_y^2 u$ to obtain the desired result.

Observing that
\begin{align*}
   \inner{(\partial_x\phi_{s,(m)})\partial_y^2u ,\ \frac{\partial_y^2\phi_{s,(m)}}{\partial_y^2u }}_{L^2}&=\inner{\partial_x\phi_{s,(m)},\ \partial_y^2\phi_{s,(m)}}_{L^2}\\&=- \inner{\partial_x\partial_y\phi_{s,(m)},\ \partial_y\phi_{s,(m)}}_{L^2} =0
\end{align*}
implied by $\phi_{s,(m)}|_{y=0,1}=0$, we take the $L^2$-product with $\frac{\partial_y^2\phi_{s,(m)}}{\partial_y^2u}$ on both sides of \eqref{eq:phis} and use the boundary condition $\partial_y^2\phi_{s,(m)}|_{y=0,1}=0$ to derive
\begin{equation}\label{aq65}
\begin{aligned}
&\frac{1}{2}\frac{d}{dt}\left\|\frac{\partial_y^2\phi_{s,(m)}}{\sqrt{\partial_y^2u }}\right\|_{L^2}^2+\left\|\frac{\partial_y^3\phi _{s,(m)}}{\sqrt{\partial_y^2u }}\right\|_{L^2}^2\\
&=\inner{\mathcal{P}_m,\ \frac{\partial_y^2\phi_{s,(m)}}{\partial_y^2u }}_{L^2}-\frac{1}{2}\int_{\Omega}\frac{(\partial_y^2\phi_{s,(m)})^2\big[u\partial_x\partial_y^2u-(\partial_xu)\partial_y^2u\big]}{(\partial_y^2u )^2}dxdy\\
&\quad-\frac{1}{2}\int_{\Omega}\frac{(\partial_y^2\phi_{s,(m)})^2\partial_t\partial_y^2u }{(\partial_y^2u )^2}dxdy+\int_{\Omega}\frac{(\partial_y^3\phi_{s,(m)})(\partial_y^2\phi_{s,(m)})\partial_y^3u}{(\partial_y^2u )^2}dxdy\\
&\quad+\inner{\partial_x^m(f\partial_x\partial_yf+g\partial_y^2f),\ \frac{\partial_y^2\phi_{s,(m)}}{\partial_y^2u }}_{L^2}+\inner{\partial_x^m\partial_y^4\phi_0+\mathcal{R}_{s,(m)},\ \frac{\partial_y^2\phi_{s,(m)}}{\partial_y^2u }}_{L^2}\\
&\stackrel{\rm def}{=}P_{1m}+P_{2m}+P_{3m}+P_{4m}+P_{5m}+P_{6m}.
\end{aligned}
\end{equation}
Here $\mathcal{P}_m$ is given by
\begin{align*}
    \mathcal{P}_m\stackrel{\rm def}{=}&-\sum^{[\frac{m}{2}]}_{k=1}\binom{m}{k}(\partial_x^ku)\partial_y^2\phi_{s,(m-k+1)}-\sum^m_{k=[\frac{m}{2}]+1}\binom{m}{k}(\partial_y\phi_{s,(k)})\partial_x^{m-k+1}\partial_yu\\
&+\sum^{m-1}_{k=[\frac{m}{2}]+1}\binom{m}{k}\phi_{s,(k+1)}\partial_x^{m-k}\partial_y^2u-\sum^{[\frac{m}{2}]}_{k=0}\binom{m}{k}(\partial_x^kv)\partial_y^3\phi_{s,(m-k)}.
\end{align*}
We then multiply the above equality \eqref{aq65} by $L_{\rho,m,r+\sigma-\frac{1}{2}}^2$, use \eqref{factone} and Proposition \ref{prop:principle}, take summation over $m$ and then integrate on $[0,T]$ with respect to $t$ to get
\begin{multline}\label{est:phis3}
\sup_{t\in[0,T]}||\partial_y^2\vec\phi_{s}||_{\rho,r+\sigma-\frac{1}{2}}^2+\beta\int^T_0||\partial_y^2\vec\phi_{s}||_{\rho,r+\sigma}^2dt
\\+\int^T_0||\partial_y^3\vec\phi_{s}||_{\rho,r+\sigma-\frac{1}{2}}^2dt\leq C\delta^{-1}\sum_{i=1}^6\int^T_0 P_idt.
\end{multline}
Here $P_i$ is defined by
\begin{equation*}
    P_i\stackrel{\rm def}{=}\sum^{+\infty}_{m=0}L_{\rho,m,r+\sigma-\frac{1}{2}}^2|P_{im}|.
\end{equation*}
Now, we estimate $P_1-P_6$ one by one.

\underline{\it The $P_1$, $P_2$, $P_3$ and $P_4$ bounds}. The boundary condition $\phi_{s,(m)}|_{y=0,1}=0$ with Poincar\'e inequality implies that
\begin{equation}\label{ps}
  \forall\ m\ge 0,\quad \norm{\phi_{s,(m)}}_{L^2}\leq C\norm{\partial_y\phi_{s,(m)}}_{L^2}\leq C\norm{\partial_y^2\phi_{s,(m)}}_{L^2}.   
\end{equation}
Then applying the argument used in estimates \eqref{est:T23}-\eqref{est:T6} with minor modifications, we use \eqref{ps} to obtain that
\begin{equation}\label{est:P1-4}
\begin{aligned}
\int^T_0\sum_{i=1}^4P_i dt\leq &C\delta^{-2}C_*^\frac{3}{2}M^\frac{3}{2}\int^T_0\big(||\partial_y^2\vec\phi_{s}||_{\rho,r+\sigma}^2+||\partial_y^2\vec\phi_{s}||_{\rho,r+\sigma}||\partial_y^3\vec\phi_{s}||_{\rho,r+\sigma-\frac{1}{2}}\big)dt\\
&+C\delta^{-2}C_*^\frac{3}{2}M^\frac{3}{2}\int^T_0||\partial_y^2\vec\phi_{s}||_{\rho,r+\sigma}^\frac{1}{2}||\partial_y^3\vec\phi_{s}||_{\rho,r+\sigma-\frac{1}{2}}^\frac{3}{2}dt.  
\end{aligned}
\end{equation}

\underline{The $P_5$ bound}. By the convexity of $\partial_y^2u$ (Proposition \ref{prop:principle}), we write
\begin{equation*}
\begin{aligned}
&\sum^{+\infty}_{m=0}L_{\rho,m,r+\sigma-\frac{1}{2}}^2\left|\inner{\partial_x^m\inner{f\partial_x\partial_yf},\ \frac{\partial_y^2\phi_{s,(m)}}{\partial_y^2u}}_{L^2}\right|\\
&\leq C\delta^{-1}\sum^{+\infty}_{m=0}\sum^{[\frac{m}{2}]}_{k=0}L_{\rho,m,r+\sigma-\frac{1}{2}}^2\binom{m}{k}\norm{\partial_x^kf}_{L^{\infty}}\norm{\partial_x^{m-k+1}\partial_yf}_{L^2}\norm{\partial_y^2\phi_{s,(m)}}_{L^2}\\
&\quad+C\delta^{-1}\sum^{+\infty}_{m=0}\sum^{m}_{k=[\frac{m}{2}]+1}L_{\rho,m,r+\sigma-\frac{1}{2}}^2\binom{m}{k}\norm{\partial_x^kf}_{L^2}\norm{\partial_x^{m-k+1}\partial_yf}_{L^{\infty}}\norm{\partial_y^2\phi_{s,(m)}}_{L^2}.
\end{aligned}
\end{equation*}
Then we use Young's inequality \eqref{young} and the estimate that for any $1\leq\sigma\leq\frac{7}{6}$ and $r\ge 10$,
\begin{equation}\label{ineqkey}
    \left\{
    \begin{aligned}
&\binom{m}{k}\frac{L_{\rho,m,r+\sigma-\frac{1}{2}}^2}{L_{\rho,k+1,r-\sigma+\frac{3}{2}}L_{\rho,m-k+1,r-\sigma+\frac{5}{2}}L_{\rho,m,r+\sigma}}\leq \frac{C}{k+1},\quad \mathrm{if}\ 0\leq k\leq \big[\frac{m}{2}\big],\\
&\binom{m}{k}\frac{L_{\rho,m,r+\sigma-\frac{1}{2}}^2}{L_{\rho,k,r-\sigma+3}L_{\rho,m-k+2,r-\sigma+1}L_{\rho,m,r+\sigma}}\leq \frac{C}{m-k+1},\quad \mathrm{if}\  \big[\frac{m}{2}\big]+1\leq k\leq m,
    \end{aligned}
    \right.
\end{equation}
to deduce 
\begin{equation*}
\sum^{+\infty}_{m=0}L_{\rho,m,r+\sigma-\frac{1}{2}}^2\left|\inner{\partial_x^m\inner{f\partial_x\partial_yf},\ \frac{\partial_y^2\phi_{s,(m)}}{\partial_y^2u}}_{L^2}\right|\leq C\delta^{-1}\mathcal{X}_\rho^\frac{1}{2}\mathcal{Y}_\rho^\frac{1}{2}||\partial_y^2\vec\phi_{s}||_{\rho,r+\sigma},   
\end{equation*}
recalling $\mathcal{X}_\rho$ and $\mathcal{Y}_\rho$ are defined in \eqref{def:energytwo}.
Similarly, we have
\begin{equation*}
\sum^{+\infty}_{m=0}L_{\rho,m,r+\sigma-\frac{1}{2}}^2\left|\inner{\partial_x^m\inner{g\partial_y^2f},\ \frac{\partial_y^2\phi_{s,(m)}}{\partial_y^2u}}_{L^2}\right|\leq C\delta^{-1}\mathcal{X}_\rho^\frac{1}{2}\mathcal{Y}_\rho^\frac{1}{2}||\partial_y^2\vec\phi_{s}||_{\rho,r+\sigma}.    
\end{equation*}
Combining the two estimates above and using assumption \eqref{ass:pri} yield
\begin{equation}\label{est:P5}
    \int^T_0P_5dt\leq C\delta^{-1}\sqrt{C_*M}\int^T_0\mathcal{Y}_\rho^\frac{1}{2}||\partial_y^2\vec\phi_{s}||_{\rho,r+\sigma}dt.     
\end{equation}

\underline{The $P_6$ bound}. Recall $\phi$ and $\mathcal{R}_{s,(m)}$ are defined in \eqref{def:phi} and \eqref{def:Rm}, respectively. A direct computation with assumption \eqref{ass:pri}, \eqref{note:norm} and Lemma \ref{lem:principle} gives for any $t\in[0,T]$
\begin{align*}
&\sum^{+\infty}_{m=0}L_{\rho,m,r+\sigma-1}^2\norm{\partial_x^m\partial_y^4\phi_0+\mathcal{R}_{s,(m)}}_{L^2}^2\\
&\leq C\norm{\partial_y^3u_0}_{2\rho_0,r}^2+CC_*M \norm{(u_0,\partial_yu_0,\partial_y^2u_0)}_{2\rho_0,r}^2  \leq CM+CC_*M^2\leq CC_*^2M^2,
\end{align*}
the last inequality using $C_*\ge 1$ and $M\ge 1$. Consequently, we have, recalling $T=\beta^{-1}$,
\begin{equation}\label{est:P6}
\begin{aligned}
&\int^T_0P_6dt\leq \delta^{-1}\int^T_0 \sum^{+\infty}_{m=0}L_{\rho,m,r+\sigma-\frac{1}{2}}^2\norm{\partial_x^m\partial_y^4\phi_0+\mathcal{R}_{s,(m)}}_{L^2}\norm{\partial_y^2\phi_{s,(m)}}_{L^2}dt\\
&\leq C\delta^{-1}C_*M\int^T_0\norm{\partial_y^2\vec\phi_s}_{\rho,r+\sigma}ds\leq \frac{C\delta^{-1}C_*M}{\beta}+C\delta^{-1}C_*M\int^T_0\norm{\partial_y^2\vec\phi_s}_{\rho,r+\sigma}^2ds.
\end{aligned}
\end{equation}
Finally, substituting \eqref{est:P1-4}, \eqref{est:P5} and \eqref{est:P6} into \eqref{est:phis3} and using Cauchy inequality yield
\begin{align*}
&\sup_{t\in[0,T]}||\partial_y^2\vec\phi_{s}||_{\rho,r+\sigma-\frac{1}{2}}^2+\beta\int^T_0||\partial_y^2\vec\phi_{s}||_{\rho,r+\sigma}^2dt+\int^T_0||\partial_y^3\vec\phi_{s}||_{\rho,r+\sigma-\frac{1}{2}}^2dt\\
&\leq C\delta^{-12}C_*^6M^6\int^T_0||\partial_y^2\vec\phi_{s}||_{\rho,r+\sigma}^2dt+\frac{1}{2}\int^T_0||\partial_y^3\vec\phi_{s}||_{\rho,r+\sigma-\frac{1}{2}}^2dt\\
&\quad+C\int^T_0\mathcal{Y}_\rho dt+\frac{C\delta^{-2}C_*M}{\beta}.
\end{align*}
Then assertion \eqref{est:phis} follows by choosing $\beta\ge 2C\delta^{-12}C_*^6M^6$ in the above inequality. Furthermore, estimate \eqref{ps} with Sobolev embedding inequality gives
\begin{align*}
||\vec\phi_{s}||_{\rho,r+\sigma}+|\partial_y\vec\phi_{s}|_{y=0,1}|_{\rho,r+\sigma}\leq C||\partial_y^2\vec\phi_{s}||_{\rho,r+\sigma}.
\end{align*}
Thus, assertion \eqref{est:phis2} holds by combining the estimate above and assertion \eqref{est:phis}. Lemma \ref{lem:phis} is completed.
\end{proof}

\subsection{The estimate of $\vec{\phi}_{b}$}
To deal with $\vec{\phi}_b$, we use the following decomposition:
\begin{equation}\label{de:phib}
 \vec\phi_b=\{\partial_x^m\phi_{H}\}_{m\ge 0}+\vec\phi_T+\vec\phi_R,
\end{equation}
where $\vec\phi_T\stackrel{\rm def}{=}\{\phi_{T,(m)}\}_{m\ge0}$ and $\vec\phi_R\stackrel{\rm def}{=}\{\phi_{R,(m)}\}_{m\ge0}$.
 We remark that the condition $1\leq \sigma\leq\frac{7}{6}$ can be relaxed to $1\leq\sigma\leq\frac{3}{2}$ in the process of estimating $\vec{\phi}_{b}$. This implies that the key obstacle preventing us from achieving the Gevrey index $\frac{3}{2}$ lies elsewhere. 
 
 We now present the definitions and estimates of $\phi_{H}$, $\vec\phi_T$ and $\vec\phi_R$ one by one and then complete the estimate of $\vec{\phi}_b$.

\subsubsection{The estimate of $\phi_{H}$: Heat equation} 
We define
\begin{equation*}
    \phi_{H}=\phi^0_{H}+\phi^1_{H},
\end{equation*}
where $\phi^i_{H}$ $(i=0,1)$ satisfies the following heat equation:
\begin{equation}\label{eq:phiH}
		\left\{
		\begin{aligned}
&(\partial_t-\partial_y^2)\partial_y^2\phi^i_{H}=0,\quad (x,y)\in \mathbb{T}\times I_i,\\
&\phi^i_{H}|_{y=i}=0,\quad \partial_y\phi^i_{H}|_{y=i}=h^i(t,x),\\
&\phi^i_{H}|_{t=0}=0,
		\end{aligned}
		\right.
\end{equation}
where $I_0=(0,+\infty)$ and $I_1=(-\infty,1)$. Here $(h^0,h^1)$ is a given boundary data which is defined later and satisfies $(h^0(t),h^1(t))=0$ for $t\leq 0$ and $t\ge T$. 

 We will only provide the estimation procedure for $\phi^0_{H}$, as the corresponding analysis for $\phi^1_{H}$ is nearly identical. The details for the latter are therefore left to the reader.

At first, we give zero extension of $\phi^0_{H}$ with $t\leq 0$ such that we can take Fourier transform in $t$. Recall $\phi^0_{H,m}=L_{\rho,m,r}\partial_x^m\phi^0_{H}$. Let $\widehat{\phi^0_{H,m}}=\widehat{\phi^0_{H,m}}(\xi,x,y)$ be the Fourier transform of $\phi^0_{H,m}$ on $t$. Then $\widehat{\phi^0_{H,m}}$ satisfies the ODE:
\begin{equation*}
		\left\{
		\begin{aligned}
&(i\xi+\beta(m+1)-\partial_y^2)\partial_y^2\widehat{\phi^0_{H,m}}=0,\quad (x,y)\in \mathbb{T}\times(0,+\infty),\\
&\widehat{\phi^0_{H,m}}|_{y=0}=0,\quad \partial_y\widehat{\phi^0_{H,m}}|_{y=0}=\widehat{h_m^0}.\\
		\end{aligned}
		\right.
\end{equation*}
Assuming the decay of $\phi^0_{H}$ and $\partial_y\phi^0_{H}$, we obtain the formula:
\begin{align}\label{express:phiH0}
    \widehat{\phi^0_{H,m}}(\xi,x,y)=-\frac{\widehat{h_m^0}}{\sqrt{i\xi+\beta(m+1)}}e^{-y\sqrt{i\xi+\beta(m+1)}}+\frac{\widehat{h_m^0}}{\sqrt{i\xi+\beta(m+1)}},\quad y>0.
\end{align}
where the square root $\sqrt{i\xi+\beta(m+1)}$ is taken so that the real part is positive, and it follows that
\begin{align}\label{6eq9}
    \sqrt{\beta(m+1)}\leq \mathbf{Re}(\sqrt{i\xi+\beta(m+1)})\leq |\sqrt{i\xi+\beta(m+1)}|\leq 2\mathbf{Re}(\sqrt{i\xi+\beta(m+1)}).
\end{align}
 In view of \eqref{express:phiH0}, it is easy to calculate that
\begin{align}\label{6eq10}
    \partial_y \widehat{\phi^0_{H,m}}(\xi,x,y)=\widehat{h_m^0}e^{-y\sqrt{i\xi+\beta(m+1)}},
\end{align}
\begin{align}\label{6eq11}
      \partial_y^2 \widehat{\phi^0_{H,m}}(\xi,x,y)=-\sqrt{i\xi+\beta(m+1)}\widehat{h_m^0}e^{-y\sqrt{i\xi+\beta(m+1)}}.  
\end{align}
The formula \eqref{6eq10} will be used in estimating velocity and \eqref{6eq11} will be used in estimating vorticity. With the same process above, we get the formula for $\widehat{\phi^1_{H,m}}(\xi,x,y)$:
\begin{align}\label{6eq12}
    \widehat{\phi^1_{H,m}}(\xi,x,y)=\frac{\widehat{h_m^1}}{\sqrt{i\xi+\beta(m+1)}}e^{(y-1)\sqrt{i\xi+\beta(m+1)}}-\frac{\widehat{h_m^1}}{\sqrt{i\xi+\beta(m+1)}},\quad y<1.
\end{align}
Correspondingly,
\begin{align}\label{6eq13}
    \partial_y \widehat{\phi^1_{H,m}}(\xi,x,y)=\widehat{h_m^1}e^{(y-1)\sqrt{i\xi+\beta(m+1)}},
\end{align}
\begin{align}\label{6eq14}
      \partial_y^2 \widehat{\phi^1_{H,m}}(\xi,x,y)=\sqrt{i\xi+\beta(m+1)}\widehat{h_m^1}e^{(y-1)\sqrt{i\xi+\beta(m+1)}}. 
\end{align}

\begin{lemma}\label{lem12}
 Let $\phi^i_{H}$ $(i=0,1)$ be the solution of \eqref{eq:phiH}. It holds that for any given $\ell\ge 0$ and $m\ge 0$,
 \begin{align}\label{6eq17}
     \norm{\widehat{\phi^i_{H,m}}}_{L^2_{\xi,y,i}}+ \norm{\widehat{\phi^i_{H,m}}|_{y=1-i}}_{L^2_{\xi}}\leq \frac{C}{\beta^{\frac{1}{2}}(m+1)^{\frac{1}{2}}}\norm{\widehat{h^i_m}}_{L^2_{\xi}},
 \end{align}
  \begin{align}\label{6eq188}
\norm{(\varphi^i)^\ell\partial_y\widehat{\phi^i_{H,m}}}_{L^2_{\xi,y,i}}\leq  \frac{C}{\beta^{\frac{2\ell+1}{4}}(m+1)^{\frac{2\ell+1}{4}}}\norm{\widehat{h^i_m}}_{L^2_{\xi}},
 \end{align}
  \begin{align}\label{6eq20}
 \norm{(\varphi^i)^{\ell+\frac{1}{2}}\partial_y^2\widehat{\phi^i_{H,m}}}_{L^2_{\xi,y,i}}+ \norm{(\varphi^i)^{\ell+\frac{3}{2}}\partial_y^3\widehat{\phi^i_{H,m}}}_{L^2_{\xi,y,i}}\leq  \frac{C}{\beta^{\frac{\ell}{2}}(m+1)^{\frac{\ell}{2}}}\norm{\widehat{h^i_m}}_{L^2_{\xi}},   
\end{align}
    \begin{align}\label{6eq22}
 \norm{\partial_y\widehat{\phi^i_{H,m}}|_{y=1-i}}_{L^2_{\xi}}\leq  \frac{C}{\beta^{10}(m+1)^{10}}\norm{\widehat{h^i_m}}_{L^2_{\xi}},   
 \end{align}
 where $L^2_{\xi,y,i}=L^2_{\xi,y}(\mathbb{R}\times I_i)$ for $i=0,1$ and $\varphi^i=\varphi^i(y)$ $(i=0,1)$ is given by
 \begin{equation}\label{def:varphi}
      \varphi^0(y)=y,\quad \varphi^1(y)=1-y.  
 \end{equation}
\end{lemma}
\begin{proof}
 The proof of \eqref{6eq17}-\eqref{6eq20} follows from a direct calculation via \eqref{express:phiH0}-\eqref{6eq14}; hence, we omit the details. As for the pointwise estimate \eqref{6eq22}, all boundary terms taken at $y=1-i$ contain an exponential factor $e^{-\sqrt{\beta(m+1)}}$ in view of \eqref{6eq9}, which allows to gain an arbitrary number of powers of $\beta(m+1)$, which explains the factor $\beta^{-10}(m+1)^{-10}$.  
\end{proof}

As a direct corollary of Lemma \ref{lem12}, we have the following estimate for $\phi^i_H$ without additional difficulty.
\begin{lemma}\label{lem:phiH}
 Let $\phi^i_{H}$ $(i=0,1)$ be the solution of \eqref{eq:phiH}. It holds that for any given $\theta\in\mathbb{R}$ and $\ell\ge 0$,
\begin{equation*}
\int^T_0\big(\norm{\phi^i_{H}}_{\rho,\theta,I_i}^2+|\phi^i_{H}|_{y=1-i}|_{\rho,\theta}^2\big)dt\leq \frac{C}{\beta}\int^T_0|h^i|_{\rho,\theta-\frac{1}{2}}^2dt,    
\end{equation*}
\begin{equation*}
\int^T_0\norm{(\varphi^i)^{\ell}\partial_y\phi^i_{H}}_{\rho,\theta,I_i}^2dt\leq \frac{C}{\beta^{\frac{2\ell+1}{2}}}\int^T_0|h^i|_{\rho,\theta-\frac{2\ell+1}{4}}^2dt,    
\end{equation*}
\begin{equation*}
\int^T_0\big(\norm{(\varphi^i)^{\ell+\frac{1}{2}}\partial_y^2\phi^i_{H}}_{\rho,\theta,I_i}^2+\norm{(\varphi^i)^{\ell+\frac{3}{2}}\partial_y^3\phi^i_{H}}_{\rho,\theta,I_i}^2\big)dt\leq \frac{C}{\beta^{\ell}}\int^T_0|h^i|_{\rho,\theta-\frac{\ell}{2}}^2dt,  
\end{equation*}
\begin{equation*}
 \int^T_0|\partial_y\phi^i_{H}|_{y=1-i}|_{\rho,\theta}^2dt\leq \frac{C}{\beta^{20}} \int^T_0|h^i|_{\rho,\theta-10}^2dt. 
\end{equation*}
Recall that $\norm{\cdot}_{\rho,\theta,I_i}$ and $\abs{\cdot}_{\rho,\theta}$ are defined in Subsection \ref{subsec:norm} and $\varphi^i$ is given by \eqref{def:varphi}.
\end{lemma}

\subsubsection{The estimate of $\vec\phi_{T}$: Vorticity transport estimate} We begin by decomposing $\vec\phi_{T}$ as 
\begin{align*}
    \vec\phi_{T}=\{\phi_{T,(m)}^0\}_{m\ge0}+\{\phi_{T,(m)}^1\}_{m\ge 0}\stackrel{\rm def}{=}\vec\phi_T^0+\vec\phi_T^1,
\end{align*}
where for given $m\ge0$, $\phi_{T,(m)}^i$ $(i=0,1)$ satisfies the following equation in the domain $\mathbb{T}\times I_i$:
\begin{equation}\label{eq:phiT}
		\left\{
		\begin{aligned}
&(\partial_t+u\partial_x-\partial_y^2)\partial_y^2\phi^i_{T,(m)}=-\sum^{[\frac{m}{2}]}_{k=1}\binom{m}{k}(\partial_x^ku)\partial_y^2\phi^i_{T,(m-k+1)}\\
&\qquad-\sum^{[\frac{m}{2}]}_{k=0}\binom{m}{k}(\partial_x^{k}v)\partial_y^3\phi^i_{T,(m-k)}-\sum^m_{k=[\frac{m}{2}]+1}\binom{m}{k}(\partial_y\phi_{T,(k)}^i)\partial_x^{m-k+1}\partial_yu+\mathcal{R}^i_{T,(m)},\\
&\phi^i_{T,(m)}|_{y=i}=0,\quad \partial_y^2\phi^i_{T,(m)}|_{y=i}=0,\\
&\phi^i_{T,(m)}|_{t=0}=0.
		\end{aligned}
		\right.
\end{equation}
Here $\mathcal{\vec R}^i_T=\{\mathcal{R}^i_{T,(m)}\}_{m\ge 0}$ is given by
\begin{equation}\label{def:RTH}
    \begin{aligned}
\mathcal{R}^i_{T,(m)}\stackrel{\rm def}{=} &-\sum^{[\frac{m}{2}]}_{k=0}\binom{m}{k}(\partial_x^ku)\partial_x^{m-k+1}\partial_y^2\phi^i_{H} -\sum^{[\frac{m}{2}]}_{k=0}\binom{m}{k}(\partial_x^{k}v)\partial_x^{m-k}\partial_y^3\phi^i_{H}\\
&-\sum^m_{k=[\frac{m}{2}]+1}\binom{m}{k}(\partial_x^k\partial_y\phi_{H}^i)\partial_x^{m-k+1}\partial_yu.
    \end{aligned}
\end{equation}
We emphasize that we extend $(u,v)$ to $y\in\mathbb{R}$ by zero which means that $(u,v)=(0,0)$ when $y<0$ and $y>1$.

Before estimating $\vec\phi^i_{T}$, we first give the estimate of $\mathcal{\vec R}^i_{T}$ as follows.
\begin{lemma}\label{lem:RiTH}
 Under the same assumption as given in Theorem \ref{thm:pri} with $1\leq \sigma\leq \frac{7}{6}$ relaxed to $1\leq\sigma\leq\frac{3}{2}$, it holds that for given $\theta\in \mathbb{R}$, $i=0,1$ and $j=0,1,2$,
 \begin{equation}\label{est:RTH}
\int^T_0\norm{(\varphi^i)^j\mathcal{\vec R}^i_{T}}_{\rho,\theta,I_i}^2 dt\leq \frac{CC_*M}{\beta^{j+\frac{1}{2}}}\int^T_0|h^i|_{\rho,\theta+\sigma-\frac{2j+1}{4}}^2dt,
 \end{equation}
provided $\beta$ is sufficiently large. Recall that $\varphi^i$  and $\mathcal{\vec R}^i_{T}$ are defined by  \eqref{def:varphi} and \eqref{def:RTH}, respectively.
\end{lemma}
\begin{proof}
We only prove assertion \eqref{est:RTH} for the case $i=0$, as the case $i=1$ is analogous and the details are omitted. In the following discussion, we note that $y\in I_0=(0,+\infty)$ in the case $i=0$ and $(u,v)=(0,0)$ for $y\in (1,+\infty)$.

At first, in view of the definition \eqref{def2} of $L_{\rho,m,r}$, it is easy to calculate that for any given $r\ge 10$, $1\leq \sigma\leq \frac{3}{2}$ and $\theta\in\mathbb{R}$,
 \begin{equation}\label{ineq6}
     \left\{
     \begin{aligned}
&\binom{m}{k}\frac{L_{\rho,m,\theta}}{L_{\rho,k+1,r-\sigma}L_{\rho,m-k+1,\theta+\sigma}}\leq \frac{C}{k+1},\quad &&\mathrm{if}\ 0\leq k\leq \big[\frac{m}{2}\big],\\&\binom{m}{k}\frac{L_{\rho,m,\theta}}{L_{\rho,k+2,r-\sigma}L_{\rho,m-k,\theta}}\leq \frac{C}{k+1},\quad &&\mathrm{if}\ 0\leq k\leq \big[\frac{m}{2}\big],\\
&\binom{m}{k}\frac{L_{\rho,m,\theta}}{L_{\rho,k,\theta}L_{\rho,m-k+2,r-\sigma}}\leq \frac{C}{m-k+1},\quad &&\mathrm{if}\  \big[\frac{m}{2}\big]+1\leq k\leq m.
     \end{aligned}
     \right.
 \end{equation}
 On the other hand, observing that
 \begin{equation*}
     \norm{y^{-1}u}_{L_y^\infty}\leq \norm{\partial_yu}_{L_y^\infty}\quad \mathrm{and}\quad\norm{y^{-2}v}_{L_y^\infty}\leq \frac{1}{2}\norm{\partial_x\partial_yu}_{L_y^\infty},
 \end{equation*}
we use assumption \eqref{ass:pri} and Lemma \ref{lem:principle} to conclude that
 \begin{equation}\label{y-1u}
     \sup_{t\in[0,T]}\sum^{+\infty}_{m=0}L_{\rho,m+1,r-\sigma}^2\big(\norm{y^{-1}\partial_x^mu}_{L^{\infty}}^2+\norm{\partial_x^m\partial_yu}_{L^{\infty}}^2\big)\leq CC_*M,
 \end{equation}
 and
  \begin{equation}\label{y-2v}
     \sup_{t\in[0,T]}\sum^{+\infty}_{m=0}L_{\rho,m+2,r-\sigma}^2\norm{y^{-2}\partial_x^mv}_{L^{\infty}}^2\leq CC_*M.
 \end{equation}
 Then, Young's inequality \eqref{young} with \eqref{ineq6}-\eqref{y-2v} and Lemma \ref{lem:phiH} yields
\begin{align*}
&\int^T_0\sum^{+\infty}_{m=0}\Bigg(\sum^{[\frac{m}{2}]}_{k=0}\binom{m}{k}L_{\rho, m,\theta}\norm{y^j(\partial_x^ku)\partial_x^{m-k+1}\partial_y^2\phi^0_{H}}_{L^2}\Bigg)^2dt\\
&\leq C\int^T_0\Bigg(\sum^{+\infty}_{m=0}L_{\rho,m+1,r-\sigma}^2\norm{y^{-1}\partial_x^mu}_{L^{\infty}}^2\Bigg)\Bigg(\sum^{+\infty}_{m=0}L_{\rho,m,\theta+\sigma}^2\norm{y^{j+1}\partial_x^{m}\partial_y^2\phi^0_{H}}_{L^2}^2\Bigg)dt\\
&\leq CC_*M\int^T_0\norm{y^{j+1}\partial_y^2\phi^0_{H}}_{\rho,\theta+\sigma,I_0}^2dt\leq \frac{CC_*M}{\beta^{j+\frac{1}{2}}}\int^T_0|h^0|_{\rho,\theta+\sigma-\frac{2j+1}{4}}^2dt,
\end{align*}
\begin{align*}
&\int^T_0\sum^{+\infty}_{m=0}\inner{\sum^{[\frac{m}{2}]}_{k=0}\binom{m}{k}L_{\rho, m,\theta}\norm{y^j(\partial_x^kv)\partial_x^{m-k}\partial_y^3\phi^0_{H}}_{L^2}}^2dt\\
&\leq C\int^T_0\inner{\sum^{+\infty}_{m=0}L_{\rho,m+2,r-\sigma}^2\norm{y^{-2}\partial_x^mv}_{L^{\infty}}^2}\inner{\sum^{+\infty}_{m=0}L_{\rho,m,\theta}^2\norm{y^{j+2}\partial_x^{m}\partial_y^3\phi^0_{H}}_{L^2}^2}dt\\
&\leq CC_*M\int^T_0\norm{y^{j+2}\partial_y^3\phi^0_{H}}_{\rho,\theta,I_0}^2dt\leq \frac{CC_*M}{\beta^{j+\frac{1}{2}}}\int^T_0|h^0|_{\rho,\theta-\frac{2j+1}{4}}^2dt,
\end{align*}
and
\begin{align*}
&\int^T_0\sum^{+\infty}_{m=0}\Bigg(\sum^{m}_{k=[\frac{m}{2}]+1}\binom{m}{k}L_{\rho, m,\theta}\norm{y^j(\partial_x^k\partial_y\phi^0_{H})\partial_x^{m-k+1}\partial_yu}_{L^2}\Bigg)^2dt\\
&\leq C\int^T_0\Bigg(\sum^{+\infty}_{m=0}L_{\rho,m+1,r-\sigma}^2\norm{\partial_x^m\partial_yu}_{L^{\infty}}^2\Bigg)\Bigg(\sum^{+\infty}_{m=0}L_{\rho,m,\theta}^2\norm{y^{j}\partial_x^{m}\partial_y\phi^0_{H}}_{L^2}^2\Bigg)dt\\
&\leq CC_*M\int^T_0\norm{y^{j}\partial_y\phi^0_{H}}_{\rho,\theta,I_0}^2dt\leq \frac{CC_*M}{\beta^{j+\frac{1}{2}}}\int^T_0|h^0|_{\rho,\theta-\frac{2j+1}{4}}^2dt.
\end{align*}
Combining these estimates above, we have, recalling $\mathcal{R}^0_{T,(m)}$ is given by \eqref{def:RTH},
\begin{equation*}
\int^T_0\norm{y^j\mathcal{\vec R}^0_{T}}_{\rho,\theta,I_0}^2   \leq \frac{CC_*M}{\beta^{j+\frac{1}{2}}}dt\int^T_0|h^0|_{\rho,\theta+\sigma-\frac{2j+1}{4}}^2dt.
\end{equation*}
The case $i=0$ of assertion \eqref{est:RTH} holds, thus completing the proof of Lemma \ref{lem:RiTH}.
\end{proof}

We are in the position to give the estimate of $\vec\phi^i_{T}$.
\begin{lemma}\label{lem:phiT}
Under the same assumption as given in Theorem \ref{thm:pri} with $1\leq \sigma\leq \frac{7}{6}$ relaxed to $1\leq\sigma\leq\frac{3}{2}$, it holds that for given $\theta\in \mathbb{R}$, $i=0,1$ and $j=0,1,2$,
\begin{multline}\label{6eq42}
\sup_{t\in[0,T]}||(\varphi^i)^{j}\partial_y^2\vec\phi^i_{T}||_{\rho,\theta,I_i}^2+\beta\int^T_0||(\varphi^i)^{j}\partial_y^2\vec\phi^i_{T}||_{\rho,\theta+\frac{1}{2},I_i}^2dt\\
+\int^T_0||(\varphi^i)^{j}\partial_y^3\vec\phi^i_{T}||_{\rho,\theta,I_i}^2dt
\leq \frac{CC_*M}{\beta^{\frac{1}{2}}}\int^T_0|h^i|_{\rho,\theta+\sigma-\frac{2j+3}{4}}^2dt, 
\end{multline}
and
\begin{equation}\label{6eq42add}
\begin{aligned}
\sup_{t\in[0,T]}&||(\varphi^i)^{j}\partial_x\partial_y^2\vec\phi^i_{T}||_{\rho,\theta-\sigma,I_i}^2+\beta\int^T_0||(\varphi^i)^{j}\partial_x\partial_y^2\vec\phi^i_{T}||_{\rho,\theta-\sigma+\frac{1}{2},I_i}^2dt\\
&\quad+\int^T_0||(\varphi^i)^{j}\partial_x\partial_y^3\vec\phi^i_{T}||_{\rho,\theta-\sigma,I_i}^2dt\leq \frac{CC_*M}{\beta^{\frac{1}{2}}}\int^T_0|h^i|_{\rho,\theta+\sigma-\frac{2j+3}{4}}^2dt,
\end{aligned}
\end{equation}
provided $\beta$ is sufficiently large. Recall that $\vec\phi^i_{T}=\{\phi^i_{T,(m)}\}_{m\ge0}$ and $\varphi^i$ is defined in \eqref{def:varphi}. Here, for each $m \ge 0$, $\phi^i_{T,(m)}$ is  the solution of \eqref{eq:phiT}.
\end{lemma}
\begin{proof}
We prove only assertion \eqref{6eq42}, as assertion \eqref{6eq42add} can be obtained similarly. Furthermore, it suffices to consider the case $i=0$ in \eqref{6eq42}. In this case,  we note that $y\in I_0=(0,+\infty)$ and $(u,v)=(0,0)$ for $y\in (1,+\infty)$.

For $j=0,1,2$, we take the $L^2$-product with $y^{2j}\partial_y^2\phi^0_{T,(m)}$ on both sides of \eqref{eq:phiT}, multiply by $L_{\rho,m,\theta}^2$, use the fact \eqref{factone}, take summation over $m\in\mathbb{Z}_+$ and then integrate over $[0,T]$ with respect to time to derive that
\begin{equation}\label{6eq44}
    \begin{aligned}
&\sup_{t\in[0,T]}||y^j\partial_y^2\vec\phi^0_{T}||_{\rho,\theta,I_0}^2+\beta\int^T_0||y^j\partial_y^2\vec\phi^0_{T}||_{\rho,\theta+\frac{1}{2},I_0}^2dt\\&\quad+\int^T_0||y^j\partial_y^3\vec\phi^0_{T}||_{\rho,\theta,I_0}^2dt-j(2j-1)\int^T_0||y^{j-1}\partial_y^2\vec\phi^0_{T}||_{\rho,\theta,I_0}^2dt\leq \sum_{k=1}^5N_{k,j},        
    \end{aligned}
\end{equation}
where
	\begin{equation*}
		\left\{
		\begin{aligned}
N_{1,j}&=\int^T_0\sum^{+\infty}_{m=0}L^2_{\rho,m,\theta}\norm{\partial_xu}_{L^{\infty}}\norm{y^j\partial_y^2\phi^0_{T,(m)}}_{L^2}^2dt,\\
N_{2,j}&=\int^T_0\sum_{m=0}^{+\infty}L^2_{\rho,m,\theta}\norm{y^j\mathcal{R}^0_{T,(m)}}_{L^2}\norm{y^j\partial_y^2\phi^0_{T,(m)}}_{L^2}dt,\\
N_{3,j}&=\int^T_0\sum^{+\infty}_{m=0}\sum^{[\frac{m}{2}]}_{k=1}\binom{m}{k}L^2_{\rho,m,\theta}\norm{y^j(\partial_x^ku)\partial_y^2\phi^0_{T,(m-k+1)}}_{L^2}\norm{y^j\partial_y^2\phi^0_{T,(m)}}_{L^2}dt,\\
N_{4,j}&=\int^T_0\sum^{+\infty}_{m=0}\sum^{m}_{k=[\frac{m}{2}]+1}\binom{m}{k}L_{\rho,m,\theta}^2\norm{y^j(\partial_y\phi^0_{T,(k)})\partial_x^{m-k+1}\partial_yu}_{L^2}\norm{y^j\partial_y^2\phi^0_{T,(m)}}_{L^2}dt,\\
N_{5,j}&=\int^T_0\sum^{+\infty}_{m=0}\sum^{[\frac{m}{2}]}_{k=0}\binom{m}{k}L^2_{\rho,m,\theta}\norm{y^j(\partial_x^{k}v)\partial_y^3\phi^0_{T,(m-k)}}_{L^2}\norm{y^j\partial_y^2\phi^0_{T,(m)}}_{L^2}dt.
		\end{aligned}
		\right.
	\end{equation*} 
We now estimate the terms $N_{1,j}-N_{5,j}$ one by one.

\underline{\it The $N_{1,j}$ bound}. It follows from assumption \eqref{ass:pri} that
\begin{equation*}
  N_{1,j}\leq   C\sqrt{C_*M}\int^T_0||y^j\partial_y^2\vec\phi^0_{T}||_{\rho,\theta,I_0}^2dt\\
\leq C\sqrt{C_*M}\int^T_0||y^j\partial_y^2\vec\phi^0_{T}||_{\rho,\theta+\frac{1}{2},I_0}^2dt.    
\end{equation*}

\underline{\it The $N_{2,j}$ bound}. Applying Lemma \ref{lem:RiTH}, we have
 \begin{multline*}
     N_{2,j}\leq C\int^T_0\norm{y^j\mathcal{\vec R}^0_{T}}_{\rho,\theta-\frac{1}{2},I_0}^2  dt+C\int^T_0||y^j\partial_y^2\vec\phi^0_{T}||_{\rho,\theta+\frac{1}{2},I_0}^2dt\\
     \leq \frac{CC_*M}{\beta^{j+\frac{1}{2}}}\int^T_0|h^0|_{\rho,\theta+\sigma-\frac{2j+3}{4}}^2dt+C\int^T_0||y^j\partial_y^2\vec\phi^0_{T}||_{\rho,\theta+\frac{1}{2},I_0}^2dt.
 \end{multline*}

\underline{\it The $N_{3,j}$ bound}. Repeating a similar argument in \eqref{est:I1}, we use Young's inequality \eqref{young} and the fact that for any given $\theta\in\mathbb{R}$,
\begin{equation*}
    \binom{m}{k}\frac{L_{\rho,m,\theta}^2}{L_{\rho,k+1,r}L_{\rho,m-k+1,\theta+\frac{1}{2}}L_{\rho,m,\theta+\frac{1}{2}}}\leq \frac{C}{k+1},\quad \mathrm{if}\  \ 1\leq k\leq \big[\frac{m}{2}\big],
\end{equation*}
to conclude that
\begin{multline*}
    N_{3,j}\leq C\int^T_0\Big(\sum_{m=0}^{+\infty}L_{\rho,m+1,r}^2\norm{\partial_x^mu}_{L^\infty}^2\Big)^\frac{1}{2}||y^j\partial_y^2\vec\phi^0_{T}||_{\rho,\theta+\frac{1}{2},I_0}^2dt\\
    \leq C\sqrt{C_*M}\int^T_0||y^j\partial_y^2\vec\phi^0_{T}||_{\rho,\theta+\frac{1}{2},I_0}^2dt, 
\end{multline*}
the last inequality using assumption \eqref{ass:pri}.

\underline{\it The $N_{4,j}$ bound}. Note that
\begin{equation*}
\binom{m}{k}\frac{L_{\rho,m,\theta}^2}{L_{\rho,k,\theta-\frac{1}{2}}L_{\rho,m-k+2,r-\sigma}L_{\rho,m,\theta+\frac{1}{2}}}\leq \frac{C}{m-k+1},\quad \mathrm{if}\ \ \big[\frac{m}{2}\big]+1\leq k\leq m.   
\end{equation*}
We use Young's inequality \eqref{young} along with assumption \eqref{ass:pri} and Lemma \ref{lem:principle} to calculate
\begin{equation*}
    \begin{aligned}
N_{4,0}&\leq C\int^T_0\Big(\sum_{m=0}^{+\infty}L_{\rho,m+1,r-\sigma}^2\norm{\partial_x^m\partial_yu}_{L^\infty}^2\Big)^\frac{1}{2}||\partial_y\vec\phi^0_{T}||_{\rho,\theta-\frac{1}{2},I_0}||\partial_y^2\vec\phi^0_{T}||_{\rho,\theta+\frac{1}{2},I_0}dt\\
&\leq C\sqrt{C_*M}\int^T_0||y\partial_y^2\vec\phi^0_{T}||_{\rho,\theta-\frac{1}{2},I_0}||\partial_y^2\vec\phi^0_{T}||_{\rho,\theta+\frac{1}{2},I_0}dt\\
&\leq C\int^T_0||y\partial_y^2\vec\phi^0_{T}||_{\rho,\theta-\frac{1}{2},I_0}^2dt+CC_*M\int^T_0||\partial_y^2\vec\phi^0_{T}||_{\rho,\theta+\frac{1}{2},I_0}^2dt,
    \end{aligned}
\end{equation*}
the second inequality using Hardy inequality. For $j=1,2$, a similar computation with the fact $u=0$ for $y\in (1,+\infty)$ gives
\begin{equation*}
    \begin{aligned}
N_{4,j}&\leq  \int \Big[\sum_{m=0}^{+\infty}L_{\rho,m+1,r-\sigma}^2\norm{\partial_x^m\partial_yu}_{L^\infty}^2\Big]^\frac{1}{2}||y^{j-1}\partial_y\vec\phi^0_{T}||_{\rho,\theta-\frac{1}{2},I_0}||y^j\partial_y^2\vec\phi^0_{T}||_{\rho,\theta+\frac{1}{2},I_0}dt\\
&\leq C\sqrt{C_*M}\int^T_0||y^j\partial_y^2\vec\phi^0_{T}||_{\rho,\theta+\frac{1}{2},I_0}^2dt.
    \end{aligned}
\end{equation*}

\underline{\it The $N_{5,j}$ bound}. By the estimate that
\begin{equation*}
    \binom{m}{k}\frac{L_{\rho,m,\theta}^2}{L_{\rho,k+2,r}L_{\rho,m-k,\theta}L_{\rho,m,\theta+\frac{1}{2}}}\leq \frac{C}{k+1},\quad \mathrm{if}\ \ 0\leq k\leq \big[\frac{m}{2}\big],
\end{equation*}
we use Young's inequality \eqref{young} and assumption \eqref{ass:pri} to conclude
\begin{equation*}
    \begin{aligned}
N_{5,j}&\leq C\int^T_0\Big(\sum_{m=0}^{+\infty}L_{\rho,m+2,r}^2\norm{\partial_x^mv}_{L^\infty}^2\Big)^\frac{1}{2}||y^{j}\partial_y^3\vec\phi^0_{T}||_{\rho,\theta,I_0}||y^j\partial_y^2\vec\phi^0_{T}||_{\rho,\theta+\frac{1}{2},I_0}dt\\
&\leq C\sqrt{C_*M}\int^T_0||y^j\partial_y^3\vec\phi^0_{T}||_{\rho,\theta,I_0}||y^j\partial_y^2\vec\phi^0_{T}||_{\rho,\theta+\frac{1}{2},I_0}dt\\
&\leq CC_*M\int^T_0||y^j\partial_y^2\vec\phi^0_{T}||_{\rho,\theta+\frac{1}{2},I_0}^2dt+\frac{1}{4}\int^T_0||y^j\partial_y^3\vec\phi^0_{T}||_{\rho,\theta,I_0}^2dt.
    \end{aligned}
\end{equation*}

Substituting the above estimates of $N_{1,j}-N_{5,j}$ into \eqref{6eq44} and using $C_*M\ge 1$ yield that 
\begin{align*}
&\sup_{t\in[0,T]}||\partial_y^2\vec\phi^0_{T}||_{\rho,\theta,I_0}^2+(\beta-CC_*M)\int^T_0||\partial_y^2\vec\phi^0_{T}||_{\rho,\theta+\frac{1}{2},I_0}^2dt+\frac{3}{4}\int^T_0||\partial_y^3\vec\phi^0_{T}||_{\rho,\theta,I_0}^2dt\\
&\leq \frac{CC_*M}{\beta^{j+\frac{1}{2}}}\int^T_0|h^0|_{\rho,\theta+\sigma-\frac{2j+3}{4}}^2dt+ C\int^T_0||y\partial_y^2\vec\phi^0_{T}||_{\rho,\theta-\frac{1}{2},I_0}^2dt,
\end{align*}
and for $j=1,2$,
\begin{align*}
&\sup_{t\in[0,T]}||y^j\partial_y^2\vec\phi^0_{T}||_{\rho,\theta,I_0}^2+(\beta-CC_*M)\int^T_0||y^j\partial_y^2\vec\phi^0_{T}||_{\rho,\theta+\frac{1}{2},I_0}^2dt\\
&\quad+\frac{3}{4}\int^T_0||y^j\partial_y^3\vec\phi^0_{T}||_{\rho,\theta,I_0}^2dt\\
&\leq \frac{CC_*M}{\beta^{j+\frac{1}{2}}}\int^T_0|h^0|_{\rho,\theta+\sigma-\frac{2j+3}{4}}^2dt+ C\int^T_0||y^{j-1}\partial_y^2\vec\phi^0_{T}||_{\rho,\theta,I_0}^2dt.
\end{align*}
Consequently, combining these estimates, we choose $\beta\gg C_*M$ such that for any $\theta\in\mathbb{R}$,
\begin{multline*}
\sup_{t\in[0,T]}\sum_{j=0}^2||y^{j}\partial_y^2\vec\phi^0_{T}||_{\rho,\theta+\frac{j}{2},I_0}^2+\beta\int^T_0\sum_{j=0}^2||y^{j}\partial_y^2\vec\phi^0_{T}||_{\rho,\theta+\frac{j+1}{2},I_0}^2dt\\
+\int^T_0\sum_{j=0}^2||y^{j}\partial_y^3\vec\phi^0_{T}||_{\rho,\theta+\frac{j}{2},I_0}^2dt\leq \frac{CC_*M}{\beta^{\frac{1}{2}}}\int^T_0|h^0|_{\rho,\theta+\sigma-\frac{3}{4}}^2dt. 
\end{multline*}
This implies that assertion \eqref{6eq42} holds for $i=0$ and thus completes the proof of Lemma \ref{lem:phiT}.
\end{proof}

From the estimate of $\partial_y^2\vec\phi^i_T$ in Lemma \ref{lem:phiT}, we obtain the estimate of $\vec\phi^i_T$.
\begin{lemma}\label{lem:phiTtwo}
Under the same assumption as given in Theorem \ref{thm:pri} with $1\leq \sigma\leq \frac{7}{6}$ relaxed to $1\leq\sigma\leq\frac{3}{2}$, it holds that for given $\theta\in \mathbb{R}$, $i=0,1$ and $j=0,1,2$,
\begin{align}\label{6eq55}
\int^T_0||\partial_y\vec\phi^i_{T}||_{\rho,\theta,I_i}^2dt\leq \frac{CC_*M}{\beta^{\frac{1}{2}}}\int^T_0|h^i|_{\rho,\theta+\sigma-\frac{7}{4}}^2dt,  
\end{align}
\begin{multline}\label{add6eq55}
\int^T_0\big(||\vec\phi^i_{T}||_{\rho,\theta,I_i}^2+||\varphi^i\partial_y\vec\phi^i_{T}||_{\rho,\theta,I_i}^2+||\partial_x\vec\phi^i_{T}||_{\rho,\theta-\sigma,I_i}^2\big)dt\\
\leq \frac{CC_*M}{\beta^{\frac{1}{2}}}\int^T_0|h^i|_{\rho,\theta+\sigma-\frac{9}{4}}^2dt,  
\end{multline}
\begin{align}\label{6eq57}
    \int^T_0\big(|\vec\phi^i_{T}|_{y=1-i}|_{\rho,\theta}^2+|\partial_x\vec\phi^i_{T}|_{y=1-i}|_{\rho,\theta-\sigma}^2\big)dt\leq \frac{CC_*M}{\beta^{\frac{1}{2}}}\int^T_0|h^i|_{\rho,\theta+\sigma-2}^2dt,    
\end{align}
and 
\begin{align}\label{6eq56} \int^T_0|\partial_y\vec\phi^i_{T}|_{y=0,1}|_{\rho,\theta}^2dt\leq \frac{CC_*M}{\beta^{\frac{1}{2}}}\int^T_0|h^i|_{\rho,\theta+\sigma-\frac{3}{2}}^2dt,  
\end{align}
provided $\beta$ is sufficiently large. Recall that $\vec\phi^i_{T}=\{\phi^i_{T,(m)}\}_{m\ge0}$ and $\varphi^i$ is defined in \eqref{def:varphi}. Here, for each $m \ge 0$, $\phi^i_{T,(m)}$ is  the solution of \eqref{eq:phiT}.
\end{lemma}
\begin{proof}
For brevity, we only prove the case $i=0$; the proof for $i=1$ is similar and thus omitted.

Observing $\phi^0_{T,(m)}\big|_{y=0}=0$ for $m\ge 0$, we integrate by parts and use Hardy inequality to conclude
\begin{multline*}
\norm{\partial_y\phi^0_{T,(m)}}_{L^2}^2=-\inner{\partial_y^2\phi^0_{T,(m)},\ \phi^0_{T,(m)}}_{L^2}\\
\leq \norm{y\partial_y^2\phi^0_{T,(m)}}_{L^2}\norm{y^{-1}\phi^0_{T,(m)}}_{L^2}\leq C\norm{y\partial_y^2\phi^0_{T,(m)}}_{L^2}\norm{\partial_y\phi^0_{T,(m)}}_{L^2},
\end{multline*}
which implies that
\begin{align*}
 \norm{\partial_y\phi^0_{T,(m)}}_{L^2}\leq C\norm{y\partial_y^2\phi^0_{T,(m)}}_{L^2}.
\end{align*}
As a result, by Lemma \ref{lem:phiT}, we have
\begin{equation*}
\int^T_0||\partial_y\vec\phi^0_{T}||_{\rho,\theta,I_0}^2dt\leq\int^T_0||y\partial_y^2\vec\phi^0_{T}||_{\rho,\theta,I_0}^2dt\leq \frac{CC_*M}{\beta^{\frac{3}{2}}}\int^T_0|h^0|_{\rho,\theta+\sigma-\frac{7}{4}}^2dt.
\end{equation*}
This with $\beta\ge 1$ gives assertion \eqref{6eq55} for $i=0$.

Moreover, using Hardy inequality and Lemma \ref{lem:phiT} again yields
\begin{multline*}
\int^T_0\big(||\vec\phi^0_{T}||_{\rho,\theta,I_0}^2+||y\partial_y\vec\phi^0_{T}||_{\rho,\theta,I_0}^2+||\partial_x\vec\phi^0_{T}||_{\rho,\theta-\sigma,I_0}^2\big)dt\\
\leq C\int^T_0\big(||y^2\partial_y^2\vec\phi^0_{T}||_{\rho,\theta,I_0}^2+||y^2\partial_x\partial_y^2\vec\phi^0_{T}||_{\rho,\theta-\sigma,I_0}^2\big)dt\\
\leq \frac{CC_*M}{\beta^{\frac{3}{2}}}\int^T_0|h^0|_{\rho,\theta+\sigma-\frac{9}{4}}^2dt\leq \frac{CC_*M}{\beta^{\frac{1}{2}}}\int^T_0|h^0|_{\rho,\theta+\sigma-\frac{9}{4}}^2dt.
\end{multline*}
Assertion \eqref{add6eq55} holds for $i=0$.

On the other hand, observing $\phi^0_{T,(m)}|_{y=0}=\partial_y\phi^0_{T,(m)}|_{y\to+\infty}=0$, we notice that
\begin{align*}  
\phi^0_{T,(m)}(t,x,y)&=-\int^y_0\int^{+\infty}_{y^{'}}\partial_y^2\phi^0_{T,(m)}(t,x,y^{''})dy^{''}dy^{'}\\&=-\int^{min\{y,(1+m)^{-\frac{1}{2}}\}}_0\int^{+\infty}_{y^{'}}\partial_y^2\phi^0_{T,(m)}(t,x,y^{''})dy^{''}dy^{'}\\&\quad-\int^y_{min\{y,(1+m)^{-\frac{1}{2}}\}}\int^{+\infty}_{y^{'}}\partial_y^2\phi^0_{T,(m)}(t,x,y^{''})dy^{''}dy^{'}.
\end{align*}
Then it follows from Hardy inequality that
\begin{align*}
    \sup_{y\ge0}|\phi^0_{T,(m)}|\leq& C(m+1)^{-\frac{1}{4}}\left\|\int^{+\infty}_y\partial_y^2\phi^0_{T,(m)}(y^{'})dy^{'}\right\|_{L_y^2}\\&+C(m+1)^{\frac{1}{4}}\left\|y\int^{+\infty}_y\partial_y^2\phi^0_{T,(m)}(y^{'})dy^{'}\right\|_{L_y^2}\\\leq&C(m+1)^{-\frac{1}{4}}\left\|y\partial_y^2\phi^0_{T,(m)}\right\|_{L_y^2}+C(m+1)^{\frac{1}{4}}\left\|y^2\partial_y^2\phi^0_{T,(m)}\right\|_{L_y^2}.
\end{align*}
By Lemma \ref{lem:phiT}, we have
\begin{equation*}
    \begin{aligned}
&\int^T_0|\vec\phi^0_{T}|_{y=1}|_{\rho,\theta}^2dt\\&\leq C\int^T_0||y\partial_y^2\vec\phi^0_{T}||_{\rho,\theta-\frac{1}{4},I_0}^2dt+C\int^T_0||y^2\partial_y^2\vec\phi^0_{T}||_{\rho,\theta+\frac{1}{4},I_0}^2dt
\\&\leq\frac{CC_*M}{\beta^{\frac{3}{2}}}\int^T_0|h^0|_{\rho,\theta+\sigma-2}^2dt\leq\frac{CC_*M}{\beta^{\frac{1}{2}}}\int^T_0|h^0|_{\rho,\theta+\sigma-2}^2dt, 
    \end{aligned}
\end{equation*}
and similarly,
\begin{equation*}
    \begin{aligned}
&\int^T_0|\partial_x\vec\phi^0_{T}|_{y=1}|_{\rho,\theta}^2dt\leq\frac{CC_*M}{\beta^{\frac{1}{2}}}\int^T_0|h^0|_{\rho,\theta+\sigma-2}^2dt. 
    \end{aligned}
\end{equation*}
Thus, assertion \eqref{6eq57} holds for $i=0$.

Finally, for the boundary term $\partial_y\phi^0_{T,(m)}|_{y=0,1}$, by the interpolation inequality, we use Lemma \ref{lem:phiT} and \eqref{6eq55} to deduce that
\begin{multline*}
\int^T_0|\partial_y\vec\phi^0_{T}|_{y=0,1}|_{\rho,\theta}^2dt\leq C\int^T_0||\partial_y\vec\phi^0_{T}||_{\rho,\theta+\frac{1}{4},I_0}||\partial_y^2\vec\phi^0_{T}||_{\rho,\theta-\frac{1}{4},I_0}dt
\\\leq\frac{CC_*M}{\beta^{\frac{1}{2}}}\int^T_0|h^0|_{\rho,\theta+\sigma-\frac{3}{2}}^2dt.
\end{multline*}
This gives assertion \eqref{6eq56} for $i=0$, thus completing the proof of Lemma \ref{lem:phiTtwo}.
\end{proof}

\subsubsection{The estimate of $\vec\phi_{R}$: Full construction of boundary corrector} All we left is the term $\vec\phi_{R}$. Like previous argument, we define 
\begin{align*}
    \vec\phi_{R}=\{\phi_{R,(m)}^0\}_{m\ge 0}+\{\phi_{R,(m)}^1\}_{m\ge0}\stackrel{\rm def}{=}\vec\phi^0_R+\vec\phi^1_R,
\end{align*}
where for given $m\ge0$, $\phi_{R,(m)}^i$ $(i=0,1)$ satisfies the following equation  in the domain $\Omega$:
\begin{equation}\label{eq:phiR}
    \left\{
    \begin{aligned}
&(\partial_t+u\partial_x-\partial_y^2)\partial_y^2\phi^i_{R,(m)}-(\partial_x\phi^i_{R,(m)})\partial_y^2u=\sum^{m-1}_{k=[\frac{m}{2}]+1}\binom{m}{k}\phi^i_{R,(k+1)}\partial_x^{m-k}\partial_y^2u\\
&\quad-\sum^{[\frac{m}{2}]}_{k=1}\binom{m}{k}(\partial_x^ku)\partial_y^2\phi^i_{R,(m-k+1)}-\sum^m_{k=[\frac{m}{2}]+1}\binom{m}{k}(\partial_y\phi^i_{R,(k)})\partial_x^{m-k+1}\partial_yu\\
&\quad-\sum^{[\frac{m}{2}]}_{k=0}\binom{m}{k}(\partial_x^kv)\partial_y^3\phi^i_{R,(m-k)}+\mathcal{R}^i_{R,(m)},\\
&\phi^i_{R,(m)}|_{y=1-i}=-(\partial_x^m\phi^i_{H}+\phi^i_{T,(m)})|_{y=1-i},\quad \phi^i_{R,(m)}|_{y=i}=0,\ \  \partial_y^2\phi^i_{R,(m)}|_{y=0,1}=0,\\
&\phi^i_{R,(m)}|_{t=0}=0,        
    \end{aligned}
    \right.
\end{equation}
where $\mathcal{\vec R}^i_R=\{\mathcal{R}^i_{R,(m)}\}_{m\ge 0}$ is defined by
\begin{equation}\label{def:RRTH}
    \begin{aligned}
\mathcal{R}^i_{R,(m)}\stackrel{\rm def}{=}  &\sum^{m}_{k=[\frac{m}{2}]+1}\binom{m}{k}(\partial_x^{k+1}\phi^i_{H})\partial_x^{m-k}\partial_y^2u +(\partial_x\phi_{T,(m)}^i)\partial_y^2u\\
&+\sum^{m-1}_{k=[\frac{m}{2}]+1}\binom{m}{k}(\phi^i_{T,(k+1)})\partial_x^{m-k}\partial_y^2u.
    \end{aligned}
\end{equation}

In order to homogenize boundary condition, we introduce $\vec F^i_R=\{F^i_{R,(m)}\}_{m\ge 0}$, $\vec G^i_R=\{G^i_{R,(m)}\}_{m\ge 0}$ and $\vec\Phi^i_{R}=\{\Phi^i_{R,(m)}\}_{m\ge0}$ by setting
\begin{equation}\label{6eq61}
		\left\{
		\begin{aligned}
&F^i_{R,(m)}\stackrel{\rm def}{=}-(\partial_x^m\phi^i_{H}+\phi^i_{T,(m)})|_{y=1-i},\\
&G^i_{R,(m)}\stackrel{\rm def}{=}\varphi^i(\phi^i_{T,(m)}+B^i_{H,(m)}),\\
&\Phi^i_{R,(m)}\stackrel{\rm def}{=}\phi^i_{R,(m)}+G^i_{R,(m)},
		\end{aligned}
		\right.
\end{equation}
where $\vec B^i_H=\{B^i_{H,(m)}\}_{m\ge 0}$ is given by
\begin{equation}\label{def:B}
    \widehat{L_{\rho,m,\theta}B^i_{H,(m)}}(\xi,x,y)\stackrel{\rm def}{=}\widehat{\phi^i_{H,m}}|_{y=1-i}e^{-\varphi^{1-i}\sqrt{i\xi+\beta(m+1)}}
\end{equation}
for given $\theta\in\mathbb{R}$. Recall that $\phi^i_{H,m}=L_{\rho,m,\theta}\partial_x^m\phi^i_H$.
Then $\Phi^i_{R,(m)}$ satisfies
\begin{equation}\label{eq:PHI}
		\left\{
		\begin{aligned}
&\partial_y^2\Phi^i_{R,(m)}=\partial_y^2\phi^i_{R,(m)}+\partial_y^2G^i_{R,(m)},\\
&\Phi^i_{R,(m)}|_{y=0,1}=0.
		\end{aligned}
		\right.
\end{equation}
We now derive some estimates for $\vec\phi_R$ by estimating $\vec F^i_{R}$, $\vec G^i_{R}$ and $\vec\Phi^i_R$.

\begin{lemma}\label{lem:FG}
Under the same assumption as given in Theorem \ref{thm:pri} with $1\leq \sigma\leq \frac{7}{6}$ relaxed to $1\leq\sigma\leq\frac{3}{2}$, it holds that for given $\theta\in \mathbb{R}$ and $i=0,1$, 
\begin{align}\label{aq52}
  \int^T_0|\vec F^i_R|_{\rho,\theta}^2dt+\int^T_0|\partial_x\vec F^i_R|_{\rho,\theta-\sigma}^2dt \leq \frac{CC_*M}{\beta^{\frac{1}{2}}}\int^T_0|h^i|_{\rho,\theta-\frac{1}{2}}^2dt, 
\end{align}
\begin{align}\label{aq54}  \int^T_0||\vec G^i_R||_{\rho,\theta}^2+||\partial_y\vec G^i_R||_{\rho,\theta}^2+||\varphi^i\partial_y^2\vec G^i_R||_{\rho,\theta}^2dt\leq \frac{CC_*M}{\beta^{\frac{1}{2}}}\int^T_0|h^i|_{\rho,\theta-\frac{3}{4}}^2dt,
\end{align}
\begin{align}\label{aq55}
    \int^T_0||\partial_y^2\vec G^i_R||_{\rho,\theta}^2dt\leq \frac{CC_*M}{\beta^{\frac{1}{2}}}\int^T_0|h^i|_{\rho,\theta-\frac{1}{4}}^2dt,
\end{align}
provided $\beta$ is sufficiently large. Moreover, $\vec\phi^i_{R}=\{\phi^i_{R,(m)}\}_{m\ge 0} (i=0,1)$ has the following estimate:
\begin{multline}\label{6eq64}
\int^T_0|\partial_y\vec {\phi}^i_{R}|_{y=0,1}|_{\rho,\theta}^2+\|\partial_y\vec \phi^i_{R}\|_{\rho,\theta}^2+\|\vec\phi^i_{R}\|_{\rho,\theta}^2dt\\
\leq\frac{CC_*M}{\beta^{\frac{1}{2}}}\int^T_0|h^i|_{\rho,\theta-\frac{3}{4}}^2dt+C\int^T_0\|\partial_y^2\vec\phi^i_{R}\|_{\rho,\theta}^2dt.
\end{multline}
Recall that $\phi^i_{R,(m)}$ is the solution of \eqref{eq:phiR} and $\vec F^i_{R}$, $\vec G^i_{R}$ are defined in \eqref{6eq61}.
\end{lemma}
\begin{proof}
Here, we only give the proof of the case $i=0$. The case $i=1$ is the same.
By the definition of $F^0_{R}$ in \eqref{6eq61}, we get
\begin{equation*}
    \begin{aligned}
&\int^T_0|\vec F^0_R|_{\rho,\theta}^2dt\leq   \int^T_0|\phi^0_{H}|_{y=1}|_{\rho,\theta}^2dt+\int^T_0|\vec\phi^0_{T}|_{y=1}|_{\rho,\theta}^2dt\\
&\leq \frac{CC_*M}{\beta^{\frac{1}{2}}}\int^T_0|h^0|_{\rho,\theta-\frac{1}{2}}^2dt+ \frac{CC_*M}{\beta^{\frac{1}{2}}}\int^T_0|h^0|_{\rho,\theta+\sigma-2}^2dt\leq  \frac{CC_*M}{\beta^{\frac{1}{2}}}\int^T_0|h^0|_{\rho,\theta-\frac{1}{2}}^2dt,      
    \end{aligned}
\end{equation*}
the second inequality following from Lemmas \ref{lem:phiH} and \ref{lem:phiTtwo} and the last inequality using $1\leq \sigma\leq \frac{3}{2}$. Similarly,
\begin{equation*}
\int^T_0|\partial_x\vec F^0_R|_{\rho,\theta-\sigma}^2dt   \leq  \frac{CC_*M}{\beta^{\frac{1}{2}}}\int^T_0|h^0|_{\rho,\theta-\frac{1}{2}}^2dt. 
\end{equation*}
Assertion \eqref{aq52} holds for $i=0$.

On the other hand, by the representation \eqref{express:phiH0} of $\widehat{\phi^0_{H,m}}$ and \eqref{def:B}, we calculate that for any $\theta\in\mathbb{R}$,
\begin{multline*}
\int^T_0\big(||(\vec{B}_H^0,y\partial_y\vec{B}_H^0,y^2\partial_y^2\vec{B}_H^0)||_{\rho,\theta}^2+||y\vec{B}_H^0||_{\rho,\theta+\frac{1}{2}}^2\big)dt\\
+\int^T_0||(\partial_y\vec{B}_H^0,y\partial_y^2\vec{B}_H^0)||_{\rho,\theta-\frac{1}{2}}^2dt\leq \frac{CC_*M}{\beta^{\frac{1}{2}}}\int^T_0|h^0|_{\rho,\theta-\frac{3}{4}}^2dt,  
\end{multline*}
Then, we use the above estimate and Lemmas \ref{lem:phiT}-\ref{lem:phiTtwo} to conclude that for given $\theta\in\mathbb{R}$, recalling that $\vec G^0_R$ is given in \eqref{6eq61} and $1\leq \sigma\leq \frac{3}{2}$,
\begin{equation}\label{hgh2}
    \begin{aligned}
\int^T_0||\vec G^0_R||_{\rho,\theta}^2dt&\leq   \int^T_0||y\vec{B}_H^0||_{\rho,\theta}^2dt+\int^T_0||y\vec\phi^0_{T}||_{\rho,\theta}^2dt\\
&\leq\int^T_0||y\vec{B}_H^0||_{\rho,\theta}^2dt+\int^T_0||\vec\phi^0_{T}||_{\rho,\theta}^2dt\leq \frac{CC_*M}{\beta^{\frac{1}{2}}}\int^T_0|h^0|_{\rho,\theta-\frac{3}{4}}^2dt,   
    \end{aligned}
\end{equation}
\begin{equation}\label{hgh}
    \begin{aligned}
&\int^T_0\big(||\partial_y\vec G^0_R||_{\rho,\theta}^2+||y\partial_y^2\vec G^0_R||_{\rho,\theta}^2\big)dt\\
&\leq  C \int^T_0||(\vec{B}_H^0,y\partial_y\vec{B}_H^0,y^2\partial_y^2\vec{B}_H^0)||_{\rho,\theta}^2dt+C\int^T_0||(\vec\phi^0_{T},y\partial_y\vec\phi^0_{T},y^2\partial_y^2\vec\phi^0_{T})||_{\rho,\theta}^2dt\\
&\leq \frac{CC_*M}{\beta^{\frac{1}{2}}}\int^T_0|h^0|_{\rho,\theta-\frac{3}{4}}^2dt,   
    \end{aligned}
\end{equation}
and
\begin{equation*}
    \begin{aligned}
\int^T_0||\partial_y^2\vec G^0_R||_{\rho,\theta}^2dt&\leq  2\int^T_0||(\partial_y\vec{B}_H^0,y\partial_y^2\vec{B}_H^0)||_{\rho,\theta}^2dt+2\int^T_0||(\partial_y\vec\phi^0_{T},y\partial_y^2\vec\phi^0_{T})||_{\rho,\theta}^2dt\\&\leq \frac{CC_*M}{\beta^{\frac{1}{2}}}\int^T_0|h^0|_{\rho,\theta-\frac{1}{4}}^2dt.   
    \end{aligned}
\end{equation*}
Thus, assertions \eqref{aq54} and \eqref{aq55} holds for $i=0$.

It remains to estimate \eqref{6eq64}. In view of \eqref{eq:PHI}, we use Hardy inequality to obtain
\begin{equation*}
  ||\partial_y\Phi^0_{R,(m)}||_{L^2}\leq C||y\partial_y^2\Phi^0_{R,(m)}||_{L^2}\leq C||y\partial_y^2\phi^0_{R,(m)}||_{L^2}+C||y\partial_y^2G^0_{R,(m)}||_{L^2},  
\end{equation*}
which with \eqref{hgh} and $0\leq y\leq 1$ gives
\begin{equation*}
    \int^T_0||\partial_y\vec{\Phi}^0_{R}||_{\rho,\theta}^2dt\leq  \frac{CC_*M}{\beta^{\frac{1}{2}}}\int^T_0|h^0|_{\rho,\theta-\frac{3}{4}}^2dt+C\int^T_0||\partial_y^2\vec\phi^0_{R}||_{\rho,\theta}^2dt. 
\end{equation*}
Bringing $\partial_y\vec\phi^0_{R}=\partial_y\vec\Phi^0_{R}-\partial_y\vec G^0_{R}$ into the above inequality and using \eqref{hgh}  again yield
\begin{multline*}
\int^T_0||\partial_y\vec\phi^0_{R}||_{\rho,\theta}^2dt\leq\int^T_0||\partial_y\vec{\Phi}^0_{R}||_{\rho,\theta}^2dt+\int^T_0||\partial_y\vec G^0_R||_{\rho,\theta}^2dt
\\
\leq  \frac{CC_*M}{\beta^{\frac{1}{2}}}\int^T_0|h^0|_{\rho,\theta-\frac{3}{4}}^2dt+C\int^T_0||\partial_y^2\vec\phi^0_{R}||_{\rho,\theta}^2dt. 
\end{multline*}
This with Sobolev embedding inequality gives
\begin{align*}
\int^T_0|\partial_y{\vec\phi}^0_{R}|_{y=0,1}|_{\rho,\theta}^2dt
&\leq C\int^T_0||\partial_y{\vec\phi}^0_{R}||_{\rho,\theta}^2dt+C\int^T_0||\partial_y^2{\vec\phi}^0_{R}||_{\rho,\theta}^2dt\\&\leq \frac{CC_*M}{\beta^{\frac{1}{2}}}\int^T_0|h^0|_{\rho,\theta-\frac{3}{4}}^2dt+C\int^T_0||\partial_y^2\vec\phi^0_{R}||_{\rho,\theta}^2dt.
\end{align*}
Moreover, Poincar\'e inequality with \eqref{hgh2} implies that
\begin{multline*}
\int^T_0||\vec\phi^0_{R}||_{\rho,\theta}^2dt\leq C\int^T_0\big(||\partial_y\vec{\Phi}^0_{R}||_{\rho,\theta}^2+||\vec G^0_R||_{\rho,\theta}^2\big)dt
\\
\leq  \frac{CC_*M}{\beta^{\frac{1}{2}}}\int^T_0|h^0|_{\rho,\theta-\frac{3}{4}}^2dt+C\int^T_0||\partial_y^2\vec\phi^0_{R}||_{\rho,\theta}^2dt.  
\end{multline*}
Thus, we get the validity of \eqref{6eq64} for $i=0$, completing the proof of Lemma \ref{lem:FG}.
\end{proof}

The estimate of $\mathcal{\vec R}^i_R$ follows from  assumption \eqref{ass:pri} and Lemmas \ref{lem:principle}, \ref{lem:phiH} and \ref{lem:phiTtwo}. This estimate is stated as follows. 
\begin{lemma}\label{lem:RR}
 Under the same assumption as given in Theorem \ref{thm:pri} with $1\leq \sigma\leq \frac{7}{6}$ relaxed to $1\leq\sigma\leq\frac{3}{2}$, it holds that for given $\theta\in \mathbb{R}$ and $i=0,1$,
 \begin{equation*}
\int^T_0\norm{\mathcal{\vec R}^i_{R}}_{\rho,\theta}^2 dt\leq \frac{CC_*^2M^2}{\beta^{\frac{1}{2}}}\int^T_0|h^i|_{\rho,\theta+1}^2dt,
 \end{equation*}
provided $\beta$ is sufficiently large. Recall that $\mathcal{\vec R}^i_R$ is defined by \eqref{def:RRTH}.
\end{lemma}

\begin{proof}
Observing $\partial_y^3u=(\partial_t+u\partial_x+v\partial_y)\partial_yu-(f\partial_x+g\partial_y)\partial_yf$,  we use assumption \eqref{ass:pri} and Lemma \ref{lem:principle} to conclude that
\begin{equation*}
\sup_{t\in[0,T]}\norm{\partial_y^3u}_{\rho,r-\sigma}\leq CC_*M.
\end{equation*}
On the other hand, for any given $r\ge 10$, $1\leq \sigma\leq \frac{3}{2}$ and $\theta\in\mathbb{R}$, we have, recalling the definition \eqref{def2} of $L_{\rho,m,r}$,
 \begin{equation*}
     \begin{aligned}
\binom{m}{k}\frac{L_{\rho,m,\theta}}{L_{\rho,k+1,\theta+\sigma}L_{\rho,m-k+2,r-\sigma}}\leq \frac{C}{m-k+1},\quad \mathrm{if}\  \ \big[\frac{m}{2}\big]+1\leq k\leq m.
     \end{aligned}
 \end{equation*}
 Following an analogous argument in Lemma \ref{lem:RiTH}, we use the above estimates along with assumption \eqref{ass:pri} and Lemmas \ref{lem:phiH} and \ref{lem:phiTtwo} to deduce that
 \begin{equation*}
 \begin{aligned}
    \int^T_0\norm{\mathcal{\vec R}^i_{R}}_{\rho,\theta}^2 dt&\leq CC_*M\int^T_0\big(\norm{\vec\phi^i_H}_{\rho,\theta+\sigma}^2+\norm{\vec\phi^i_T}_{\rho,\theta+\sigma}^2+\norm{\partial_x\vec\phi^i_T}_{\rho,\theta}^2\big)dt\\
    &\leq \frac{CC_*M}{\beta}\int^T_0|h^i|_{\rho,\theta+\sigma-\frac{1}{2}}^2dt+\frac{CC_*^2M^2}{\beta^\frac{1}{2}}\int^T_0|h^i|_{\rho,\theta+2\sigma-\frac{9}{4}}^2dt\\
    &\leq \frac{CC_*^2M^2}{\beta^\frac{1}{2}}\int^T_0|h^i|_{\rho,\theta+1}^2dt,
 \end{aligned}
 \end{equation*}
 the last inequality using $\beta\ge 1$, $1\leq\sigma\leq \frac{3}{2}$ and $C_*M\ge 1$. This completes the proof.
\end{proof}

We now proceed to the main part of this subsection: deriving estimates for the system \eqref{eq:phiR}.
\begin{lemma}\label{lem:phiR}
Under the same assumption as given in Theorem \ref{thm:pri} with $1\leq \sigma\leq \frac{7}{6}$ relaxed to $1\leq\sigma\leq\frac{3}{2}$, it holds that for given $\theta\in \mathbb{R}$ and $i=0,1$, 
\begin{multline}\label{6eq69}
\sup_{t\in[0,T]}||\partial_y^2\vec\phi^i_{R}||_{\rho,\theta}^2+\beta\int^T_0||\partial_y^2\vec\phi^i_{R}||_{\rho,\theta+\frac{1}{2}}^2dt
+\int^T_0||\partial_y^3\vec\phi^i_{R}||_{\rho,\theta}^2dt\\
\leq \frac{C\delta^{-3}C_*^2M^2}{\beta^\frac{1}{2}}\int^T_0|h^i|_{\rho,\theta+\frac{1}{2}}^2dt 
\end{multline}
provided $\beta$ is sufficiently large. Recall that $\vec\phi^i_R=\{\phi^i_{R,(m)}\}_{m\ge0}$, where for each $m\ge 0$, $\phi^i_{R,(m)}$ is the solution of system \eqref{eq:phiR}.
\end{lemma}

\begin{proof}
We establish the result by applying the procedure of Lemmas \ref{lem:yu} and \ref{lem:phis}.  The main difference comes from the boundary conditions:
\begin{equation*}
    \phi^i_{R,(m)}|_{y=1-i}=-(\partial_x^m\phi^i_{H}+\phi^i_{T,(m)})|_{y=1-i},\quad \phi^i_{R,(m)}|_{y=i}=0.
\end{equation*}
This leads to 
\begin{align*}
   & \inner{(\partial_x\phi^i_{R,(m)})\partial_y^2u,\ \frac{\partial_y^2\phi^i_{R,(m)}}{\partial_y^2u}}_{L^2}=\inner{\partial_x\phi^i_{R,(m)},\ \partial_y^2\phi^i_{R,(m)}}_{L^2}\\
 &=(-1)^i\int_{\mathbb{T}}(\partial_x\phi^i_{R,(m)}|_{y=1-i})(\partial_y\phi^i_{R,(m)}|_{y=1-i})dx- \inner{\partial_x\partial_y\phi^i_{R,(m)},\ \partial_y\phi^i_{R,(m)}}_{L^2} \\&=(-1)^i\int_{\mathbb{T}}(\partial_x\phi^i_{R,(m)}|_{y=1-i})(\partial_y\phi^i_{R,(m)}|_{y=1-i})dx. 
\end{align*}
Recalling $\phi^i_{R,(m)}|_{y=1-i}=F^i_{R,(m)}$, we use Lemma \ref{lem:FG} to conclude that
\begin{align*}
  & \int^T_0\sum^{+\infty}_{m=0}L_{\rho,m,\theta}^2\inner{(\partial_x\phi^i_{R,(m)})\partial_y^2u,\ \frac{\partial_y^2\phi^i_{R,(m)}}{\partial_y^2u}}_{L^2}dt\\
  &\leq \int^T_0\sum^{+\infty}_{m=0}\inner{L_{\rho,m,\theta-\frac{1}{2}}\norm{\partial_xF^i_{R,(m)}}_{L_x^2}}\inner{L_{\rho, m,\theta+\frac{1}{2}}\norm{\partial_y\phi^i_{R,(m)}|_{y=1-i}}_{L_x^2}}dt\\
  &\leq \frac{CC_*M}{\beta^{\frac{1}{2}}}\int^T_0|h^i|_{\rho,\theta+\frac{1}{2}}^2dt+C\int^T_0||\partial_y^2\vec\phi^i_{R}||_{\rho,\theta+\frac{1}{2}}^2dt.
\end{align*}
On the other hand, for the remainder term $\mathcal{\vec R}^i_{R}=\{\mathcal{R}^i_{R,(m)}\}_{m\ge 0}$, Lemma \ref{lem:RR} with the convexity of $\partial_y^2u$ gives
\begin{align*}
     \int^T_0\sum^{+\infty}_{m=0}L_{\rho,m,\theta}^2\inner{\mathcal{R}^i_{R,(m)},\ \frac{\partial_y^2\phi^i_{R,(m)}}{\partial_y^2u}}_{L^2}dt
    &\leq C\delta^{-1}\int^T_0\norm{\mathcal{\vec R}^i_{R}}_{\rho,\theta-\frac{1}{2}}||\partial_y^2\vec\phi^i_{R}||_{\rho,\theta+\frac{1}{2}}dt\\
    \leq \frac{C\delta^{-2}C_*^2M^2}{\beta^\frac{1}{2}}&\int^T_0|h^i|_{\rho,\theta+\frac{1}{2}}^2dt+C\int^T_0||\partial_y^2\vec\phi^i_{R}||_{\rho,\theta+\frac{1}{2}}^2dt.
\end{align*}
Then combining the above estimates, we follow a similar procedure in Lemma \ref{lem:yu} to obtain, observing $0<\delta<\frac{1}{2}$ and $C_*M\ge 1$,
\begin{align*} &\sup_{t\in[0,T]}||\partial_y^2\vec\phi^i_{R}||_{\rho,\theta}^2+\beta\int^T_0||\partial_y^2\vec\phi^i_{R}||_{\rho,\theta+\frac{1}{2}}^2dt
+\int^T_0||\partial_y^3\vec\phi^i_{R}||_{\rho,\theta}^2dt\\
&\leq \frac{C\delta^{-3}C_*^2M^2}{\beta^\frac{1}{2}}\int^T_0|h^i|_{\rho,\theta+\frac{1}{2}}^2dt+C\delta^{-3}C_*^\frac{3}{2}M^\frac{3}{2}\int^T_0||\partial_y^2\vec\phi^i_{R}||_{\rho,\theta+\frac{1}{2}}^2 dt\\
&\quad+C\delta^{-3}C_*^\frac{3}{2}M^\frac{3}{2}\int^T_0||\partial_y^2\vec\phi^i_{R}||_{\rho,\theta+\frac{1}{2}}||\partial_y^3\vec\phi^i_{R}||_{\rho,\theta}dt\\
&\quad
+C\delta^{-3}C_*^\frac{3}{2}M^\frac{3}{2}\int^T_0||\partial_y^2\vec\phi^i_{R}||_{\rho,\theta+\frac{1}{2}}^\frac{1}{2}||\partial_y^3\vec\phi^i_{R}||_{\rho,\theta}^\frac{3}{2}dt\\
&\leq \frac{C\delta^{-3}C_*^2M^2}{\beta^\frac{1}{2}}\int^T_0|h^i|_{\rho,\theta+\frac{1}{2}}^2dt+C\delta^{-12}C_*^6M^6\int^T_0||\partial_y^2\vec\phi^i_{R}||_{\rho,\theta+\frac{1}{2}}^2 dt\\
&\quad+\frac{1}{2}\int^T_0||\partial_y^3\vec\phi^i_{R}||_{\rho,\theta}^2dt.\\
\end{align*}
Consequently, assertion \eqref{6eq69} follows by choosing $\beta$ large enough such that $\beta\ge 2C\delta^{-12}C_*^6M^6$ in the inequality above. The proof of Lemma \ref{lem:phiR} is thus completed.
\end{proof}

\subsubsection{The estimate of $\vec\phi_{b}$} With the estimates for $\phi_H$, $\vec\phi_T$, and $\vec\phi_R$ established, we now turn to deriving the estimate for $\vec\phi_{b}$.
\begin{lemma}\label{lem:phib}
Under the same assumption as given in Theorem \ref{thm:pri} with $1\leq \sigma\leq \frac{7}{6}$ relaxed to $1\leq\sigma\leq\frac{3}{2}$, it holds that for given $\theta\in \mathbb{R}$,
\begin{multline}\label{equa64}
\int^T_0\inner{||\vec\phi_{b}||_{\rho,\theta}^2+||y(1-y)\partial_y\vec\phi_{b}||_{\rho,\theta}^2+||y^2(1-y)^2\partial_y^2\vec\phi_{b}||_{\rho,\theta}^2}dt\\
\leq \frac{CC_*^2M^2\delta^{-3}}{\beta^{\frac{1}{2}}}\int^T_0|(h^0,h^1)|_{\rho,\theta}^2dt,
\end{multline}
provided $\beta$ is sufficiently large. Recall that $\vec\phi_b=\{\phi_{b,(m)}\}_{m\ge0}$ where for each $m\ge 0$, $\phi_{b,(m)}$ is the solution of \eqref{eq:phib}.
\end{lemma}

\begin{proof}
The proof is straightforward. Observing that 
\begin{align*}
    \phi_{b,(m)}&=\partial_x^m\phi_{H}+\phi_{T,(m)}+\phi_{R,(m)}\\
    &=\partial_x^m\phi_{H}^0+\partial_x^m\phi_{H}^1+\phi_{T,(m)}^0+\phi_{T,(m)}^1+\phi_{R,(m)}^0+\phi_{R,(m)}^1,
\end{align*}
we have for any given $\theta\in \mathbb{R}$, recalling $\varphi^i$ is given by \eqref{def:phi},
\begin{multline*}
    ||\vec\phi_{b}||_{\rho,\theta}^2+||y(1-y)\partial_y\vec\phi_{b}||_{\rho,\theta}^2+||y^2(1-y)^2\partial_y^2\vec\phi_{b}||_{\rho,\theta}^2\\
    \leq C\sum_{i=0}^1\sum_{j=0}^2\big(\norm{(\varphi^i)^j\partial_y^j\phi_H^i}_{\rho,\theta}^2+\norm{(\varphi^i)^j\partial_y^j\vec\phi^i_T}_{\rho,\theta}^2+\norm{\partial_y^j\vec\phi^i_R}_{\rho,\theta}^2\big).
\end{multline*}
Then assertion \eqref{equa64} of Lemma \ref{lem:phib} follows from Lemmas \ref{lem:phiH}, \ref{lem:phiT}- \ref{lem:phiTtwo} and \ref{lem:FG}-\ref{lem:phiR}, thus completing the proof.
\end{proof}

\subsection{Existence of $(h_0,h_1)$ and completing the proof of Proposition \ref{prop:phi}}
In this part, we establish the existence of $(h^0, h^1)$, thus proving the validity of decomposition \eqref{de:phib}. In view of
\begin{align*}
\forall\ m\ge0,\quad\phi_{b,(m)}=\partial_x^m\phi_{H}+\phi_{T,(m)}+\phi_{R,(m)}, 
\end{align*}
we deduce from \eqref{eq:phiH}, \eqref{eq:phiT} and \eqref{eq:phiR} that
\begin{equation*}
		\begin{aligned}
\partial_y\phi_{b,(m)}|_{y=0}=\partial_x^mh^0+R^{00}_{b,(m)}+R^{01}_{b,(m)},\ \ \partial_y\phi_{b,(m)}|_{y=1}=\partial_x^mh^1+R^{10}_{b,(m)}+R^{11}_{b,(m)}.
		\end{aligned}
\end{equation*}
Here $R^{ji}_{b,(m)}(i,j=0,1)$ are linear operators and are defined by
\begin{equation*}
		\left\{
		\begin{aligned}
&R^{00}_{b,(m)}\stackrel{\rm def}{=}\inner{\partial_y\phi^0_{T,(m)}+\partial_y\phi^0_{R,(m)}}|_{y=0},\\
&R^{01}_{b,(m)}\stackrel{\rm def}{=}\inner{\partial_x^m\partial_y\phi^1_{H}+\partial_y\phi^1_{T,(m)}+\partial_y\phi^1_{R,(m)}}|_{y=0},\\
&R^{10}_{b,(m)}\stackrel{\rm def}{=}\inner{\partial_x^m\partial_y\phi^0_{H}+\partial_y\phi^0_{T,(m)}+\partial_y\phi^0_{R,(m)}}|_{y=1},\\
&R^{11}_{b,(m)}\stackrel{\rm def}{=}\inner{\partial_y\phi^1_{T,(m)}+\partial_y\phi^1_{R,(m)}}|_{y=1}.
		\end{aligned}
		\right.
\end{equation*}
Compared with the boundary conditions in system \eqref{eq:phib}, we need to find $(h^0,h^1)$ such that
\begin{equation}\label{7eq2}
		\left\{
		\begin{aligned}
&\partial_x^mh^0+R^{00}_{b,(m)}+R^{01}_{b,(m)}=-\partial_y\phi_{s,(m)}|_{y=0}+\partial_x^m\mathcal{C}(0)-\partial_x^m\mathcal{C}(t),\\
&\partial_x^mh^1+R^{10}_{b,(m)}+R^{11}_{b,(m)}=-\partial_y\phi_{s,(m)}|_{y=1}+\partial_x^m\mathcal{C}(0)-\partial_x^m\mathcal{C}(t).
		\end{aligned}
		\right.
\end{equation}
 To do that, we defined an operator $R_{b}: X_{b,\theta}\to \tilde{X}_{b,\theta}$, which is defined by
\begin{align}\label{7eq3}
 R_{b}[h^0,h^1]=\inner{\left\{R^{00}_{b,(m)}+R^{01}_{b,(m)}\right\}_{m\ge 0},\ \left\{R^{10}_{b,(m)}+R^{11}_{b,(m)}\right\}_{m\ge 0}}.   
\end{align}
Here the Banach spaces $X_{b,\theta,T}$ and $\tilde{X}_{b,\theta,T}$ are defined by
\begin{equation*}
X_{b,\theta,T}= \left\{(h^0,h^1)\in L^2(0,T;L_x^2)\big| \int^T_0|(h^0,h^1)|^2_{\rho,\theta}dt<+\infty\right\},   
\end{equation*}
and
\begin{equation*}
\tilde{X}_{b,\theta,T}= \left\{(\vec h^0,\vec h^1)\in L^2(0,T;L_x^2)\big| \int^T_0|(\vec h^0,\vec h^1)|^2_{\rho,\theta}dt<+\infty\right\}.  
\end{equation*}
We remark that $R_b$ is a linear operator due to the linearity of systems \eqref{eq:phiH}, \eqref{eq:phiT} and \eqref{eq:phiR}.

\begin{lemma}[Existence of $(h^0,h^1)$]\label{lem:exist}
Under the same assumption as given in Theorem \ref{thm:pri}, it holds that for given $\theta\in \mathbb{R}$,
\begin{align}\label{7eq6}
   \int^T_0|R_{b}[h^0,h^1]|^2_{\rho,\theta}dt\leq \frac{CC_*^2M^2\delta^{-3}}{\beta^{\frac{1}{2}}}\int^T_0|(h^0,h^1)|^2_{\rho,\theta}dt,
\end{align}
provided $\beta$ is sufficiently large. Moreover, there exists $(h^0,h^1)\in X_{b,r+\sigma,T}$ such that
\eqref{7eq2} holds with $(h^0,h^1)$ satisfying
\begin{align*}
   \int^T_0|(h^0,h^1)|^2_{\rho,r+\sigma}dt\leq \frac{C}{\beta}\int^T_0\mathcal{Y}_\rho dt+\frac{C\delta^{-2}C_*M}{\beta}. 
\end{align*}
Here the operator $R_{b}$ is defined by \eqref{7eq3}.
\end{lemma}

\begin{proof}
First, it follows from Lemmas \ref{lem:phiH}, \ref{lem:phiTtwo}-\ref{lem:FG} and \ref{lem:phiR}, it is easy to get for any $\theta\in\mathbb{R}$ and $1\leq \sigma\leq \frac{3}{2}$,
\begin{equation*}
   \int^T_0|R_{b}[h^0,h^1]|^2_{\rho,\theta}dt\leq \frac{CC_*^2M^2\delta^{-3}}{\beta^{\frac{1}{2}}}\int^T_0|(h^0,h^1)|^2_{\rho,\theta}dt.
\end{equation*}
Thus, estimate \eqref{7eq6} holds.

On the other hand, in view of \eqref{7eq2}, we define the map $S_{b}: X_{b,\theta,T}\to X_{b,\theta,T}$ by
\begin{align*}
   S_{b}[h^0,h^1]= x_{s}+R_{b}[h^0,h^1],
\end{align*}
where
\begin{align*}
    x_{s}=\left(\{\partial_y\phi_{s,(m)}|_{y=0}\}_{m\ge 0},\{\partial_y\phi_{s,(m)}|_{y=1}\}_{m\ge 0}\right).
\end{align*}
For any $ (h^0,h^1), (\tilde{h}^0,\tilde{h}^1)\in X_{b,\theta,T}$, we have
\begin{align*}
    &\int^T_0|S_{b}[h^0,h^1]-S_{b}[\tilde{h}^0,\tilde{h}^1]|^2_{\rho,\theta}dt\leq \int^T_0|R_{b}[h^0,h^1]-R_{b}[\tilde{h}^0,\tilde{h}^1]|^2_{\rho,\theta}dt\\
    &\leq \int^T_0|R_{b}[h^0-\tilde{h}^0,h^1-\tilde{h}^1]|^2_{\rho,\theta}dt\leq \frac{CC_*^2M^2\delta^{-3}}{\beta^{\frac{1}{2}}}\int^T_0|(h^0-\tilde{h}^0,h^1-\tilde{h}^1)|^2_{\rho,\theta} dt.
\end{align*}
Choosing $\beta$ large enough such that $\frac{CC_*^2M^2\delta^{-3}}{\beta^{\frac{1}{2}}}\leq \frac{1}{2}$, we have that the map $S_{b}$ is a contract map. Moreover, Lemma \ref{lem:phis} indicates that $\partial_y\vec\phi_{s}|_{y=0,1}\in X_{b,r+\sigma,T}$.
As a consequence, there exists a unique $(h^0,h^1)\in X_{\rho,r+\sigma,T}$ such that
\eqref{7eq2} holds satisfying
\begin{multline*}
       \int^T_0|(h^0,h^1)|^2_{\rho,r+\sigma}dt\leq C\int^T_0\inner{|\partial_y\vec\phi_{s}|_{y=0,1}|^2_{\rho,r+\sigma}+|\mathcal{C}(t)-\mathcal{C}(0)|^2}dt\\
      \leq \frac{C}{\beta}\int^T_0\mathcal{Y}_\rho dt+\frac{C\delta^{-2}C_*M}{\beta^2}+\beta^{-1}\sup_{t\in [0,T]}\norm{u}_{L^2}^2\leq \frac{C}{\beta}\int^T_0\mathcal{Y}_\rho dt+\frac{C\delta^{-2}C_*M}{\beta},
\end{multline*}
where in the last line we have used Lemma \ref{lem:phis} as well as assumption \eqref{ass:pri}.
The proof of Lemma \ref{lem:exist} is thus completed.
\end{proof}

 With the estimates of $\vec\phi_s$ and $\vec\phi_b$, we now complete the proof of Proposition \ref{prop:phi}.
\begin{proof}[Completing the proof of Proposition \ref{prop:phi}]
In view of \eqref{def:phia} and \eqref{de:phia}, we have
\begin{align*}
\norm{\phi}_{\rho,r+\sigma}\leq \norm{\vec\phi_s}_{\rho,r+\sigma}+\norm{\vec\phi_b}_{\rho,r+\sigma}+\norm{\phi_0}_{\rho,r+\sigma}.
\end{align*}
Then assertion \eqref{est:phi} of Proposition \ref{prop:phi} follows by combining these estimates in Lemmas \ref{lem:phis}, \ref{lem:phib} and \ref{lem:exist} and choosing $\beta$ large enough. This completes the proof.
\end{proof}

\section{Proof of Theorem \ref{thm:pri}}\label{sec:proof}
This section is devoted to completing the proof of Theorem \ref{thm:pri}. Combining the estimates in Propositions \ref{prop:x0} and \ref{prop:phi} and using Cauchy inequality yield
 \begin{equation*}
 \begin{aligned}
&\sup_{t\in[0,T]}\mathcal{X}_{\rho,0}+\beta\int^T_0\mathcal{Y}_{\rho,0}dt+\int^T_0\mathcal{Z}_{\rho,0}dt\\
 &\leq C\delta^{-2}\int^T_0\mathcal{Y}_\rho dt+\frac{C\delta^{-4}C_*M}{\beta}+C\delta^{-3}C_*^\frac{3}{2}M^\frac{3}{2}\int^T_0\mathcal{Y}_\rho dt\\
 &\quad+C\delta^{-3}C_*^\frac{3}{2}M^\frac{3}{2}\int^T_0\mathcal{Y}_\rho^\frac{1}{2}\mathcal{Z}_\rho^\frac{1}{2} dt+C\delta^{-3}C_*M\int^T_0\mathcal{Y}_\rho^\frac{1}{4}\mathcal{Z}_\rho^\frac{3}{4}dt+C\delta^{-2}M\\
 &\leq C\delta^{-12}C_*^6M^6\int^T_0\mathcal{Y}_\rho dt+\frac{1}{4}\int^T_0\mathcal{Z}_{\rho}dt+\frac{C\delta^{-4}C_*M}{\beta}+C\delta^{-2}M.
  \end{aligned}
 \end{equation*}
For the terms $\mathcal{X}_{\rho,1}$ and $\mathcal{X}_{\rho,2}$, following an analogous argument without additional difficulty, we deduce that for $k=1,2$,
\begin{multline*}
\sup_{t\in[0,T]}\mathcal{X}_{\rho,k}+\beta\int^T_0\mathcal{Y}_{\rho,k}dt+\int^T_0\mathcal{Z}_{\rho,k}dt\\
\leq C\delta^{-12}C_*^6M^6\int^T_0\mathcal{Y}_\rho dt+\frac{1}{4}\int^T_0\mathcal{Z}_{\rho}dt+\frac{C\delta^{-4}C_*M}{\beta}+C\delta^{-2}M.
\end{multline*}
Recalling $\mathcal{X}_\rho$, $\mathcal{Y}_\rho$ and $\mathcal{Z}_\rho$ are defined in \eqref{def:energytwo},
we combine the estimates above to obtain
\begin{equation}\label{xxxx}
\sup_{t\in[0,T]}\mathcal{X}_{\rho}+\big(\beta-C\delta^{-12}C_*^6M^6\big)\int^T_0\mathcal{Y}_{\rho}dt+\frac{1}{4}\int^T_0\mathcal{Z}_{\rho}dt\leq \frac{C\delta^{-4}C_*M}{\beta}+C\delta^{-2}M.    
\end{equation}
For fixed $C_*\ge 1$, we choose $\beta$ large enough such that
\begin{equation*}
\beta-C\delta^{-12}C_*^6M^6\ge \frac{\beta}{2}\quad\mathrm{and}\quad \beta\ge C_*\delta^{-2}.
\end{equation*}
Then we deduce the above estimate \eqref{xxxx} that
\begin{equation*}
\sup_{t\in[0,T]}\mathcal{X}_{\rho}+\beta\int^T_0\mathcal{Y}_{\rho}dt+\int^T_0\mathcal{Z}_{\rho}dt\leq C\delta^{-2}M. 
\end{equation*}
Consequently, the desired assertion \eqref{ret:pri} follows by choosing $C_*\ge C\delta^{-2}$ in the above inequality. This completes the proof of Theorem \ref{thm:pri}.

\subsection*{Acknowledgements}

This work was supported by Natural Science Foundation of China
(Nos. 12325108, 12131017, 12221001), and the Natural Science Foundation of Hubei Province (No. 2019CFA007).  Xu
 would like to thank the support from the Research Centre for Nonlinear Analysis
in The Hong Kong Polytechnic University.

	

\begin{thebibliography}{10}

\bibitem{MR4362378}
N.~Aarach.
\newblock Hydrostatic approximation of the 2{D} {MHD} system in a thin strip
  with a small analytic data.
\newblock {\em J. Math. Anal. Appl.}, 509(2):Paper No. 125949, 55, 2022.

\bibitem{MR3327535}
R.~Alexandre, Y.-G. Wang, C.-J. Xu, and T.~Yang.
\newblock Well-posedness of the {P}randtl equation in {S}obolev spaces.
\newblock {\em J. Amer. Math. Soc.}, 28(3):745--784, 2015.

\bibitem{MR3795028}
D.~Chen, Y.~Wang, and Z.~Zhang.
\newblock Well-posedness of the linearized {P}randtl equation around a
  non-monotonic shear flow.
\newblock {\em Ann. Inst. H. Poincar\'e{} C Anal. Non Lin\'eaire},
  35(4):1119--1142, 2018.

\bibitem{MR3765768}
D.~Chen, Y.~Wang, and Z.~Zhang.
\newblock Well-posedness of the {P}randtl equation with monotonicity in
  {S}obolev spaces.
\newblock {\em J. Differential Equations}, 264(9):5870--5893, 2018.

\bibitem{MR4834601}
K.~Chen, W.-X. Li, and T.~Yang.
\newblock Local well-posedness of the three dimensional linearized {MHD}
  boundary layer system.
\newblock {\em Commun. Math. Anal. Appl.}, 3(4):483--500, 2024.

\bibitem{MR3925144}
H.~Dietert and D.~G\'erard-Varet.
\newblock Well-posedness of the {P}randtl equations without any structural
  assumption.
\newblock {\em Ann. PDE}, 5(1):Paper No. 8, 51, 2019.

\bibitem{MR1476316}
W.~E and B.~Engquist.
\newblock Blowup of solutions of the unsteady {P}randtl's equation.
\newblock {\em Comm. Pure Appl. Math.}, 50(12):1287--1293, 1997.

\bibitem{MR2601044}
D.~G\'erard-Varet and E.~Dormy.
\newblock On the ill-posedness of the {P}randtl equation.
\newblock {\em J. Amer. Math. Soc.}, 23(2):591--609, 2010.

\bibitem{MR4621052}
D.~Gerard-Varet, S.~Iyer, and Y.~Maekawa.
\newblock Improved well-posedness for the triple-deck and related models via
  concavity.
\newblock {\em J. Math. Fluid Mech.}, 25(3):Paper No. 69, 34, 2023.

\bibitem{MR4818200}
D.~G\'erard-Varet, Y.~Maekawa, and N.~Masmoudi.
\newblock Optimal {P}randtl expansion around a concave boundary layer.
\newblock {\em Anal. PDE}, 17(9):3125--3187, 2024.

\bibitem{MR3429469}
D.~Gerard-Varet and N.~Masmoudi.
\newblock Well-posedness for the {P}randtl system without analyticity or
  monotonicity.
\newblock {\em Ann. Sci. \'Ec. Norm. Sup\'er. (4)}, 48(6):1273--1325, 2015.

\bibitem{MR4149066}
D.~G\'erard-Varet, N.~Masmoudi, and V.~Vicol.
\newblock Well-posedness of the hydrostatic {N}avier-{S}tokes equations.
\newblock {\em Anal. PDE}, 13(5):1417--1455, 2020.

\bibitem{MR2849481}
Y.~Guo and T.~Nguyen.
\newblock A note on {P}randtl boundary layers.
\newblock {\em Comm. Pure Appl. Math.}, 64(10):1416--1438, 2011.

\bibitem{MR3461362}
M.~Ignatova and V.~Vicol.
\newblock Almost global existence for the {P}randtl boundary layer equations.
\newblock {\em Arch. Ration. Mech. Anal.}, 220(2):809--848, 2016.

\bibitem{MR4508068}
R.~Ji, J.~Wu, and X.~Xu.
\newblock Global well-posedness of the 2{D} {MHD} equations of damped wave type
  in {S}obolev space.
\newblock {\em SIAM J. Math. Anal.}, 54(6):6018--6053, 2022.

\bibitem{MR2975371}
I.~Kukavica and V.~Vicol.
\newblock On the local existence of analytic solutions to the {P}randtl
  boundary layer equations.
\newblock {\em Commun. Math. Sci.}, 11(1):269--292, 2013.

\bibitem{MR3590519}
I.~Kukavica, V.~Vicol, and F.~Wang.
\newblock The van {D}ommelen and {S}hen singularity in the {P}randtl equations.
\newblock {\em Adv. Math.}, 307:288--311, 2017.

\bibitem{MR4293727}
S.~Li and F.~Xie.
\newblock Global solvability of 2{D} {MHD} boundary layer equations in analytic
  function spaces.
\newblock {\em J. Differential Equations}, 299:362--401, 2021.

\bibitem{MR4465902}
W.-X. Li, N.~Masmoudi, and T.~Yang.
\newblock Well-posedness in {G}evrey function space for 3{D} {P}randtl
  equations without structural assumption.
\newblock {\em Comm. Pure Appl. Math.}, 75(8):1755--1797, 2022.

\bibitem{MR4661716}
W.-X. Li, M.~Paicu, and P.~Zhang.
\newblock Gevrey solutions of quasi-linear hyperbolic hydrostatic
  {N}avier-{S}tokes system.
\newblock {\em SIAM J. Math. Anal.}, 55(6):6194--6228, 2023.

\bibitem{MR3493958}
W.-X. Li, D.~Wu, and C.-J. Xu.
\newblock Gevrey class smoothing effect for the {P}randtl equation.
\newblock {\em SIAM J. Math. Anal.}, 48(3):1672--1726, 2016.

\bibitem{MR4498949}
W.-X. Li and R.~Xu.
\newblock Gevrey well-posedness of the hyperbolic {P}randtl equations.
\newblock {\em Commun. Math. Res.}, 38(4):605--624, 2022.

\bibitem{MR4494626}
W.-X. Li, R.~Xu, and T.~Yang.
\newblock Global well-posedness of a {P}randtl model from {MHD} in {G}evrey
  function spaces.
\newblock {\em Acta Math. Sci. Ser. B (Engl. Ed.)}, 42(6):2343--2366, 2022.

\bibitem{MR4915128}
W.-X. Li, X.~Xu, and T.~Yu.
\newblock On the hydrostatic approximation of the 3{D} {B}oussinesq equations
  of damped wave type.
\newblock {\em SIAM J. Math. Anal.}, 57(3):2952--2982, 2025.

\bibitem{lixuyang2025}
W.-X. Li, Z.~Xu, and A.~Yang.
\newblock Global well-posedness of the MHD boundary layer equations in the
 Sobolev space.
\newblock {\em Sci China Math}, 2025. DOI: 10.1007/s11425-025-2456-6.

\bibitem{MR4055987}
W.-X. Li and T.~Yang.
\newblock Well-posedness in {G}evrey function spaces for the {P}randtl
  equations with non-degenerate critical points.
\newblock {\em J. Eur. Math. Soc. (JEMS)}, 22(3):717--775, 2020.

\bibitem{MR4270479}
W.-X. Li and T.~Yang.
\newblock Well-posedness of the {MHD} boundary layer system in {G}evrey
  function space without structural assumption.
\newblock {\em SIAM J. Math. Anal.}, 53(3):3236--3264, 2021.

\bibitem{MR4700999}
W.-X. Li and T.~Yang.
\newblock 3{D} hyperbolic {N}avier-{S}tokes equations in a thin strip: global
  well-posedness and hydrostatic limit in {G}evrey space.
\newblock {\em Commun. Math. Anal. Appl.}, 1(4):471--502, 2022.

\bibitem{MR4726840}
W.-X. Li, T.~Yang, and P.~Zhang.
\newblock Gevrey well-posedness of quasi-linear hyperbolic {P}randtl equations.
\newblock {\em Commun. Math. Anal. Appl.}, 2(4):388--420, 2023.

\bibitem{MR3975147}
C.-J. Liu, F.~Xie, and T.~Yang.
\newblock Justification of {P}randtl ansatz for {MHD} boundary layer.
\newblock {\em SIAM J. Math. Anal.}, 51(3):2748--2791, 2019.

\bibitem{MR3882222}
C.-J. Liu, F.~Xie, and T.~Yang.
\newblock M{HD} boundary layers theory in {S}obolev spaces without monotonicity
  {I}: {W}ell-posedness theory.
\newblock {\em Comm. Pure Appl. Math.}, 72(1):63--121, 2019.

\bibitem{MR2049030}
M.~C. Lombardo, M.~Cannone, and M.~Sammartino.
\newblock Well-posedness of the boundary layer equations.
\newblock {\em SIAM J. Math. Anal.}, 35(4):987--1004, 2003.

\bibitem{MR3385340}
N.~Masmoudi and T.~K. Wong.
\newblock Local-in-time existence and uniqueness of solutions to the {P}randtl
  equations by energy methods.
\newblock {\em Comm. Pure Appl. Math.}, 68(10):1683--1741, 2015.

\bibitem{MR4151564}
T.~Matsui, R.~Nakasato, and T.~Ogawa.
\newblock Singular limit for the magnetohydrodynamics of the damped wave type
  in the critical {F}ourier-{S}obolev space.
\newblock {\em J. Differential Equations}, 271:414--446, 2021.

\bibitem{MR1697762}
O.~A. Oleinik and V.~N. Samokhin.
\newblock {\em Mathematical models in boundary layer theory}, volume~15 of {\em
  Applied Mathematics and Mathematical Computation}.
\newblock Chapman \& Hall/CRC, Boca Raton, FL, 1999.

\bibitem{MR4125518}
M.~Paicu, P.~Zhang, and Z.~Zhang.
\newblock On the hydrostatic approximation of the {N}avier-{S}tokes equations
  in a thin strip.
\newblock {\em Adv. Math.}, 372:107293, 42, 2020.

\bibitem{MR2563627}
M.~Renardy.
\newblock Ill-posedness of the hydrostatic {E}uler and {N}avier-{S}tokes
  equations.
\newblock {\em Arch. Ration. Mech. Anal.}, 194(3):877--886, 2009.

\bibitem{MR1617542}
M.~Sammartino and R.~E. Caflisch.
\newblock Zero viscosity limit for analytic solutions, of the {N}avier-{S}tokes
  equation on a half-space. {I}. {E}xistence for {E}uler and {P}randtl
  equations.
\newblock {\em Comm. Math. Phys.}, 192(2):433--461, 1998.

\bibitem{MR4384041}
C.~Wang and Y.~Wang.
\newblock On the hydrostatic approximation of the {MHD} equations in a thin
  strip.
\newblock {\em SIAM J. Math. Anal.}, 54(1):1241--1269, 2022.

\bibitem{MR4803680}
C.~Wang and Y.~Wang.
\newblock Optimal {G}evrey stability of hydrostatic approximation for the
  {N}avier-{S}tokes equations in a thin domain.
\newblock {\em J. Inst. Math. Jussieu}, 23(4):1521--1566, 2024.

\bibitem{MR2020656}
Z.~Xin and L.~Zhang.
\newblock On the global existence of solutions to the {P}randtl's system.
\newblock {\em Adv. Math.}, 181(1):88--133, 2004.

\bibitem{Xu2026}
R.~Xu and Z.~Xu.
\newblock Gevrey well-posedness of the {MHD}-wave boundary layer equations
  without any structure assumption.
\newblock {\em Sci China Math}, 2026. DOI: 10.1007/s11425-025-2462-7.

\bibitem{MR3464051}
P.~Zhang and Z.~Zhang.
\newblock Long time well-posedness of {P}randtl system with small and analytic
  initial data.
\newblock {\em J. Funct. Anal.}, 270(7):2591--2615, 2016.

\end{thebibliography}

\end{document}